\documentclass[12pt]{amsart}

\usepackage{hyperref}
\usepackage{geometry}                
\usepackage{graphicx}
\usepackage{amssymb}

\usepackage{float,subfigure,booktabs}

\theoremstyle{plain} 
\newtheorem{theorem}{Theorem}[section]
\newtheorem{lemma}[theorem]{Lemma}
\newtheorem{corollary}[theorem]{Corollary}
\newtheorem{proposition}[theorem]{Proposition}
\theoremstyle{definition} 
\newtheorem{definition}[theorem]{Definition}

\newcommand{\Z}{\mathbb{Z}}
\newcommand{\bR}{\mathbb{R}}
\newcommand{\bS}{\mathbb{S}}
\newcommand{\bH}{\mathbb{H}}

\DeclareMathOperator{\arcosh}{arcosh}

\begin{document}

\begin{titlepage}

\begin{title}
{Periodic Polyhedra in Spaces of Constant Curvature}
\end{title}

\author
{Christina Duffield}
\address{Christina Duffield\\Department of Mathematics\\Indiana University\\
Bloomington, IN 47405
\\USA}

\author
{Daniel Freese}
\address{Daniel Freese\\Department of Mathematics\\Indiana University\\
Bloomington, IN 47405
\\USA}

\author
{William Holt}
\address{William Holt\\Department of Mathematics\\Indiana University\\
Bloomington, IN 47405
\\USA}

\author{
Matthias Weber
}
\address{Matthias Weber\\Department of Mathematics\\Indiana University\\
Bloomington, IN 47405
\\USA}

\author
{Ramazan Yol}
\address{Ramazan Yol\\Department of Mathematics\\Indiana University\\
Bloomington, IN 47405
\\USA}

\date{\today}

\begin{abstract}
We show the existence of families of  periodic polyhedra in spaces of constant curvature whose fundamental domains can be obtained by attaching prisms and antiprisms to Archimedean solids. These polyhedra  have constant discrete curvature and are weakly regular in the sense that all faces are congruent regular polygons and all vertex figures are congruent as well. Some of our examples have stronger conformal or metric regularity. The polyhedra are invariant under either a group generated by reflections at the faces of a Platonic solid, or a group generated by transformations that are reflections at the faces of a Platonic solid, followed by a rotation about an axis perpendicular to the respective face. In particular, suitable quotients will be compact polyhedral surfaces in (possibly non-compact) spaceforms. 
\end{abstract}

\maketitle

\end{titlepage}
\section{Introduction}

Imagine a collection of Platonic solids in space, with regular prisms (or antiprisms) connecting the faces in pairs. In the case of prisms (antiprisms), this will be an infinite polyhedral surface with all square (triangle) faces and vertices that have the same valency, given them a weak form of regularity. 

\begin{definition}
A  polyhedral surfaces that consist entirely of prisms (resp. antiprisms) will be called {\em prismatic} (resp. {\em antiprismatic}).
\end{definition}

A simple example of  a prismatic polyhedron in Euclidean space is the {\em mucube} (figure \ref{fig:mucube}), where cubes are connected by square prisms. Are there others? 

An inductive construction might begin with a Platonic solid, attach prisms (or antiprisms) to all its faces, and then Platonic solids to the open ends of the prisms (or antiprisms), etc. We have no reason to expect that this construction will ``close up''. In fact, in general it doesn't. But when we allow the polyhedra to reside in Hyperbolic or Spherical space, it turns out that this construction will close up, for certain choices of the size of the polyhedra. 
This is somewhat surprising. The closing-up condition is reminiscent of the period condition for minimal surfaces. Here we have many  periods to solve but only one free parameter, the size (or dihedral angle) of the Platonic solid we begin with. 

That this construction works for both prisms and antiprisms for many sizes of the polyhedra is due to symmetry, which we will explain next.

In the case of  the hyperbolic mucube, we will start with a Platonic cube that tiles hyperbolic space, i.e. whose dihedral angle is of the form $2\pi/n$ for some integer $n$. Inside of it we will place a smaller Platonic cube so that the regular prisms erected on top of its faces are symmetric with respect to a reflection at the faces of the larger cube. This argument will require justification,  taking into account that the larger cube might be {\em hyperideal}, i.e. whose vertices do not exist in hyperbolic space.

In the antiprismatic case, the construction hinges on finding a fundamental domain for discrete groups generated by reflections at a Platonic cube that are followed by a $45^\circ$ rotation. These fundamental domains can be obtained from the cube by attaching pyramids to the triangular faces of an Archimedean truncated cube. The pyramids will need to have  dihedral angle of the form $2\pi/n$, and a  matching condition will impose an additional angle condition on the dihedral angles of the truncated cube and the pyramids.  For $n=3,4,5$, the pyramids will fit together to form a hyperbolic Platonic tetrahedron, octahedron, and icosahedron.

%
%

In this paper, we will focus on adding prisms and antiprisms to Platonic solids in order to obtain periodic polyhedra with a high degree of regularity: all faces and vertex figures will be congruent. There are some straightforward variations that instead use  certain Archimedean solids obtained as truncations and rectifications of Platonic solids, which we will also discuss, see sections \ref{sec:ppricub}, \ref{sec:pricubsp}, \ref{sec:apricuboc}. 

Our  prismatic surfaces will be invariant under a discrete group of isometries that is generated by the reflections at the faces of a suitable Platonic solid (in the space under consideration) where the dihedral angle between the faces is an integral fraction of $2\pi$.

For instance, we can use the $72^\circ$ cube in hyperbolic space, place a smaller cube inside, so that the faces of the $72^\circ$ cube divide the attached regular prisms in half, see figure \ref{fig:hypmucube}.

\begin{figure}[h] 
   \centering
   \includegraphics[width=2.5in]{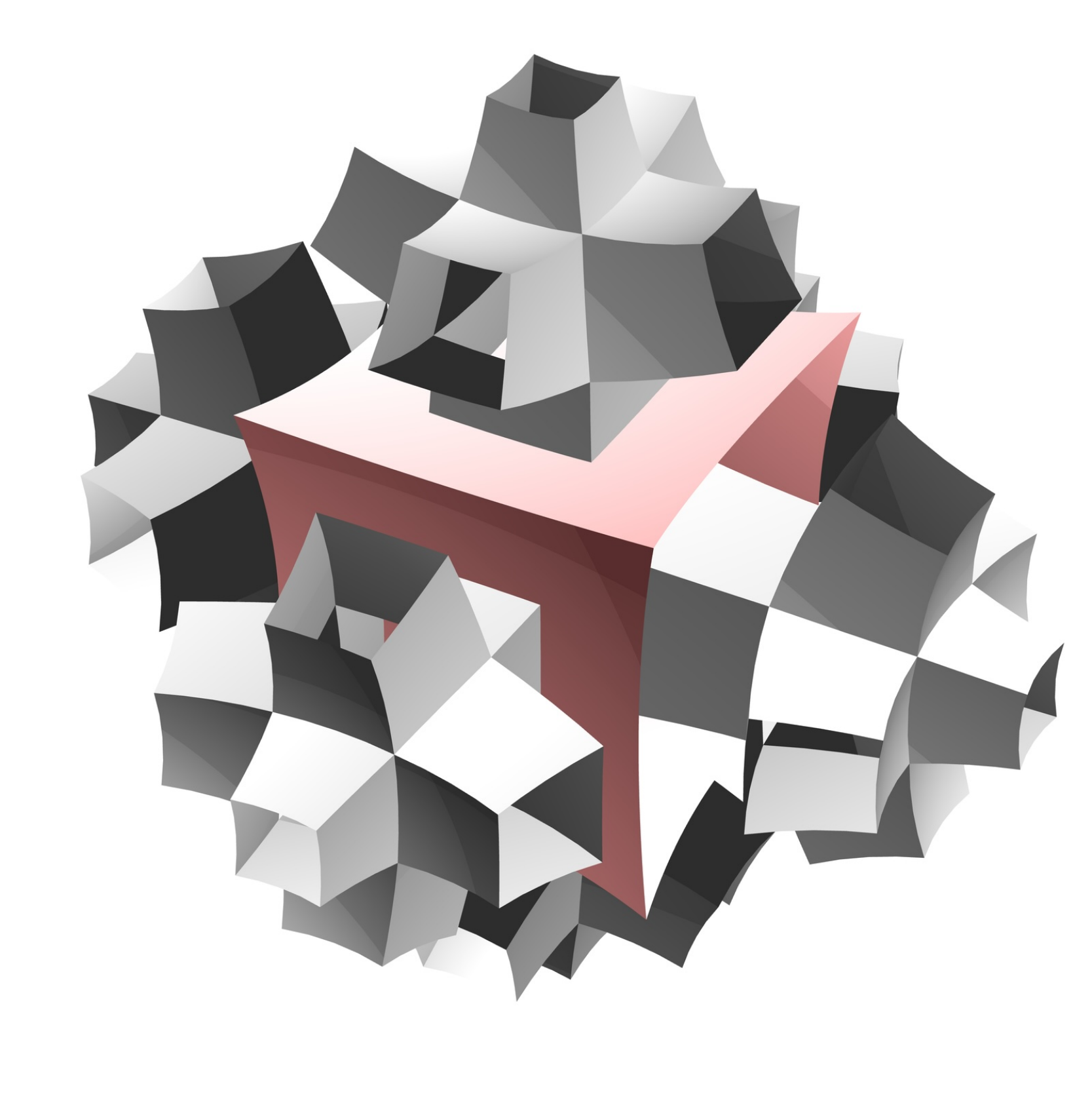}\qquad 
   \includegraphics[width=2.5in]{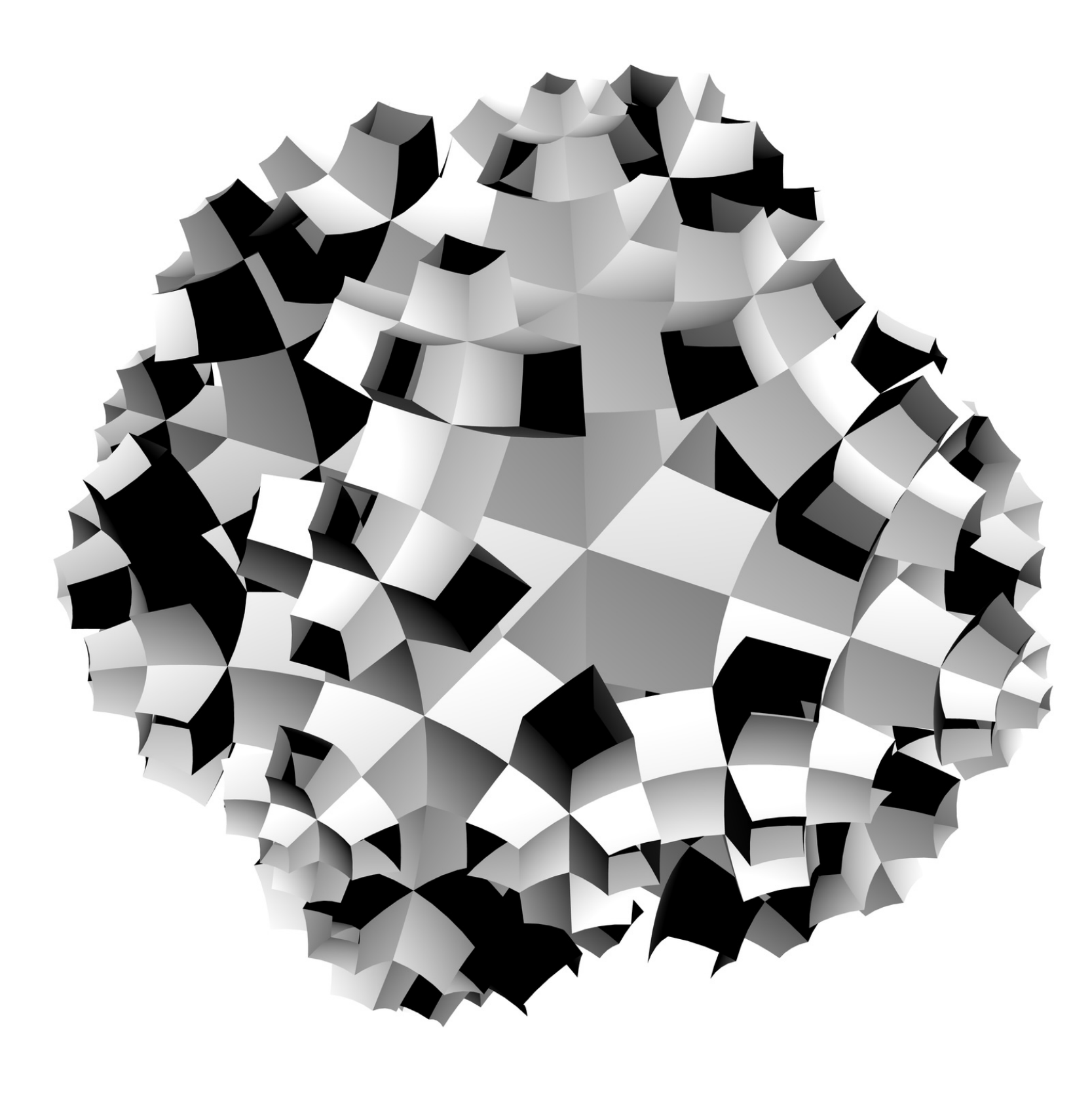} 
   \caption{Generation 2 and 3 of the hyperbolic prismatic cube ($n=5$)}
   \label{fig:hypmucube}
\end{figure}

While in Euclidean space, all cubes differ only by scale and have the same dihedral angle, in spherical and hyperbolic geometry, the dihedral angle between the faces of a cube varies with the size of the cube. We will see that in hyperbolic space there is a discrete family of polyhedral surfaces of each type.

The resulting polyhedral surfaces will  be regular in a weaker sense: Isometries of the ambient space will act transitively on vertex figures, and all faces will be regular polygons.

\section{Notations and Definitions}

\subsection{Platonic Solids, Truncations, and Rectifications}

Let us denote a Platonic solid with $p$-gonal faces and of vertex valency $q$ by $P_{p,q}$. Thus $P_{3,3}$ is a tetrahedron, $P_{4,3}$ a cube, $P_{3,4}$ an octahedron, $P_{5,3}$ a dodecahedron and $P_{3,5}$ an icosahedron. 

We will also consider truncations and rectifications. Denote by $TP_{p,q}$ the truncation of $P_{p,q}$ where we cut off a pyramid for each vertex whose base meets the edges of $P_{p,q}$ so that all new faces become regular polygons. The truncation will have both regular $2p$-gons and $q$-gons as faces.

Similarly, denote by $RP_{p,q}$ the rectification of $P_{p,q}$ where we cut off a pyramid for each vertex whose base meets the edges of $P_{p,q}$ at their midpoints. The rectification will have both regular $p$-gons and $q$-gons as faces.

\begin{table}[h]
   \centering
   \begin{tabular}{@{} llll @{}} 
      \toprule
      \multicolumn{4}{c}{Platonic polyhedra and relatives} \\
      \midrule 
       index   & polyhedron & truncation& rectification\\
      \midrule
      $\{3,3\} $     & tetrahedron & truncated tetrahedron & octahedron \\
      $ \{4,3\}$         & cube     &  truncated cube & cuboctahedron \\
      $ \{3,4\}$       & octahedron  & truncated octahedron & cuboctahedron \\
      $\{5,3\} $       & dodecahedron  & truncated dodecahedron &icosidodecahedron\\
      $\{3,5\} $       & icosahedron  & truncated icosahedron & icosidodecahedron\\
      \bottomrule
   \end{tabular}
   \label{tab:booktabs}
\end{table}

All these polyhedra have realizations in $\bS^3$, $\bR^3$ and $\bH^3$ which we will discuss in detail in later sections. 

We will also need a modification $KP_{p,q}^n$ of the truncated polyhedron $TP_{p,q}$ that is related to Conway's {\em kis} operation (\cite{brink}): We attach to each truncated face (a $q$-gon) of a $TP_{p,q}$ a pyramid $PY_{q}^n$ whose sides make dihedral angles of $2\pi/n$ for $n\ge2$. For $n=2$ these pyramids become just planar polygons, and $KP_{p,q}^2=TP_{p,q}$. 

In Euclidean space, examples of pyramids $PY_{q}^n$ can be obtained by centrally subdividing a Platonic solid: $P_{p,q}$ is the union of as many  pyramids $PY_p^q$ as $P_{p,q}$ has faces. The existence of general and typically non-compact pyramids $PY_{q}^n$ will be established in section \ref{sec:pyramids}. These infinite pyramids exist only in Euclidean and hyperbolic space and can be used to construct infinite   Platonic solids $P_{p,q}$, e.g. a polyhedron $P_{7,3}$ with  regular heptagons as faces and valency 3.

An  example of a modification $KP_{p,q}^n$ in Euclidean space that is relevant to us is the {\em Triakis Truncated Tetrahedron}  $KP_{3,3}^3$. It tiles space when using $60^\circ$ rotations to glue copies of it together along the hexagonal faces, see section \ref{sec:qch}.

\subsection{Notions of Regularity}

We will consider  finite and infinite polyhedral surfaces in the three simply connected spaces of constant curvature that enjoy certain regularity properties. The strongest notion of regularity we consider is that of Platonicity:

\begin{definition}
A polyhedron $P$ is called {\em geometrically Platonic} if the group of isometries of the ambient space that preserve $P$ acts transitively on the set of flags $(v\in e \in f)$ of $P$.
\end{definition}

A much weaker notion can be obtained by forgetting the ambient space and merely considering the combinatorial structure: The polyhedron then becomes a map on a surface.

\begin{definition}
A map $P$ is called {\em combinatorially Platonic} if the group of map preserving automorphisms  of  $P$ acts transitively on the set of flags $(v\in e \in f)$ of $P$.
\end{definition}

In between these two notions is what we will call conformal Platonicity: Here we forget ambient space but keep the map as well as the conformal structure of the polyhedron:

\begin{definition}
A polyhedron $P$ is called {\em conformally Platonic} if the group of conformal and map preserving automorphisms  of  $P$ acts transitively on the set of flags $(v\in e \in f)$ of $P$.
\end{definition}

One of the purposes of this paper is to exhibit examples of infinite conformally Platonic polyhedra that are not geometrically Platonic, i.e. that possess hidden conformal symmetries.

All polyhedral surfaces we construct will have in common that all vertex figures are congruent: That is, for any two vertices there is an isometry of ambient space that maps one vertex to the other and all edges incident with  the first vertex to the edges incident with the second vertex. This implies that our polyhedra have constant (discrete) curvature in any sense that assigns a local curvature quantity to the vertices of the polyhedron, see for instance \cite{sull} for  definitions of discrete  curvatures. 

This observation  makes our polyhedra candidates of discrete approximations of smooth constant mean curvature surfaces. In fact, the mucube (or prismatic cube) and the antiprismatic octahedron are conformally correct discretizations of the Hermann Schwarz's  P surface and Alan Schoen's I-WP surface (\cite{lwy}). Polthier (\cite{polth}) has constructed smooth minimal surfaces in hyperbolic space with the same symmetries and topology as our prismatic Platonic polyhedra.

\subsection{Geometric Formulas in Hyperbolic and Spherical Space}
\label{sec:formulas}

We will (mainly) use the ball model for hyperbolic space $\bH^3$.

\begin{definition}
For any point $x$ with $|x|>1$, denote by $S(x)$ the sphere centered at $x$ with radius $\sqrt{|x|^2-1}$. The intersection of this sphere with the unit ball represents a geodesic (hyperbolic) plane. Denote by $H(x)$ the unbounded component of the complement of $S(x)$ in $\bR^3$; its portion in the unit ball represents a geodesic half space.
\end{definition}

We will use intersections of these half spaces to construct Platonic solids.

To compute dihedral angles we will frequently use the formula for the angle of intersection of two general spheres:

\begin{lemma} 
The spheres centered at $x_i$ with radius $r_i$ for $i=1,2$ intersect at an angle $\phi$ with 
\[
\cos \phi = \frac{|x_1 - x_2|^2 - r_1^2 - r_2^2}{2 r_1 r_2} \ .
\]
\end{lemma}

We also recall the formula for the hyperbolic metric in the ball model:

\begin{lemma}
The hyperbolic distance  between two points $p,q$ with $|p|,|q|<1$ is given by
\[
\cosh d_h(p,q) = 1+2 \frac{|p-q|^2}{(1-|p|^2)(1-|q|^2)} \ .
\]
\end{lemma}

As we will only need to compare distances, all equations involving distances will be algebraic. In fact, all our equations have algebraic solutions.

For computations in the sphere $\bS^3$, we work in $\bR^3$ using the stereographic projection

\[
\sigma(x_1,x_2,x_3,x_4) = \frac1{1-x_4}(x_1,x_2,x_3) 
\]
and its inverse. This allows us to use the same coordinates for the vertices of polyhedra as we are used to in $\bR^3$. Reflections in $\bS^3\subset \bR^4$ are just reflections at hyperplanes in $\bR^4$. To find a reflection at a face of a spherical polyhedron stereographically projected into $\bR^3$, we take three of its vertices, find their preimages in $\bS^3$ under the stereographic projection and the normal vector to the hyperplane in $\bR^4$ through these points. Conjugation of the reflection at this hyperplane with the stereographic projection becomes a sphere inversion in $\bR^3$ at the sphere that contains the face of the polyhedron.

General rotations about great circles in $\bS^3$ are harder to describe. Fortunately, we  will only need very special rotations: The Platonic solid we begin with will be centered at the origin in $\bR^3$, and the groups we are interested in will be generated by reflections at a face followed by a rotation about the line perpendicular to that face and through the origin.
 
The $90^\circ$ rotation about a line through the origin in $\bR^3$ in direction of a unit vector $v=(a,b,c)$ is given by the matrix

\[
R^{90}_v =\left(
\begin{array}{ccc}
 0 & -c & b \\
 c & 0 & -a \\
 -b & a & 0 \\
\end{array}
\right) \ ,
\]
and general rotations by an angle $\phi$ are then given by

\[
R^{\phi}_v =\sin(\phi)R^{90}_v + \cos(\phi) (Id-P_v)+P_v
 \]
 
 where 
 \[
 P_v = v\cdot v^\top = 
 \left(
\begin{array}{ccc}
 a^2 & a b & a c \\
 a b & b^2 & b c \\
 a c & b c & c^2 \\
\end{array}
\right) \ .
\]

\subsection{Platonic Reflection Groups}

The following is well known:

\begin{lemma} For every $n\ge 3$, there exists a possibly infinite Platonic solid $P_{p,q}^n$ of each type in spherical, Euclidean, or hyperbolic space with dihedral angle $2\pi/n$. Table \ref{tab:plato} lists which value of $n$ occurs in which space form.
\end{lemma}

\begin{table}[h]
   \centering
   \begin{tabular}{@{} lllllll @{}} 
      \toprule
      \multicolumn{7}{c}{Values of $n$ for Platonic polyhedra with dihedral angle $2\pi/n$ sorted by space} \\
      \midrule 
       symbol   & name  &Spherical  &Euclidean& & Hyperbolic & \\
          &  	  &	  &	& finite & ideal & hyperideal	\\
      \midrule
      $\{3,3\} $     & tetrahedron & $3,4,5$ & -- & -- & $6$ & $\ge7$ \\
      $ \{4,3\}$      & cube     &  3 & 4 & 5 & 6 & $\ge7$ \\
      $ \{3,4\}$       & octahedron  & 3 & -- & -- & 4 & $\ge 5$ \\
      $\{5,3\} $       & dodecahedron  & 3  &-- & $4,5$ & 6 &$\ge 7$\\
      $\{3,5\} $       & icosahedron  & -- & -- & $3$ & -- & $\ge 4$ \\
      \bottomrule
   \end{tabular}
   \label{tab:plato}
\end{table}

\begin{proof}
To see why there are Platonic solids with arbitrarily small dihedral angle in hyperbolic space, we use the ball model. Place spheres $S(a x)$ at the vertices $x$ of the dual Platonic solid which has been normalized so that $|x|=1$. For $a\to 1+$ these spheres are certainly disjoint, and for $a$ large enough they certainly intersect. Thus there must be a value of $a$ so that  the spheres $S(a x)$ for neighboring vertices $x$ touch. The intersection of the corresponding half spaces is the extreme case of a Platonic solid  whose dihedral angle has become $0^\circ$. By increasing $a$ and by continuity, we can obtain hyperideal polyhedra with arbitrarily small positive dihedral angles in hyperbolic space.
\end{proof}

By the Poincaré polyhedron theorem, these polyhedra $P_{p,q}^n$ are fundamental domains for a discrete group of isometries $\Gamma_{p,q}^n$ that is generated by the reflections at its faces. The prismatic polyhedral surfaces we construct will be invariant under $\Gamma_{p,q}^n$. Note that while $P_{p,q}^n$ will not be compact for large $n$, the quotient of our polyhedral surfaces $\Pi_{p,q}^n$ by  $\Gamma_{p,q}^n$ will be compact.

\section{The Prismatic Cubes}

\subsection{The Euclidean Mucube}
We begin by reviewing a classical example.

The {\em mucube} is an infinite regular polyhedron in Euclidean space discovered by Coxeter and Petrie (\cite{cox38}).
It can be obtained by an inductive procedure that uses cubes and regular prims over squares (which are just cubes with top and bottom removed), see figure \ref{fig:mucube}.

\begin{figure}[h] 
   \centering
   \includegraphics[width=2in]{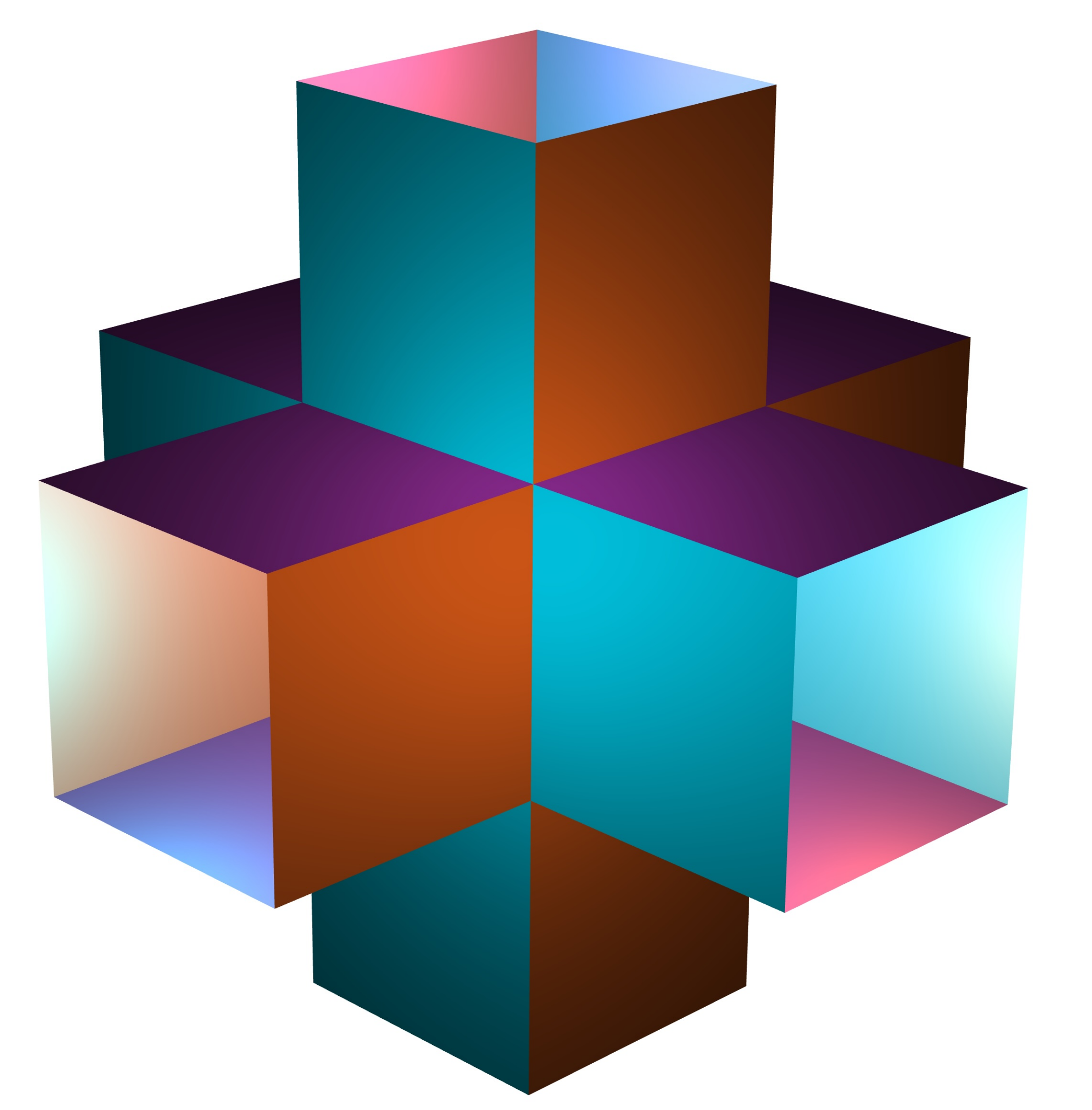}\qquad 
   \includegraphics[width=2in]{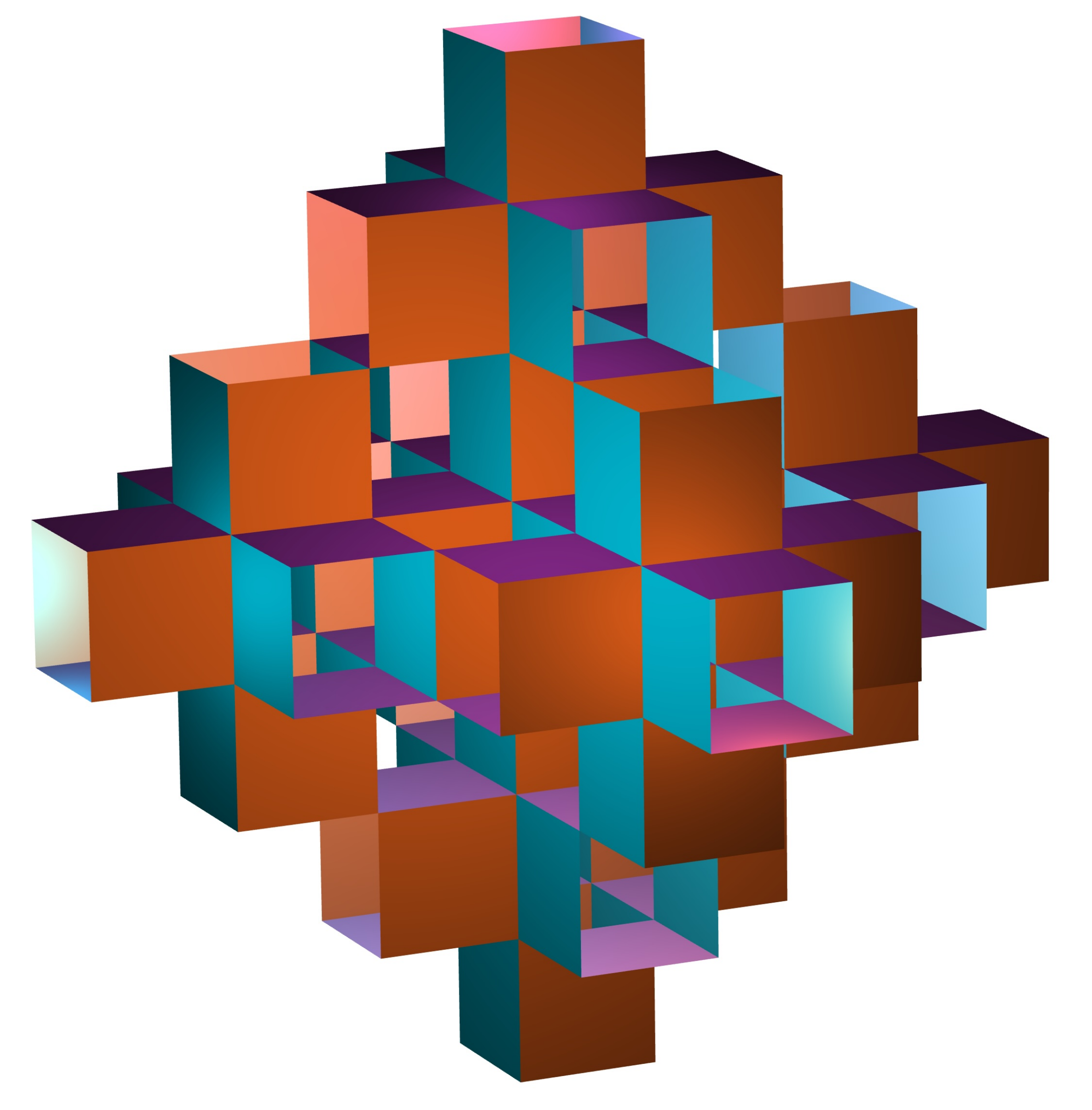} 
   \caption{The first two mucube generations}
   \label{fig:mucube}
\end{figure}

We begin with a cube of fixed size. Attach six regular prisms along their bottoms to each  face of the cube. 

Then attach cubes to the  tops of the prisms.  We repeat this construction, adding regular prisms to the unattached faces of the cubes,  and cubes to the tops of the newly attached prisms, and so on.

When done, we discard all cubes and just retain the squares of the prism sides.
They form an infinite Platonic polyhedron $\Pi_{4,3}^4$, called the {\em mucube}.

The translational symmetries of this surface form a cubical lattice with fundamental cell twice as large as the cubes used for the construction. The quotient of the mucube by its translational symmetries is a genus 3 polyhedral surface tiled by 12 squares so that 6 squares meet at each of its 8 vertices. A more abstract construction of the same surface is that of a double cover of a cube, branched at its 8 vertices. 

Thus we have a geometric realization of an algebraic curve as a triply periodic polyhedral surface.

\subsection{The  Prismatic Cube in $\bH^3$}
\label{sec:hypricu}

We begin by constructing (generalized) cubes in hyperbolic space.

\begin{definition}
For a real number $a>1$, denote by
$P_{4,3}(a)$ the intersection of the  sets $H(a x)$ for all $x\in \{(\pm1,0,0), (0,\pm1,0), (0,0,\pm1)$ with the unit ball. We call $P_{4,3}(a)$ a (generalized) {\em hyperbolic cube}. 
\end{definition}

We distinguish four types of cubes $P_{4,3}(a)$:
\begin{itemize}
\item finite cubes when $S(a,0,0)$,  $S(0,a,0)$ and $S(0,0,a)$  have an intersection that completely lies inside the open unit ball;
\item the ideal cube when $S(a,0,0)$,  $S(0,a,0)$ and $S(0,0,a)$  have an intersection on the unit sphere
\item hyperideal cubes when the spheres like $S(a,0,0)$ and $S(0,a,0)$ only pairwise intersect 
\item edgeless cubes when $S(a,0,0)$ and $S(0,a,0)$ are disjoint.
\end{itemize}

For the construction of prismatic cubes in hyperbolic space, we will use finite cubes to attach prisms to, and allow finite, ideal and hyperideal cubes with dihedral angle $2\pi/n$ as fundamental domains for the reflection groups $\Gamma_{4,3}^n$. This ensures that the fundamental domains have edges and we can determine their dihedral angles. In particular, edgeless cubes will be irrelevant for us.

The following lemma quantifies for which values of $a$ we obtain what type of cube:

\begin{lemma}
\begin{enumerate}
\item For  $a>\sqrt3$,  $S(a,0,0)$,  $S(0,a,0)$ and $S(0,0,a)$ have an intersection $p$ with $|p|<1$. In this case $P_{4,3}(a)$ is a finite cube. For $a=\sqrt3$, the cube is ideal.
\item For $a=\sqrt2$, $S(a,0,0)$ and  $S(0,a,0)$ touch.
\item For $a<\sqrt{2}$, $S(a,0,0)$ and $S(0,a,0)$ are disjoint, so $P_{4,3}(a)$ is edgeless.
\item For $a>\sqrt2$, $S(a,0,0)$ and $S(0,a,0)$ meet at a dihedral angle  $\alpha$ with $\cos \alpha = \frac1{a^2-1}$.
\item For   $\sqrt3>a>\sqrt2$, $P_{4,3}(a)$ is hyperideal.
\end{enumerate}
\end{lemma}

\begin{proof}
The two intersections of  $S(a,0,0)$,  $S(0,a,0)$ and $S(0,0,a)$ are given by
\[
z_{1,2} = \frac{1}{3} \left(a\pm\sqrt{a^2-3}\right)(1,1,1)\ .
\]
So the intersection is non-empty if $a\ge \sqrt3$. 

The angle $\alpha$ between the spheres $S(a,0,0)$ and $S(0,a,0)$ satisfies $\cos \alpha = \frac1{a^2-1}$. Thus the spheres touch if $a=\sqrt2$. For smaller $a$ they are disjoint. For larger $a$ they make an angle between $0$ and $\pi/2$. 
\end{proof}

As a consequence, we can determine the (well-known)  cubes that tile hyperbolic space:

For $a=\sqrt{2+\sqrt5}$ we have $\alpha=72^\circ$ and $n=5$. This is the only finite cube that tiles hyperbolic space.

For $a=\sqrt{3}$  the dihedral angle becomes $\alpha=60^\circ$ and $n=6$. This cube is ideal. To see this without computation, we first show that the dihedral angle of an ideal cube is $60^\circ$. To this end we switch to the upper half space model and place one vertex of the ideal cube at $(0,0,\infty)$. The edges resp. faces incident with that vertex become vertical lines resp. vertical halfplanes. By symmetry, the halfplanes meet at a $60^\circ$ angle. Thus, back in the ball model, the cube becomes ideal for the value of $a$ for which $\alpha = 60^\circ$.

For any $n>6$, the cubes with dihedral angle $2\pi/n$ are hyperideal.

In particular the dihedral angle of these cubes ranges in the interval $(0,\pi/2)$,
and  for any integer $n>4$ there is a (unique) $a=a_n$ so that  $P_{4,3}(a_n)$ has dihedral angle $2\pi/n$.

\begin{corollary}
For any $n>4$, $P_{4,3}(a_n)$ is a fundamental domain for the group generated by the reflections at the faces of $P_{4,3}(a_n)$.
\end{corollary}
\begin{proof}
This follows from the Poincaré polyhedron theorem.
\end{proof}

Fix $n>4$, and denote $a=a_n$. Consider a finite cube $P_{4,3}(b)$ with $b>\sqrt3$ and $b>a$ so that $P_{4,3}(b)$ is finite and lies inside $P_{4,3}(a)$. Take any of the faces of $P_{4,3}(b)$ and reflect it (using sphere inversion) at the corresponding face of $P_{4,3}(a)$ (which is a hyperbolic isometry).  The original and the reflected squares are the bottom and top faces of a  prism over a square with rectangular sides.

\begin{proposition}
For a suitable (and unique) value of $b$, the rectangular sides of this prism become squares.
\end{proposition}

\begin{proof}
For $b$ very large, $P_{4,3}(b)$ becomes arbitrarily small, and the rectangular sides of the prisms tall.

When we make $b$ smaller, two things can happen: If $P_{4,3}(a)$ is finite, then $b$ will approach $a$. If $P_{4,3}(a)$ is ideal or hyperideal, $b$ will approach $\sqrt3$ and become ideal.

In the first case (for $b\to a$), the height of the rectangles goes to zero while the length of the base  approaches that of the edge length of $P_{4,3}(a)$.

In the second case the edge length of the base square goes to infinity as $P_{4,3}(b)$ approaches the ideal cube, while the height of the rectangles remains bounded by the distance to a face of $P_{4,3}(a)$.

By the intermediate value theorem, there is a  value of $b$ where the sides of the prism are squares. We note that with increasing $b$ the base of the rectangles decreases in length while the height increases, so uniqueness follows from monotonicity.
\end{proof}

\begin{corollary}
For $n\ge 5$, there exists a prismatic cube $\Pi_{4,3}^n$ in hyperbolic space.
\end{corollary}

The above argument is  valid for all prismatic constructions that we will discuss below, with minor modifications. 

The case of the cube is simple enough to give the explicit data:

To determine the parameter $b$ for the inner cube $P_{4,3}(b)$, we first compute its edge length in terms of an arbitrary $b$. The vertices of $P_{4,3}(b)$  are given by $\frac{1}{3} \left(b-\sqrt{b^2-3}\right)(\pm1,\pm1,\pm1)$, and the hyperbolic distance between two adjacent vertices is $\arcosh(1+2\alpha)$ for $\alpha=\frac{1}{b^2-3}$.

The faces of this cube $P_{4,3}(b)$ are the bottoms of the prisms of our construction. The tops are obtained by reflecting the faces of $P_{4,3}(b)$ at the corresponding faces of $P_{4,3}(a)$. The hyperbolic distance of a vertex of $P_{4,3}(b)$ to its reflection at $P_{4,3}(a)$ is $\arcosh(1+2\beta)$ for $\beta=\frac{(b-a)^2}{\left(b^2-3\right)
   \left(b^2-1\right)}$.
   
 We want the prism to be regular, i.e. $\alpha=\beta$. The resulting quadratic equation has one solution with $b>1$, namely $b=a+\sqrt{a^2-1}$.

\subsection{The Prismatic Cube in $\bS^3$}

The 3-sphere $\bS^3$ can be tiled with eight spherical cubes that have a dihedral angle of $120^\circ$; these cubes correspond to the eight 3-dimensional faces of a 4-dimensional hypercube (in the same way the six square faces of a 3-dimensional cube correspond to the tiling of the 2-sphere by spherical $120^\circ$ squares).

These spherical cubes can be used to construct a prismatic cube $\Pi_{4,3}^3$  inside of $\bS^3$, see figure \ref{fig:sphmucube}.

\begin{figure}[h] 
   \centering
   \includegraphics[width=2in]{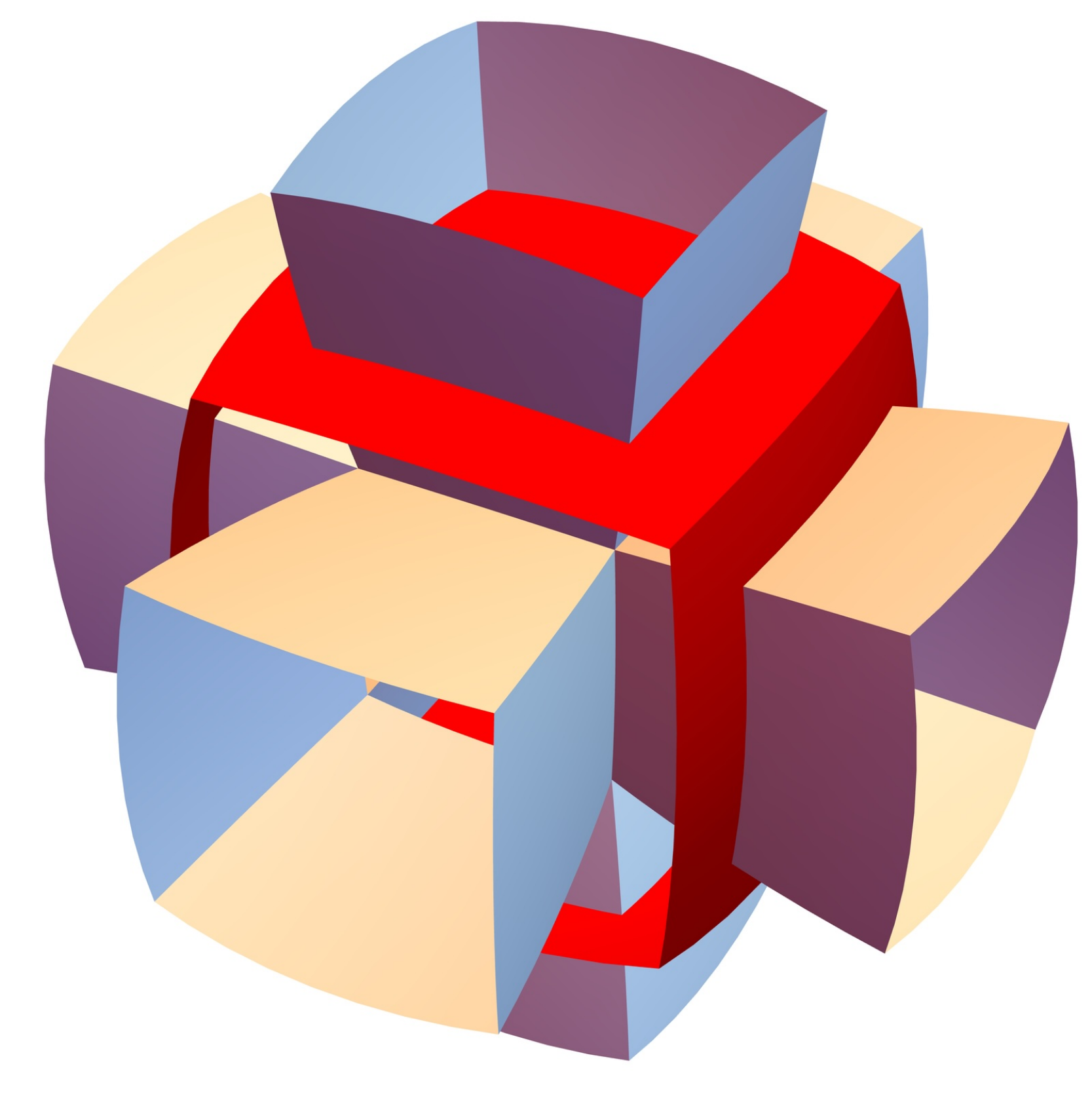}\qquad \includegraphics[width=2in]{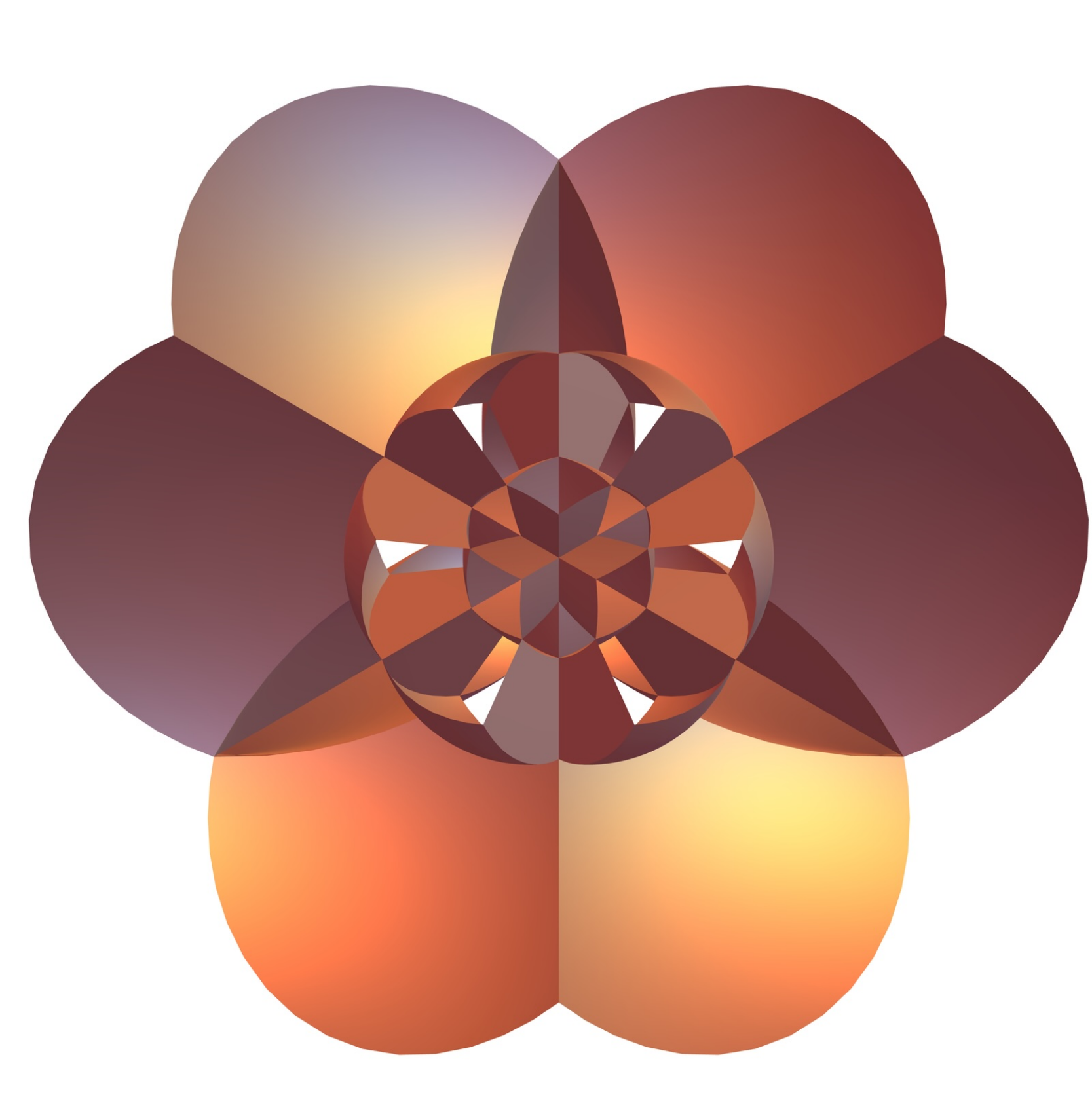} 
   \caption{Constructing the spherical prismatic cube for $n=3$}
   \label{fig:sphmucube}
\end{figure}

We visualize  polyhedral surfaces in $\bS^3$ by applying a stereographic projection which is normalized so that the equator $\bS^2\subset \bS^3$ is mapped to the unit sphere in $\bR^3$.

The following claims can be verified by easy computation:

Consider the 8 vertices with coordinates $a(\pm1, \pm1,\pm1)\in\bR^3$ with $a=1/\sqrt3$. These correspond in $\bS^3$ to the vertices of a spherical cube that has faces making a dihedral angle of $120^\circ$. 

The 8 vertices with coordinates $b(\pm1, \pm1,\pm1)\in\bR^3$ with 
$$b=\left(-1-\sqrt{2}+\sqrt{2 \left(3+\sqrt{2}\right)}\right)/{\sqrt{3}}$$

correspond in $\bS^3$ to the vertices of a smaller spherical cube. The regular prisms attached to its faces are symmetric with respect to reflection  at the  spherical faces of the larger cube. 

Extending further by reflection results in a closed polyhedral surface tiled by 96 squares with valency 6 at each vertex.


This is related to the {\em Poincaré cube space} (also called {\em quaternionic space}, see \cite{gord}) where opposite faces of the $120^\circ$ cube are identified by a translation followed by a $90^\circ$ rotation. This identification ensures that the cycles around each edge have length 3, so we obtain a spherical 3-manifold.

%

We mention in passing the case that $n=2$. Then the spherical cube will have dihedral angles of $180^\circ$, making it a great $\bS^2\subset \bS^3$. The prismatic cube $\Pi_{4,3}^2$ can be obtained by using two copies of  the six prisms over the faces of a specific cube, so the figure \ref{fig:smucube2} shows only half of the truth: At the vertices where three squares appear to meet, there are actually six of them, each face occurring twice.

The same construction works (trivially) for any Platonic solid in $\bS^3$, giving us degenerate examples for the case $n=2$, which we will not mention again.

\begin{figure}[h] 
   \centering
   \includegraphics[width=2in]{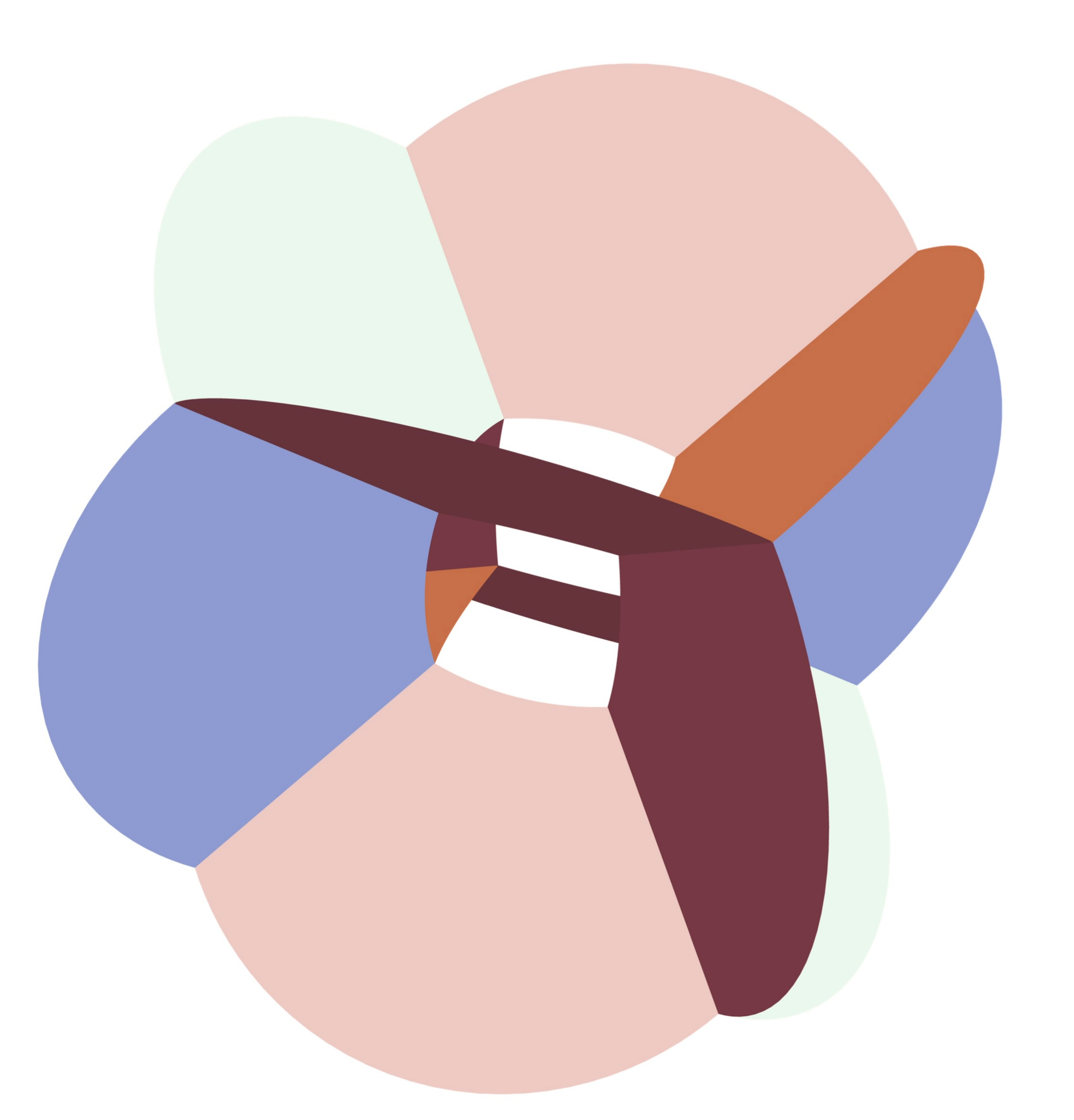} 
   \caption{The degenerate spherical prismatic cube for $n=2$}
   \label{fig:smucube2}
\end{figure}

\section{The Prismatic Tetrahedra}

The regular tetrahedron in Euclidean space has a dihedral angle of $\arccos(1/3)\approx 70.5^\circ$ and doesn't tile. Thus we will have hyperbolic tetrahedra $P_{3,3}^n$ with dihedral angle $2\pi/n$ for $n\ge 6$, and spherical tetrahedra for $n=3$, $n=4$, and $n=5$. We begin with the hyperbolic case:

\subsection{Hyperbolic Prismatic Tetrahedra}

Tetrahedra $P_{3,3}(a)$ can be constructed in the ball model of hyperbolic space as the intersection of the complements of spheres with centers at $a(1,1,1)$, $a(1,-1,-1)$, $a(-1,1,-1)$ and $a(-1,-1,1)$ that meet the unit sphere at a right angle. The dihedral angle of $P_{3,3}(a)$ is given by 
\[
\cos\alpha = \frac{a^2+3}{3(a^2-1)}
\]
so that a tiling hyperbolic tetrahedron $P_{3,3}(a_n)$ exists for all $\alpha = 2\pi/n$ for $n\ge 6$. 

For $n=6$, $P_{3,3}(a_6)$ is ideal, and for $n>6$ all $P_{3,3}(a_n)$ are hyperideal. 

\begin{figure}[h] 
   \centering
   \includegraphics[width=3in]{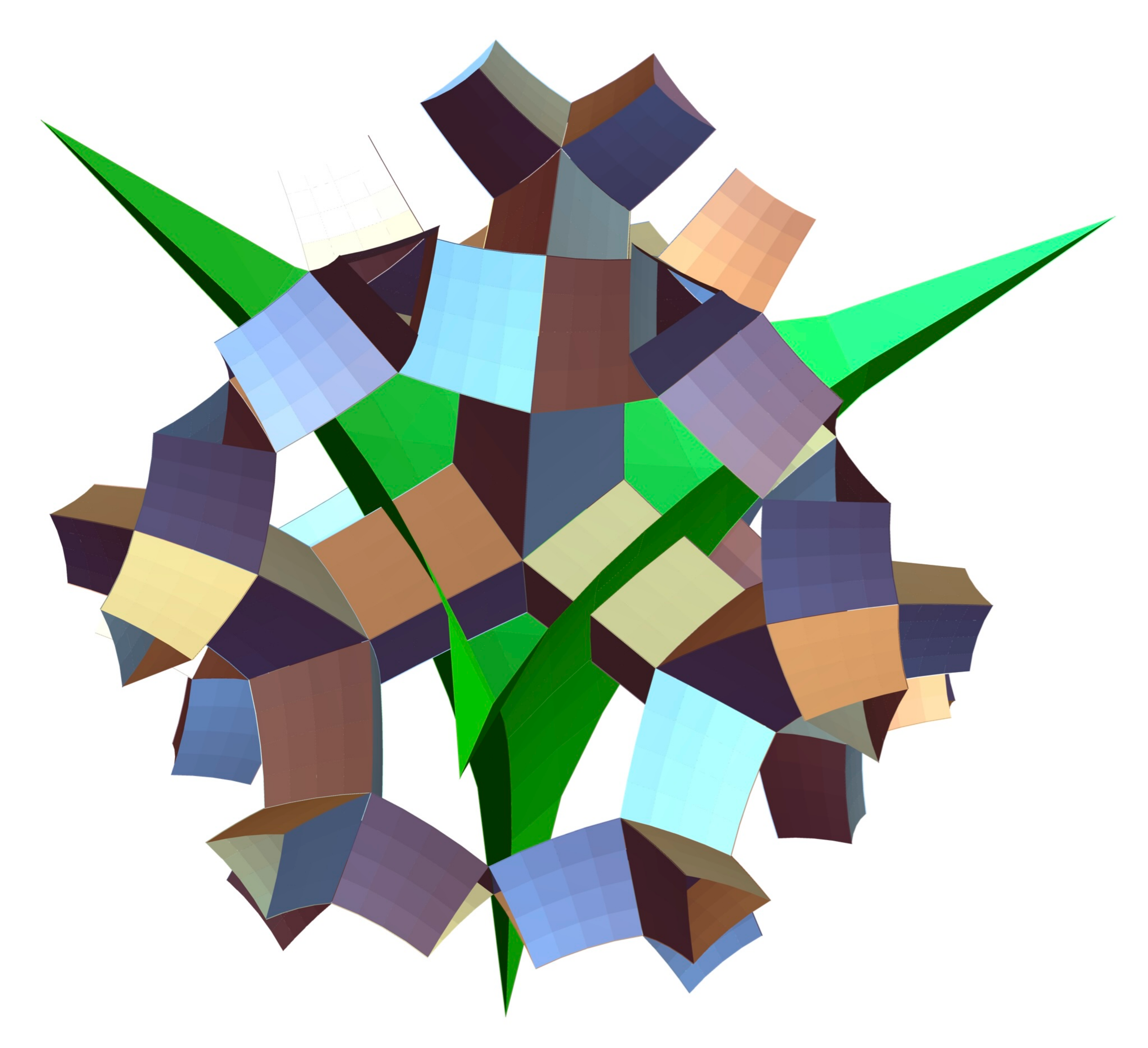}
   \qquad \includegraphics[width=3in]{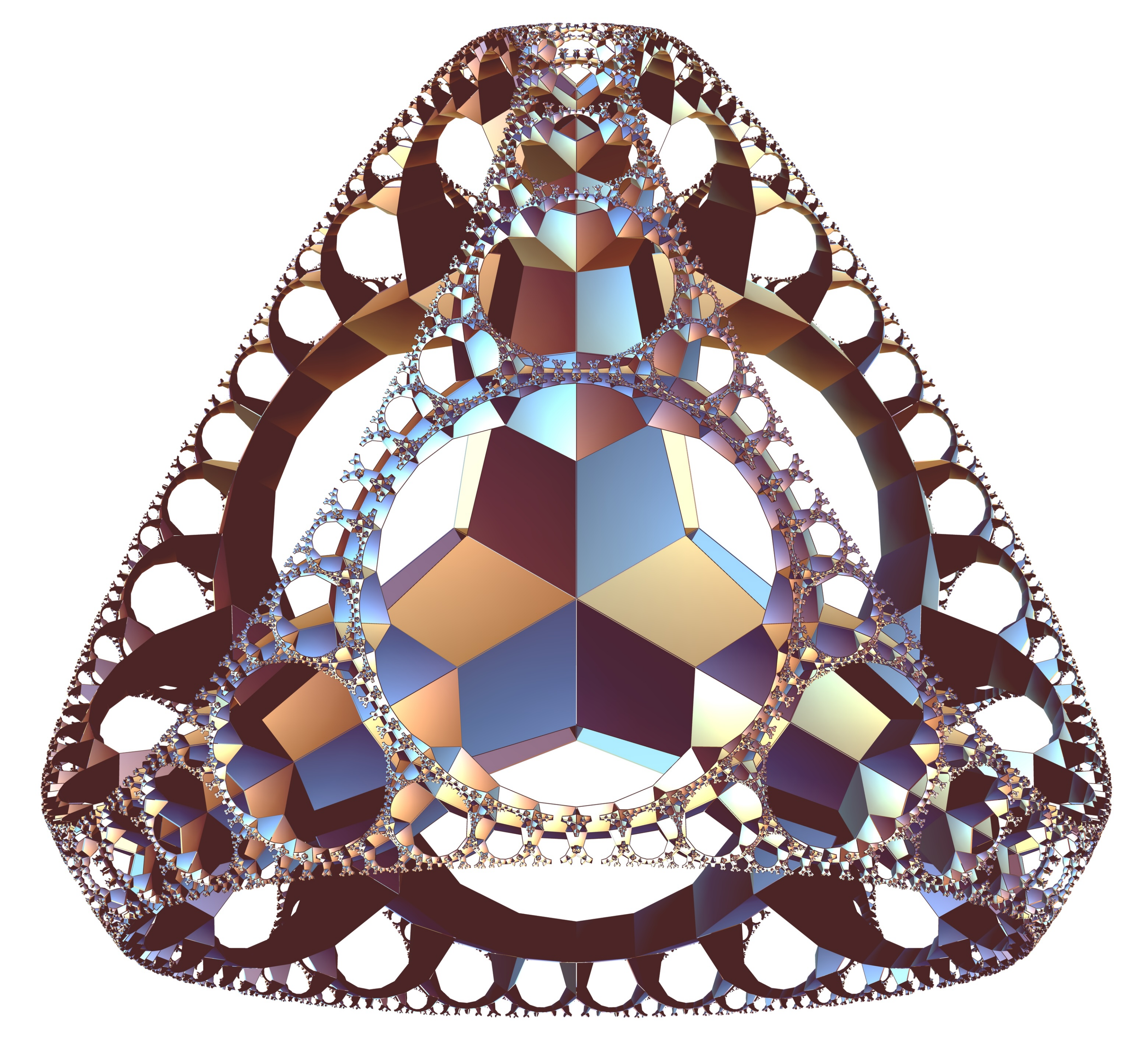} 
   \caption{A small portion of $\Pi_{3,3}^6$ and a larger portion of $\Pi_{3,3}^{10}$}
   \label{fig:pritetra}
\end{figure}

The same argument as in the cubical prismatic case from section \ref{sec:hypricu} shows:

\begin{theorem}
For $n\ge 6$ there exists a  hyperbolic prismatic tetrahedron $\Pi_{3,3}^n$.
\end{theorem}

Figure \ref{fig:pritetra} shows examples. One can clearly see that for $n>6$ the limit set of $\Gamma_{3,3}^n$ leaves large gaps in $\bS^2$ that originate from the circles where the hyperideal $P_{3,3}(a_n)$ meet $\bS^2$.

\subsection{Spherical Prismatic Tetrahedra}

We have spherical tetrahedra for $n=3$, $n=4$, and $n=5$ with dihedral angle $2\pi/n$. Their tilings of  $\bS^3$ correspond to the 5-cell, 16-cell and  600-cell, respectively.

\begin{figure}[h] 
   \centering
   \includegraphics[width=3in]{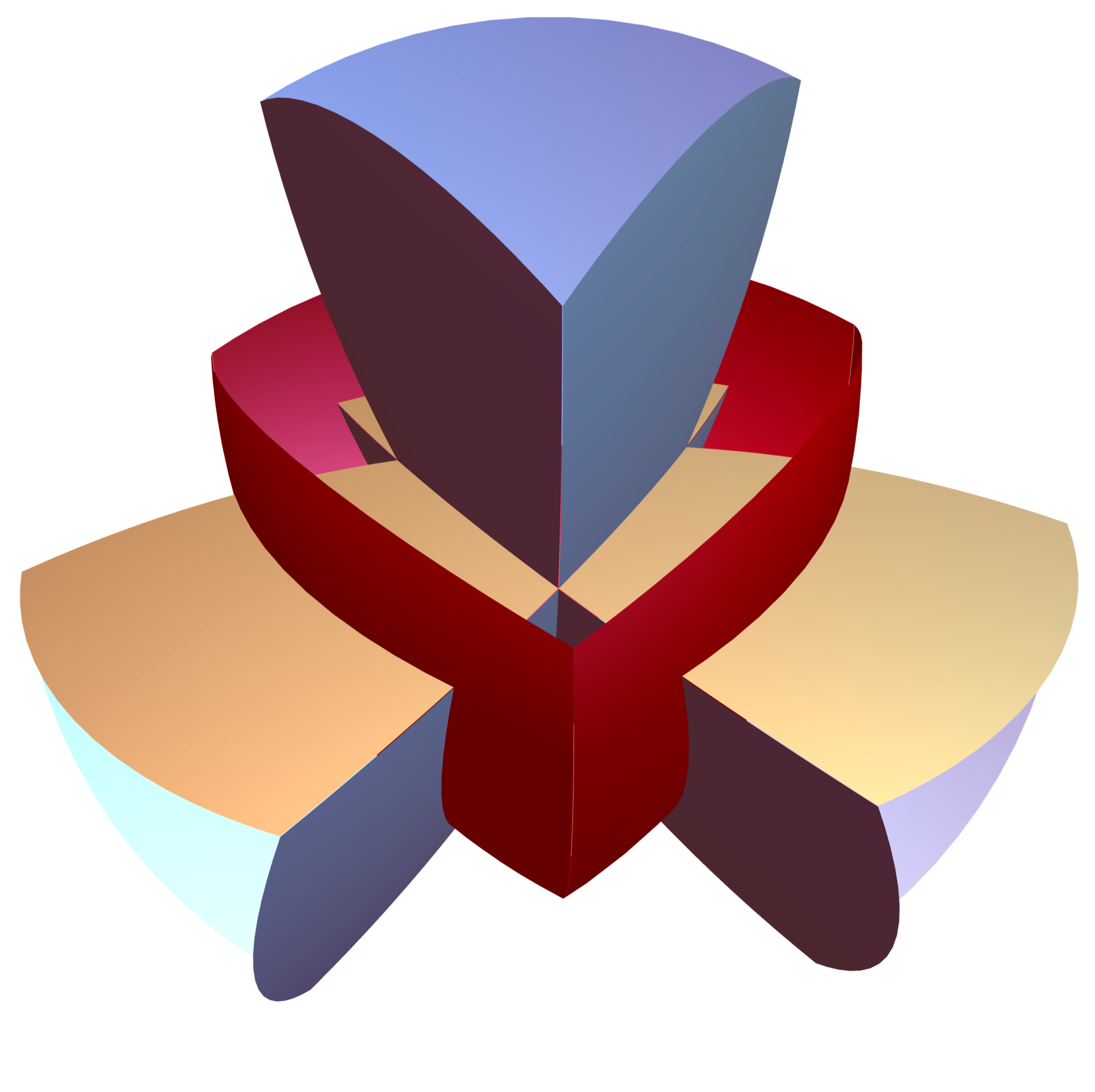
   }\qquad \includegraphics[width=3in]{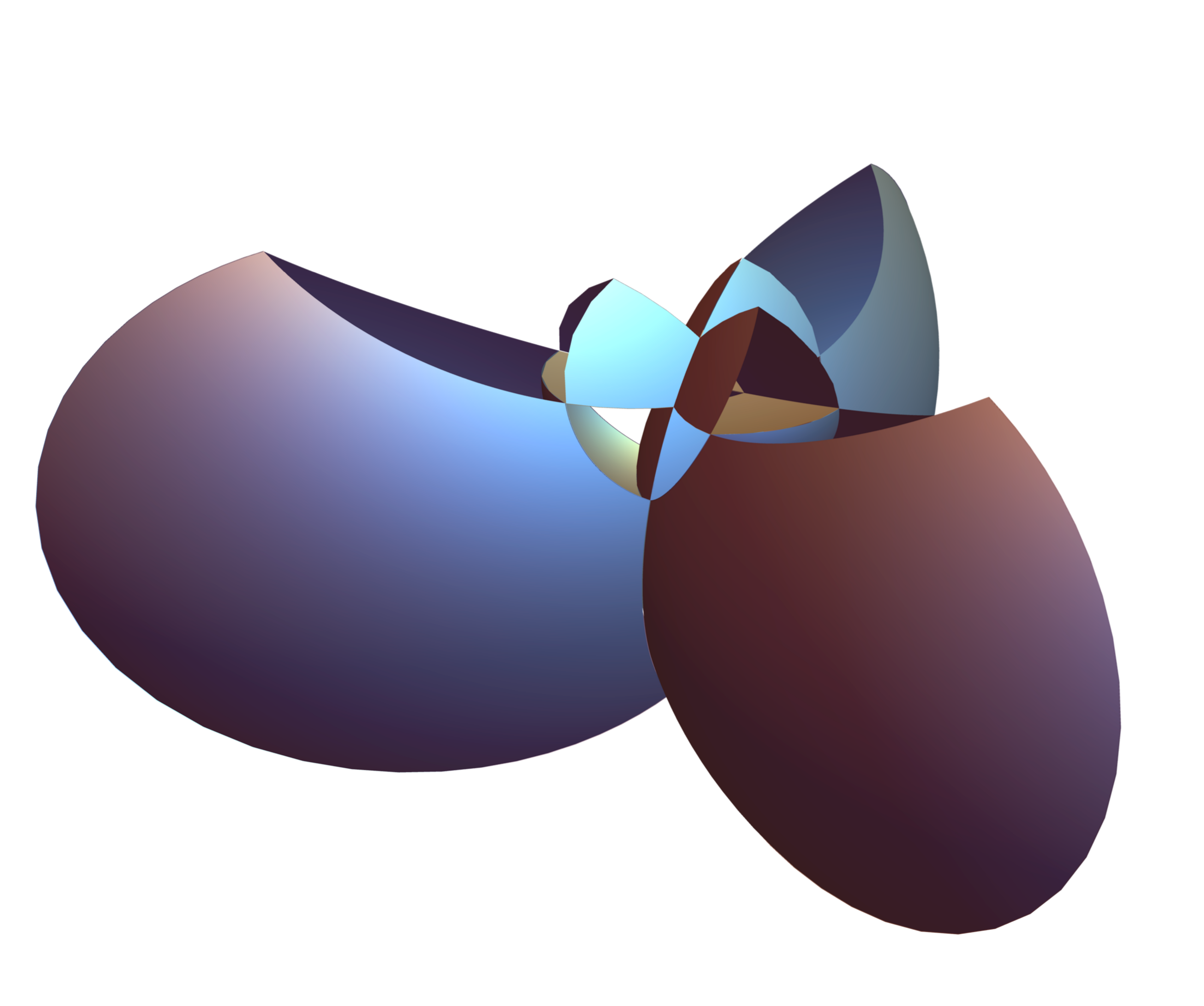} 
   \caption{Domain of  $\Pi_{3,3}^3$ and  most  of $\Pi_{3,3}^3$}
   \label{fig:pritetra3}
\end{figure}

Each of the 10, 32 and 1200 faces of these 4-polytopes is cut by a prism, so the spherical prismatic tetrahedra $\Pi_{3,3}^n$ will have 30, 96 and 3600 square faces.

\begin{figure}[h] 
   \centering
   \includegraphics[width=3in]{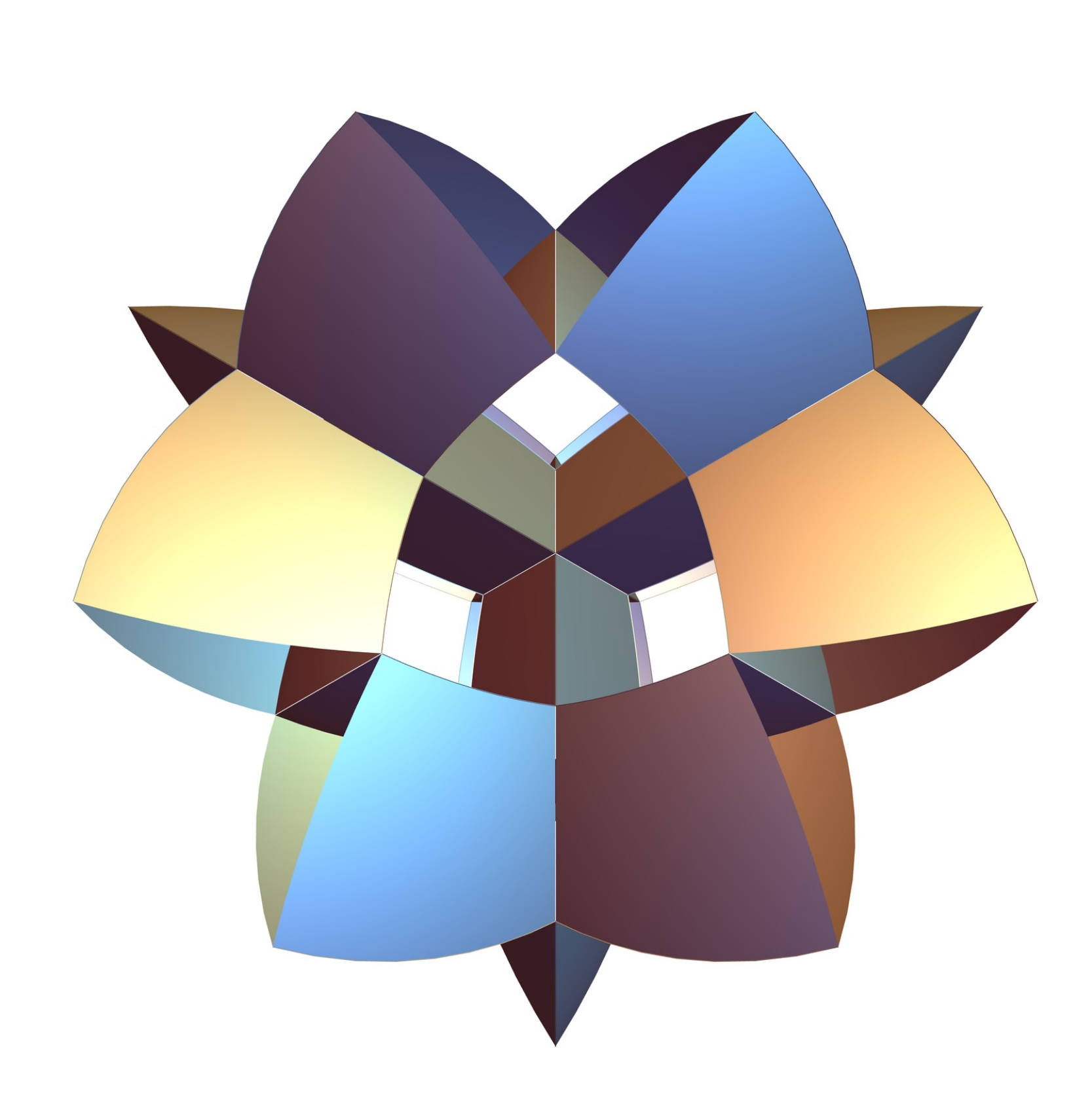
   }\qquad \includegraphics[width=3in]{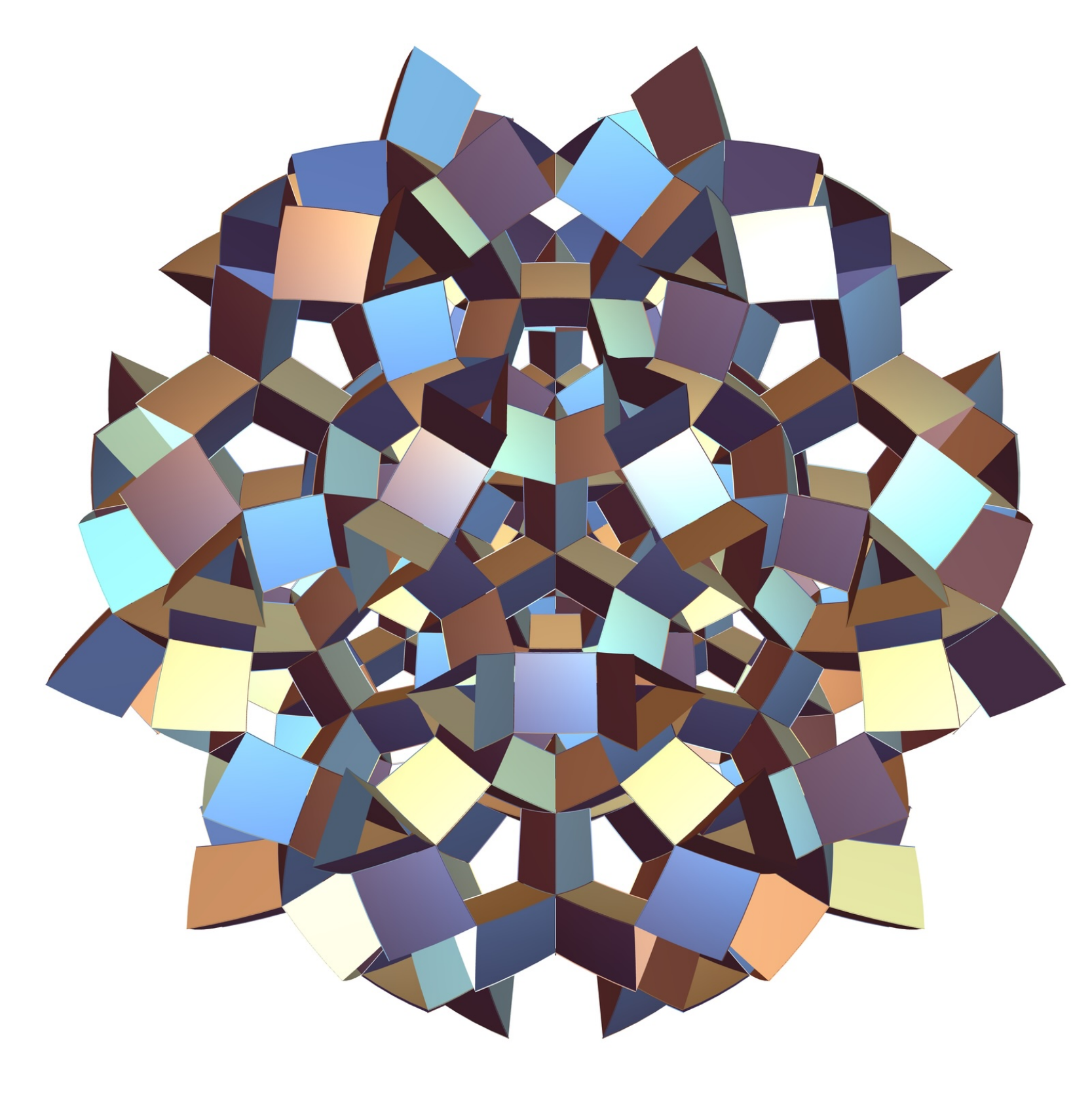} 
   \caption{Half of  $\Pi_{3,3}^4$ and  half of of $\Pi_{3,3}^5$}
   \label{fig:pritetrasphere}
\end{figure}

\section{The Prismatic Octahedron, Cuboctahedron, and Truncated Octahedron}

In this section we will discuss three different prismatic polyhedra that are all invariant under a group $\Gamma_{3,4}^n$ that is generated by reflections at the faces of an octahedron $P_{3,4}(a_n)$. 

The prismatic polyhedra will be constructed using   smaller octahedra $P_{3,4}(b_n)$ and truncations $TP_{3,4}$ and $RP_{3,4}$ of suitable sizes to attach prisms to their faces. In the two latter cases, the prisms are being attached only to the non-square faces of $TP_{3,4}$ and $RP_{3,4}$.

\subsection{Hyperbolic Prismatic Octahedra}

We begin by describing octahedra in hyperbolic space.

\begin{definition}
For a real number $a>1$, denote by $P_{3,4}(a)$ the intersection of the sets $H(a p)$ for all $p\in \{(\pm1,\pm1,\pm1)/\sqrt3\}$
with the unit ball. We call $P_{3,4}(a)$ a generalized hyperbolic octahedron.
\end{definition}

A simple computation shows:

\begin{lemma}
\begin{enumerate}
\item For  $a>\sqrt3$,  the four spheres $S(\pm a,\pm a,a)$ have an intersection $p$ with $|p|<1$. In this case $P_{3,4}(a)$ is a finite octahedron. For $a=\sqrt3$, the octahedron is ideal with dihedral angle $\pi/2$. For $a\to\infty$, the dihedral angle increases to $\arccos(-1/3)\approx 109.47^\circ$, the dihedral angle of the Euclidean octahedron.
\item For $a=\sqrt{3/2}$, $S(a(1,1,1)/\sqrt3$ and  $S(a(1,-1,1)/\sqrt3$ touch.
\item For $a<\sqrt{3/2}$, $S(a(1,1,1)/\sqrt3$ and  $S(a(1,-1,1)/\sqrt3$ are disjoint.
\item For $a>\sqrt{3/2}$, $S(a(1,1,1)/\sqrt3$ and  $S(a(1,-1,1)/\sqrt3$ meet at a dihedral angle  $\alpha$ with $\cos \alpha = \frac{3-a^2}{3(a^2-1)}$.
\item For   $\sqrt3>a>\sqrt{3/2}$, $P_{3,4}(a)$ is hyperideal.
\end{enumerate}
\end{lemma}

As a consequence, we can determine all octahedra that tile hyperbolic space:

For $a=\sqrt{3}$ we have $\alpha=90^\circ$ and $n=4$. This octahedron is ideal. No finite octahedron tiles hyperbolic space.

For any $n\ge 5$, 
\[
a_n=\frac{\sqrt{6} \cos \left(\frac{\pi }{n}\right)}{\sqrt{3 \cos \left(\frac{2 \pi
   }{n}\right)+1}}
   \]
determines the hyperideal octahedron that tiles hyperbolic space with dihedral angle $2\pi/n$.

\subsection{The Prismatic Octahedron in $\bH^3$}

\begin{figure}[h] 
   \centering
   \includegraphics[width=2.5in]{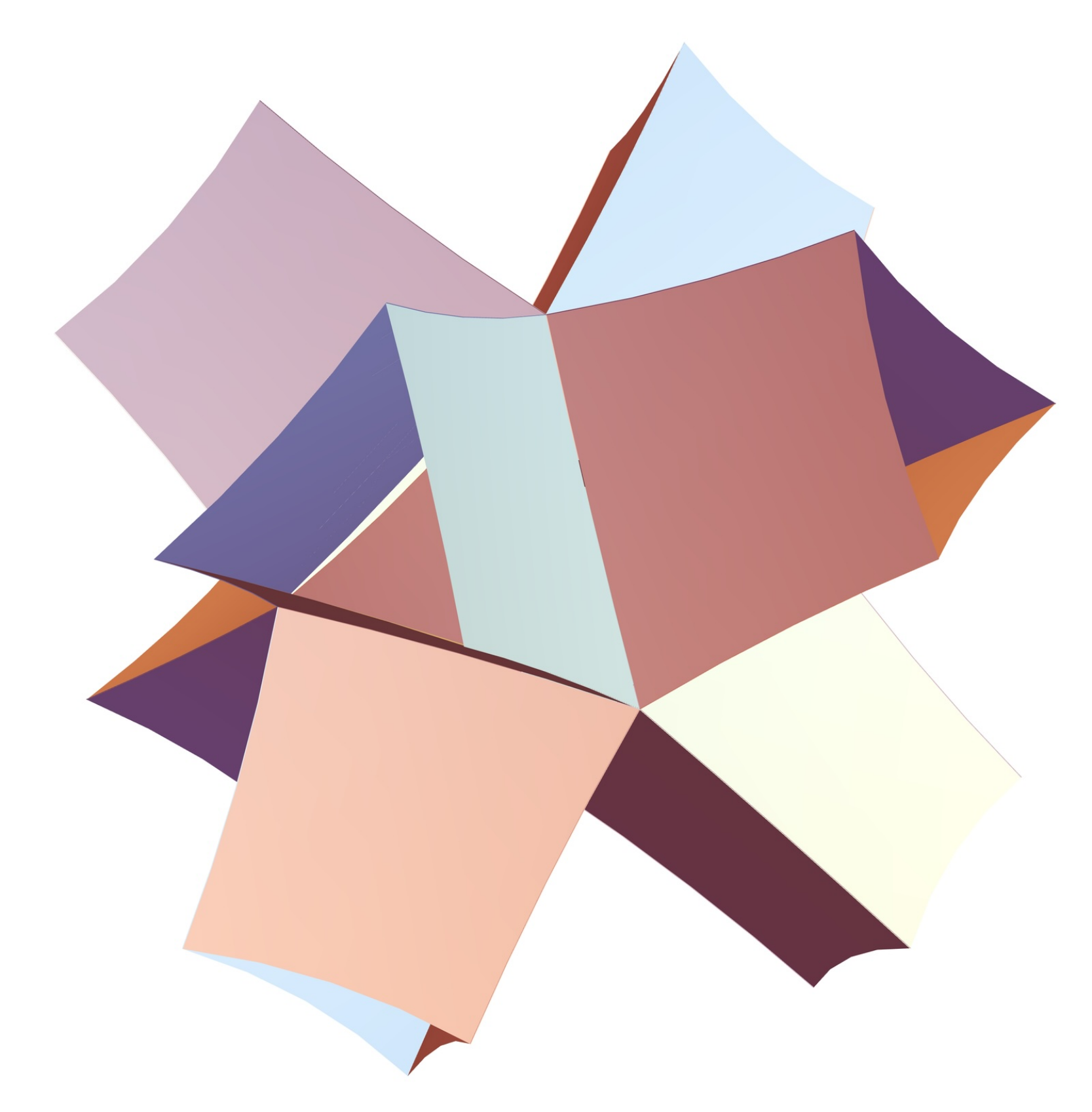}
   \qquad 
   \includegraphics[width=2.8in]{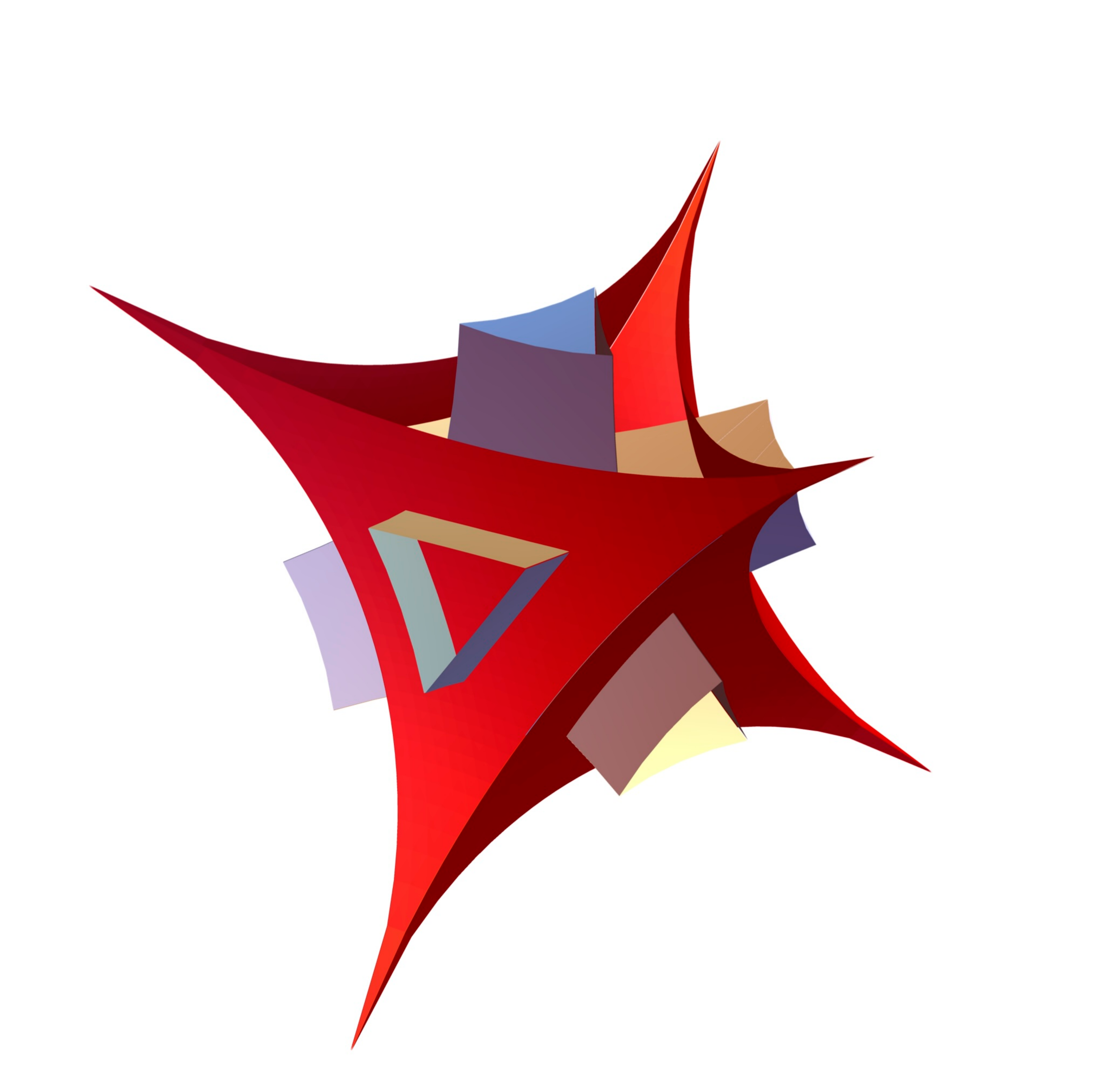}
      \caption{The Hyperbolic Prismatic Octahedron $\Pi_{3,4}^4$}
   \label{fig:hypprisocta4}
\end{figure}

For $n\ge 4$,  place a finite octahedron $P_{3,4}(b)$ inside $P_{3,4}(a_n)$. Replace each face of $P_{3,4}(b)$ by a regular prism over that face pointing outside. We have, with the same proof as before:

\begin{theorem}
For each $n\ge 4$ there is a value of $b=b_n$ so that the prisms over the faces of   $P_{3,4}(b)$ are symmetric with respect to inversions at the corresponding faces of $P_{3,4}(a_n)$. Continuing to invert at the faces of $P_{3,4}(a_n)$ creates an infinite hyperbolic polyhedron $\Pi_{3,4}^n$ with square faces and valency 8 at each vertex. The identification space is a  compact surface of genus 4 tiled by 12 squares with valency 8.
\end{theorem}

In figure \ref{fig:hypprisocta46} you can see portions of hyperbolic prismatic octahedra $\Pi_{3,4}^4$ and $\Pi_{3,4}^6$.

\begin{figure}[h] 
   \centering
   \includegraphics[width=2.5in]{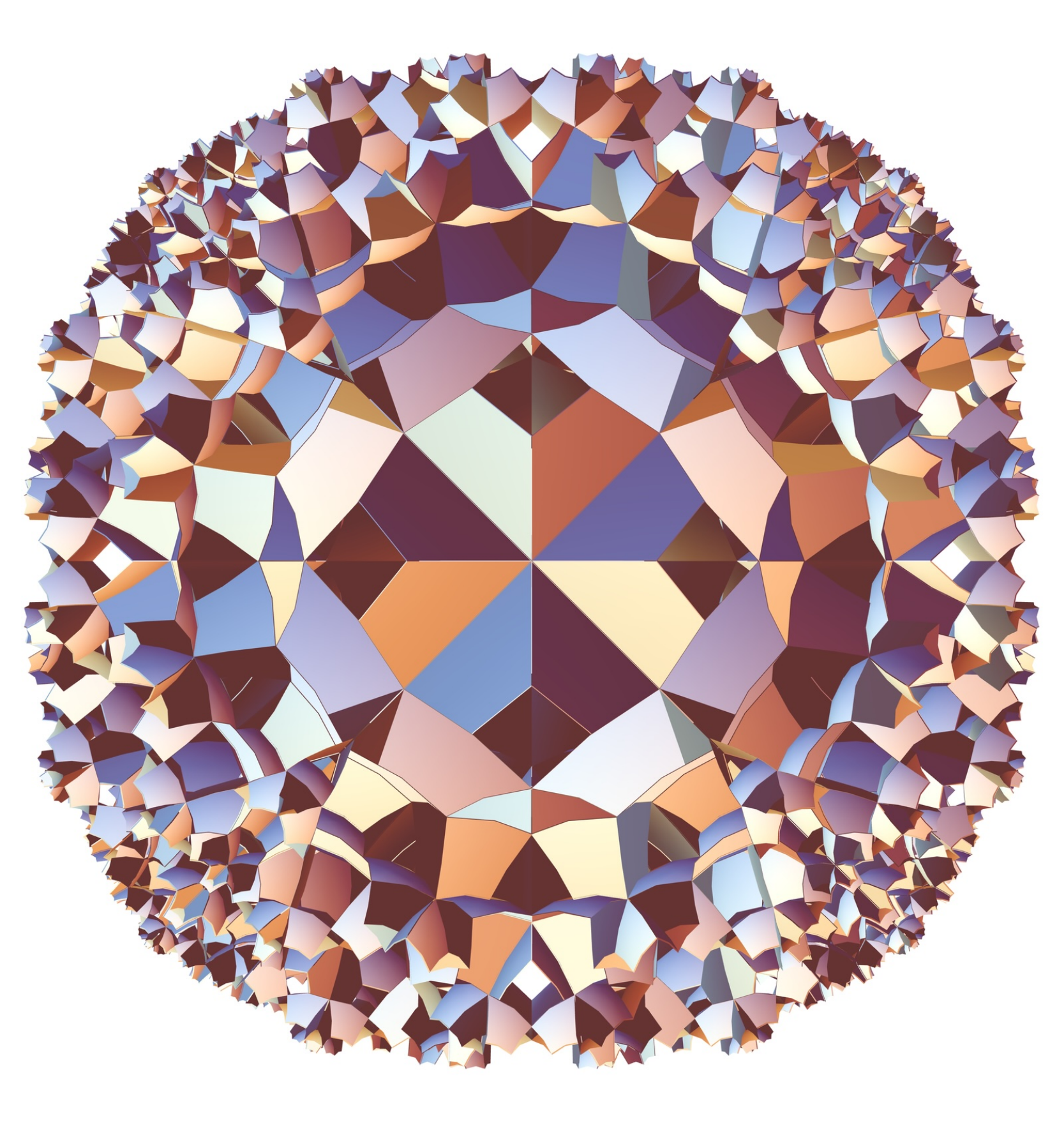}
   \qquad 
   \includegraphics[width=2.5in]{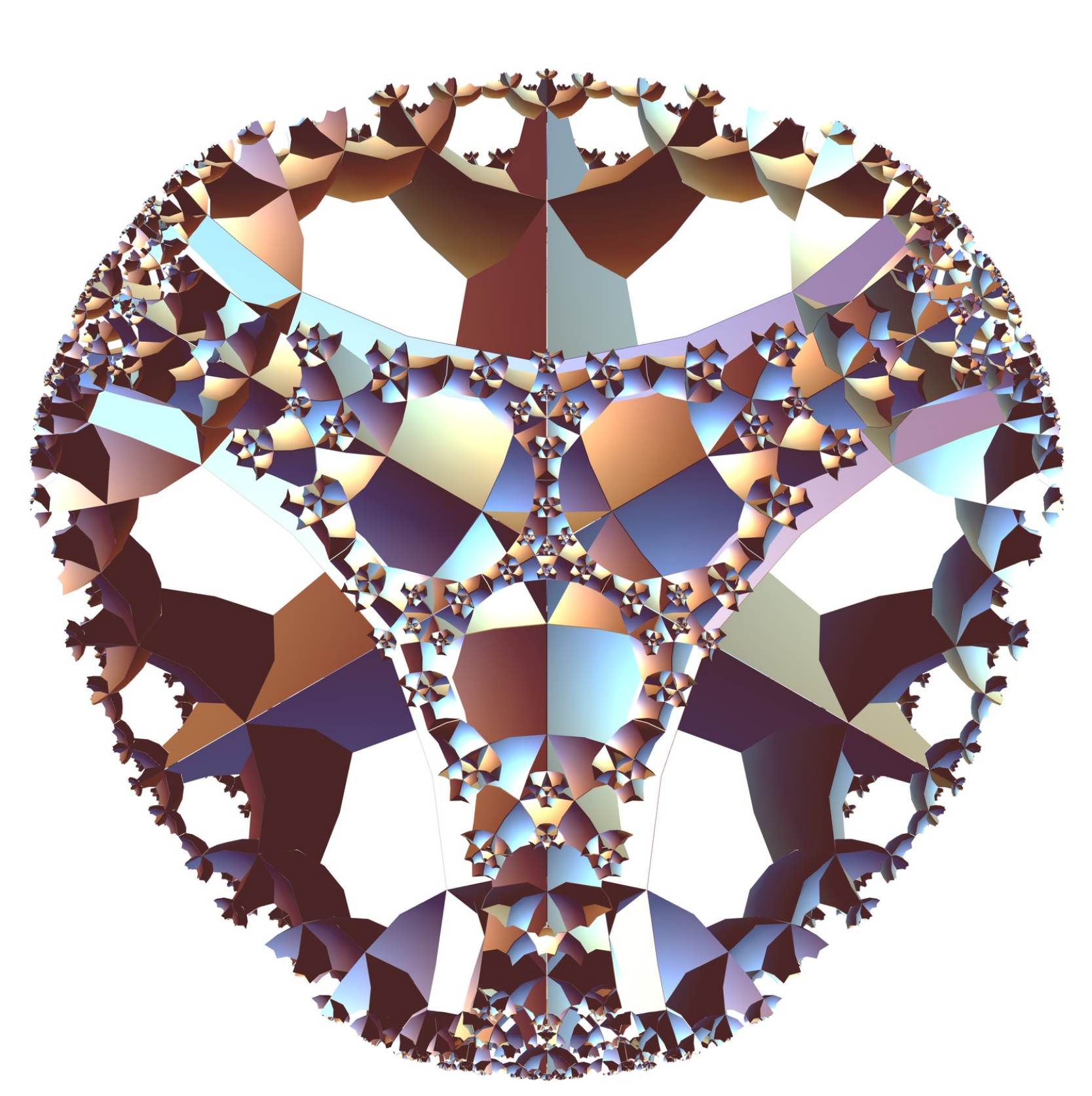}
   \caption{Hyperbolic Prismatic Octahedra for $n=4$ and $n=6$}
   \label{fig:hypprisocta46}
\end{figure}

\subsection{The Prismatic Cuboctahedron and Truncated Octahedron in $\bH^3$}
\label{sec:ppricub}

The construction for the prismatic octahedron can be varied as follows. We still use the same hyperbolic octahedra $P_{3,4}(a_n)$ as fundamental domains but use cuboctahedra $RP_{3,4}$ or truncated octahedra $TP_{3,4}$ as building blocks for the prismatic versions. In either case, we keep the square faces of the truncations and construct prisms over the triangular (resp. hexagonal) faces of the cuboctahedron (resp. truncated octahedron), see figure \ref{fig:hyppristruncocta1}.

\begin{figure}[h] 
   \centering
   \includegraphics[width=2.5in]{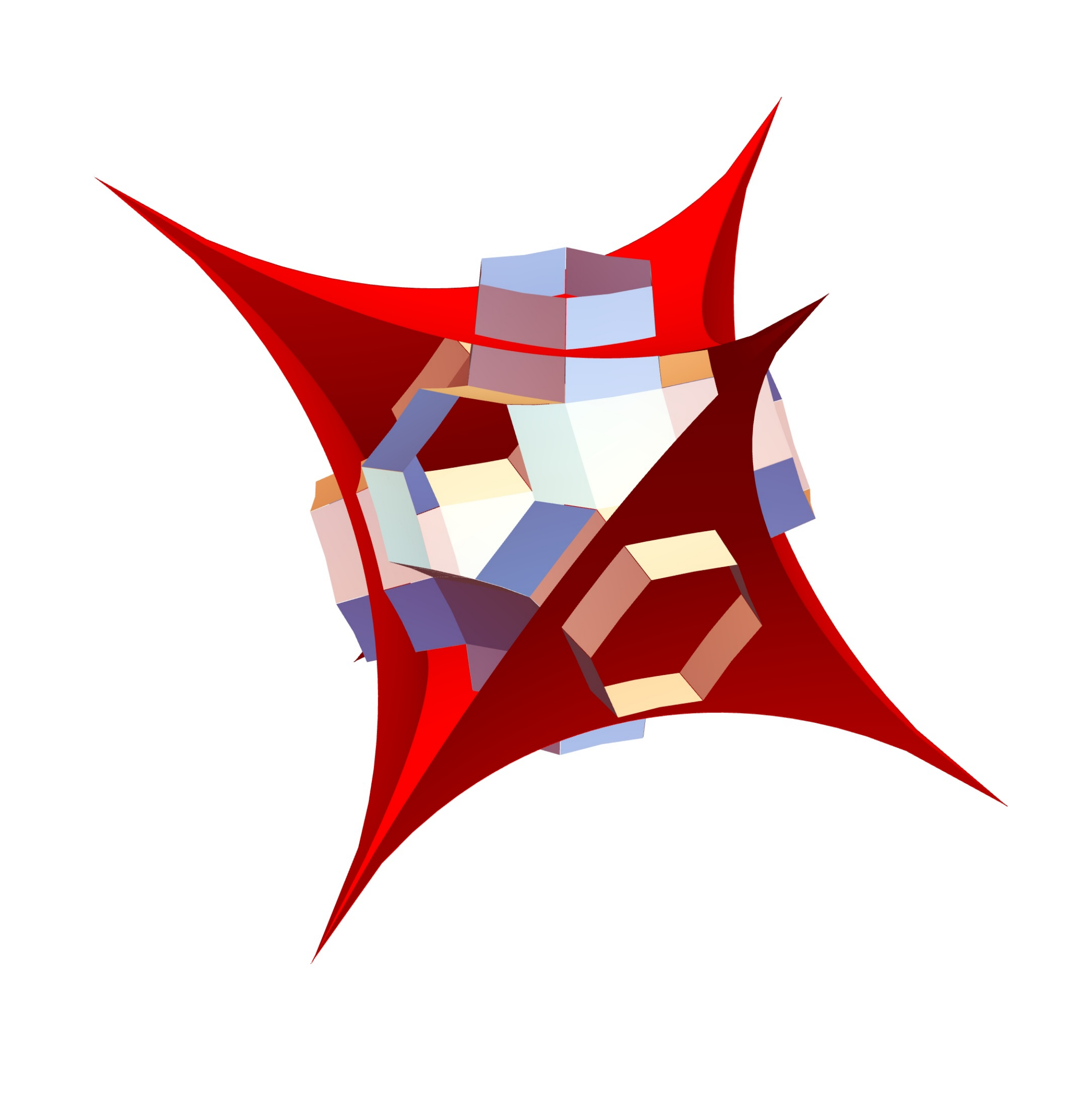}
   \qquad 
   \includegraphics[width=2.5in]{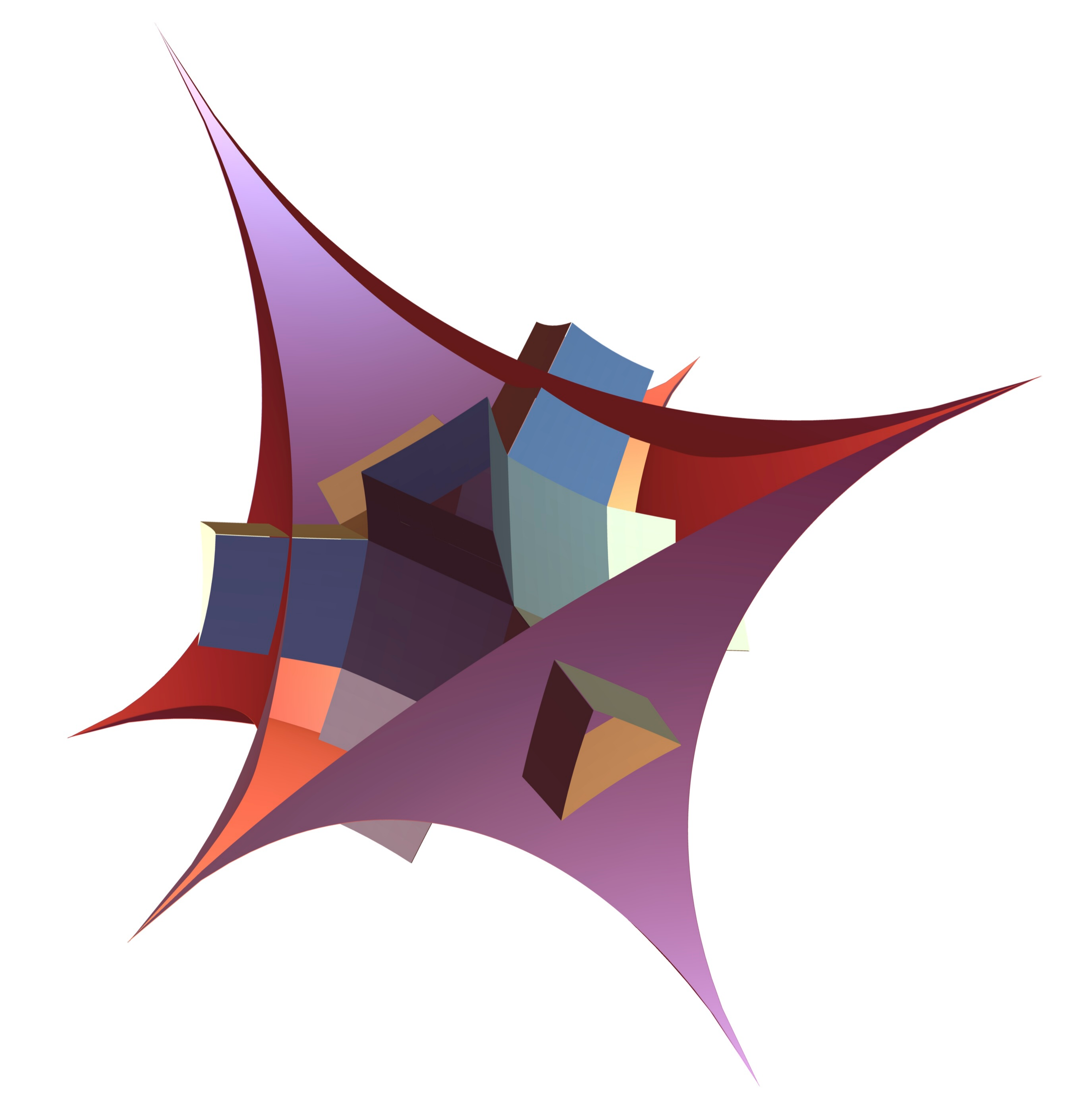}
   \caption{Hyperbolic Prismatic Truncated Octahedron and Cuboctahedron}
   \label{fig:hyppristruncocta1}
\end{figure}

The resulting surfaces exist again for any $n\ge 4$ (with essentially the same proof):

\begin{theorem}
For $n\ge 4$ there are hyperbolic prismatic cuboctahedra $RP_{3,4}^n$ and truncated octahedra $TP_{3,4}^n$ invariant under $\Gamma_{3,4}^n$. 

The identification surface of  $RP_{3,4}^n$ has genus 4 and is tiled by 18 hyperbolic squares with valency 6.

The identification surface of  $TP_{3,4}^n$ has genus 4 and is tiled by 30 hyperbolic squares with valency 5.
\end{theorem}

None of the maps for the three kinds of prismatic octahedral surfaces can be Platonic as can be seen by inspecting closed (combinatorial) geodesics that pass through opposite edges of squares. For instance, for the prismatic octahedra, we will have such closed geodesics of length 3 (i.e. passing through three squares) for the waists of the prisms, but regardless how we identify the prisms, closed geodesics through bottom and top of the prisms will always have an even length.

\begin{figure}[h] 
   \centering
   \includegraphics[width=2.5in]{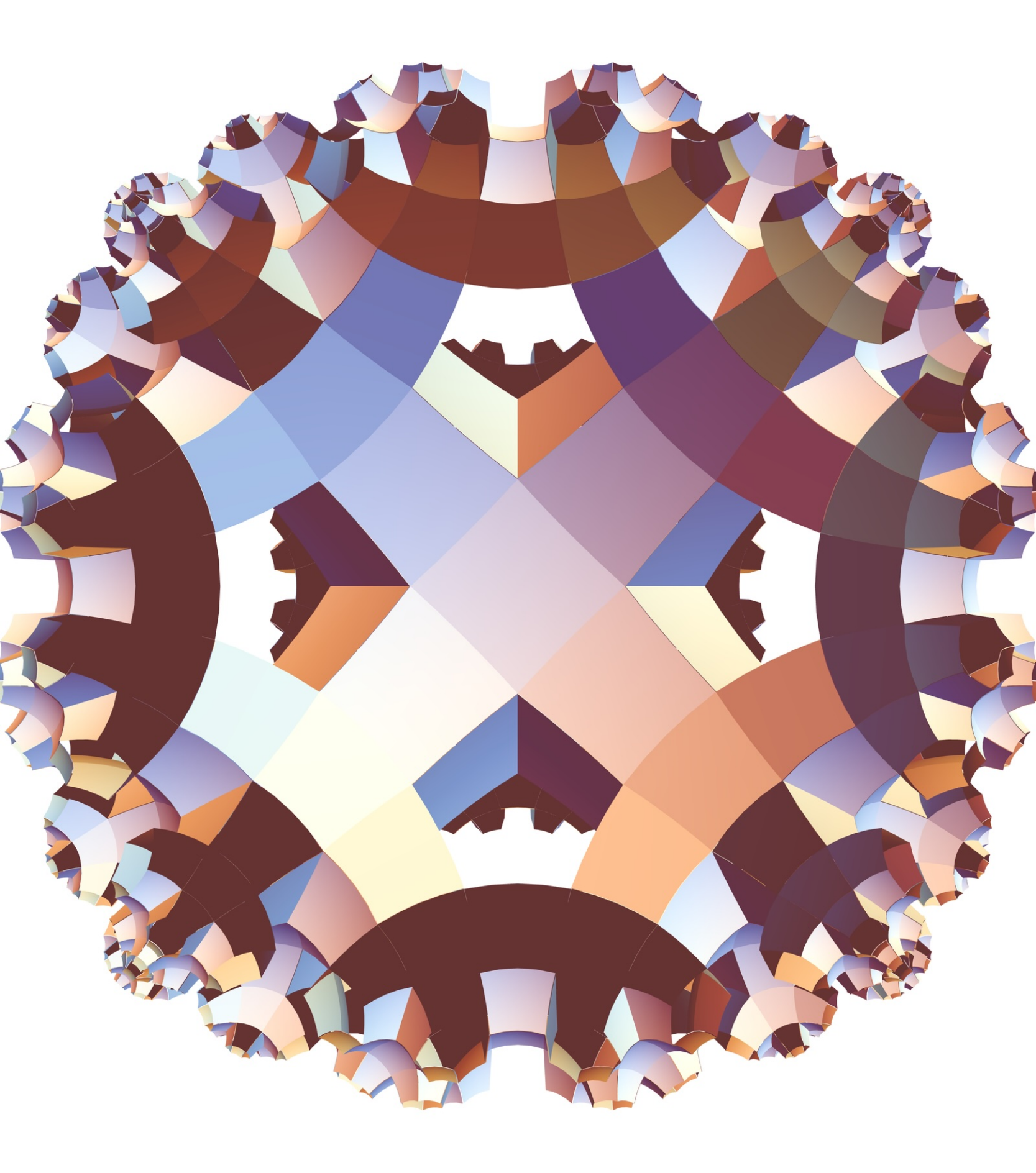}
   \qquad 
   \includegraphics[width=2.5in]{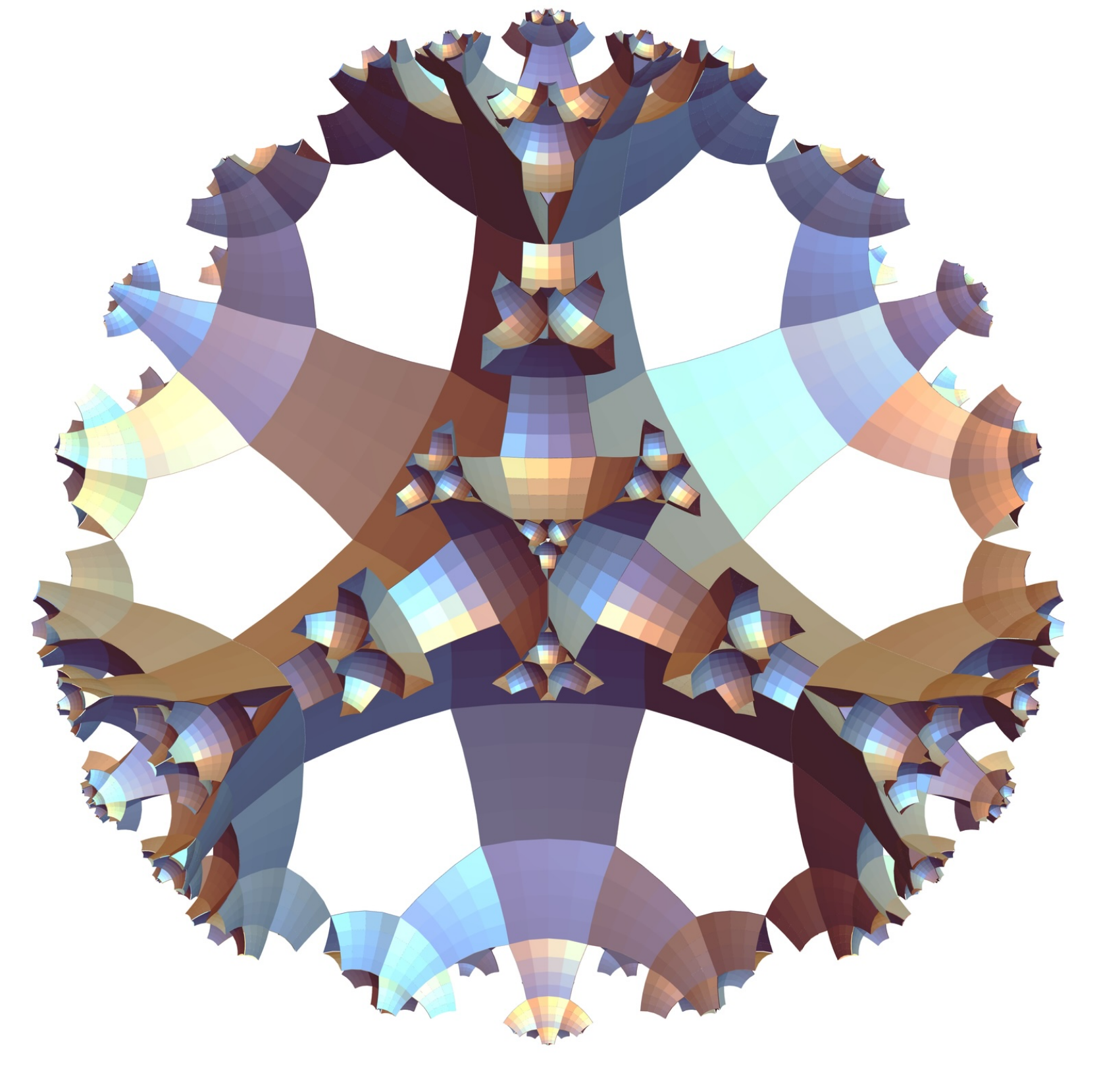}
   \caption{Hyperbolic Prismatic Truncated Octahedron and Cuboctahedron}
   \label{fig:hyppristruncocta}
\end{figure}

\subsection{The Prismatic Cuboctahedron and Truncated Octahedron in $\bS^3$}
\label{sec:pricubsp}

There also exist spherical versions of octahedron, truncated octahedron and cuboctahedron for $n=3$, using the octahedra of the 24-cell for the construction. See figure \ref{fig:sphprismocta}.

\begin{figure}[h] 
   \centering
   \includegraphics[width=1.9in]{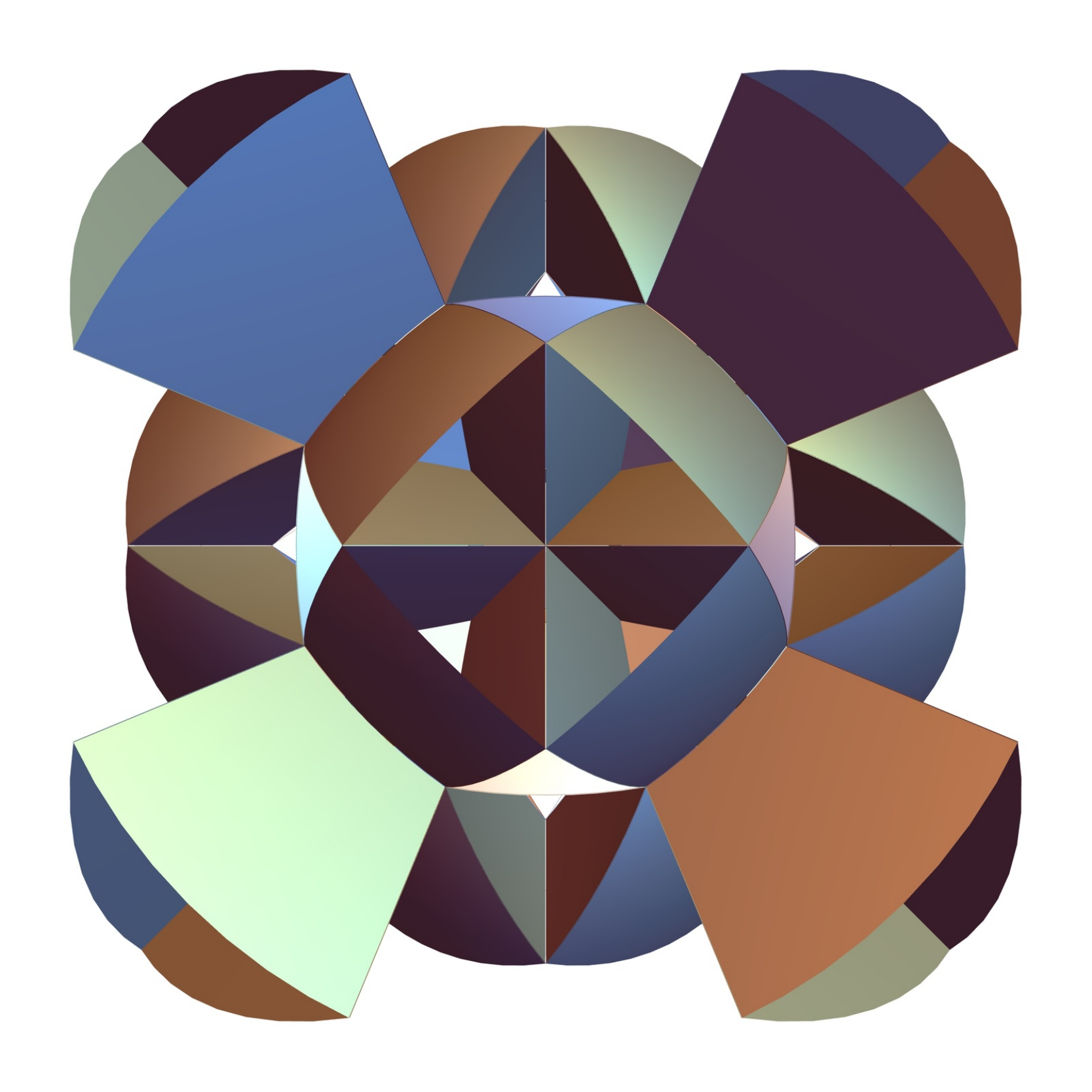}
   \qquad 
   \includegraphics[width=1.9in]{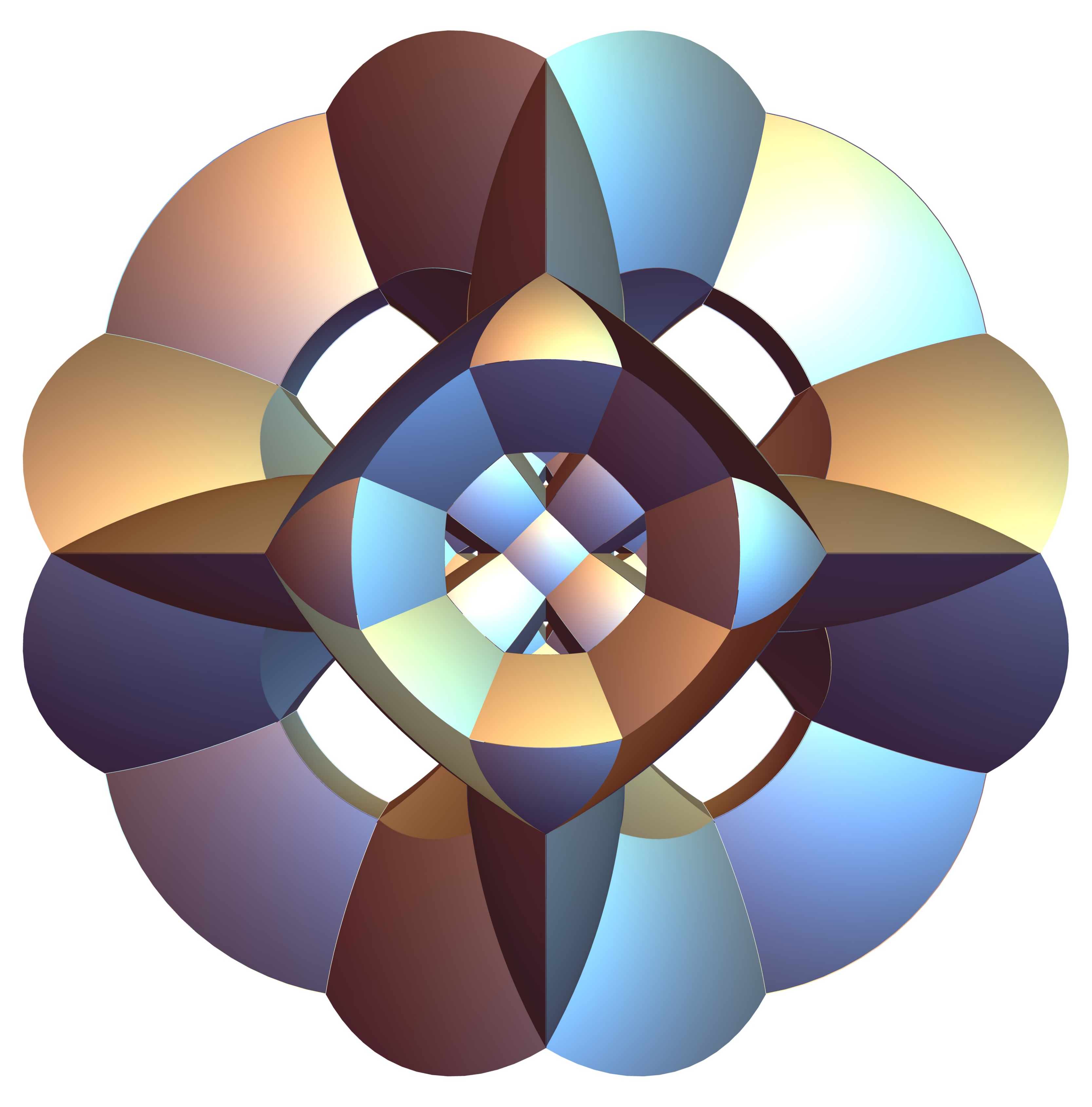}
   \qquad 
   \includegraphics[width=1.9in]{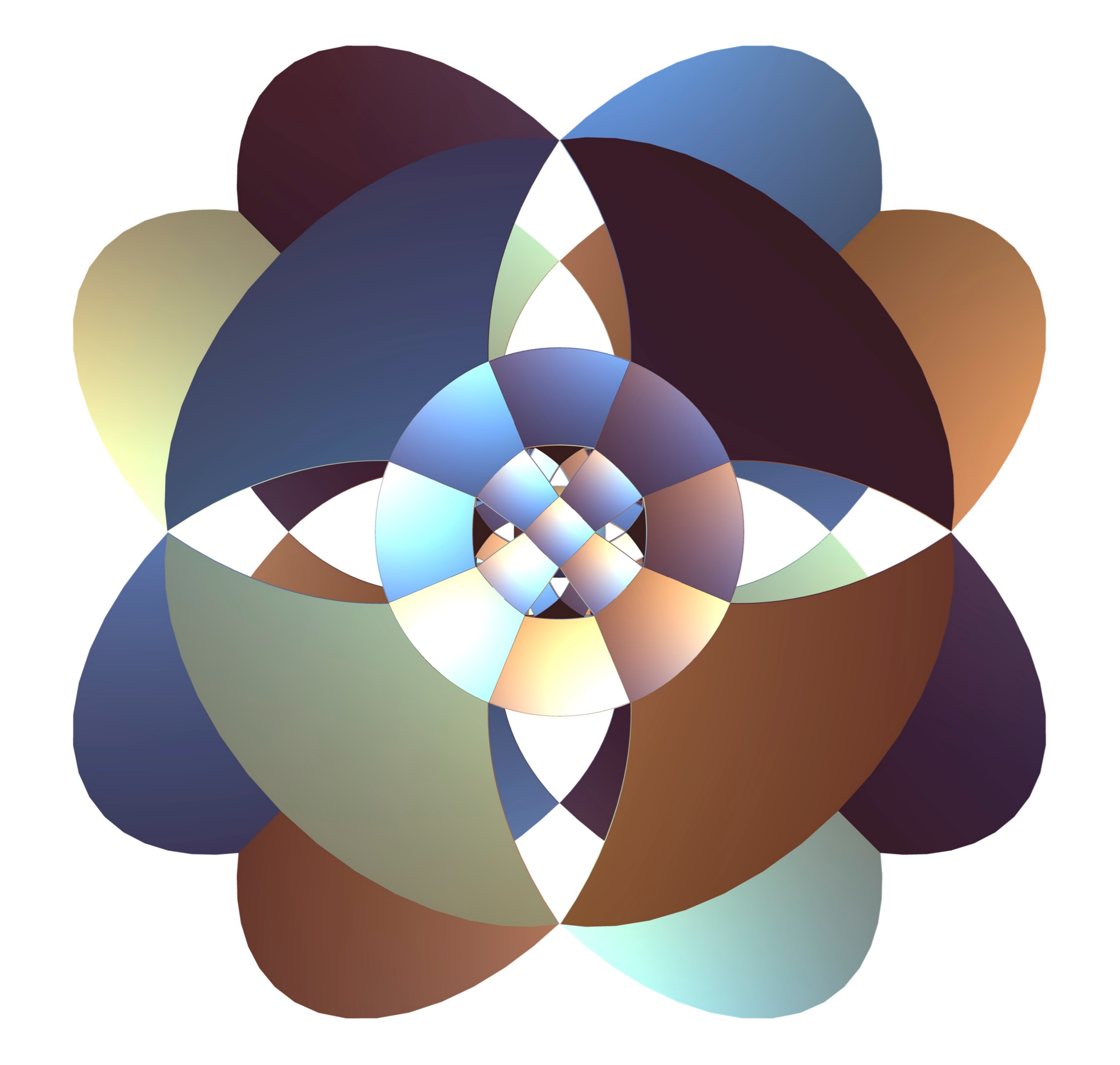}
   \caption{Spherical Prismatic Octahedron, Truncated Octahedron and Cuboctahedron}
   \label{fig:sphprismocta}
\end{figure}

\section{Prismatic Dodecahedra and Icosahedra}

We can also use dodecahedra and icosahedra to create hyperbolic and spherical prismatic polyhedra.

\begin{figure}[h] 
   \centering
   \includegraphics[width=2.5in]{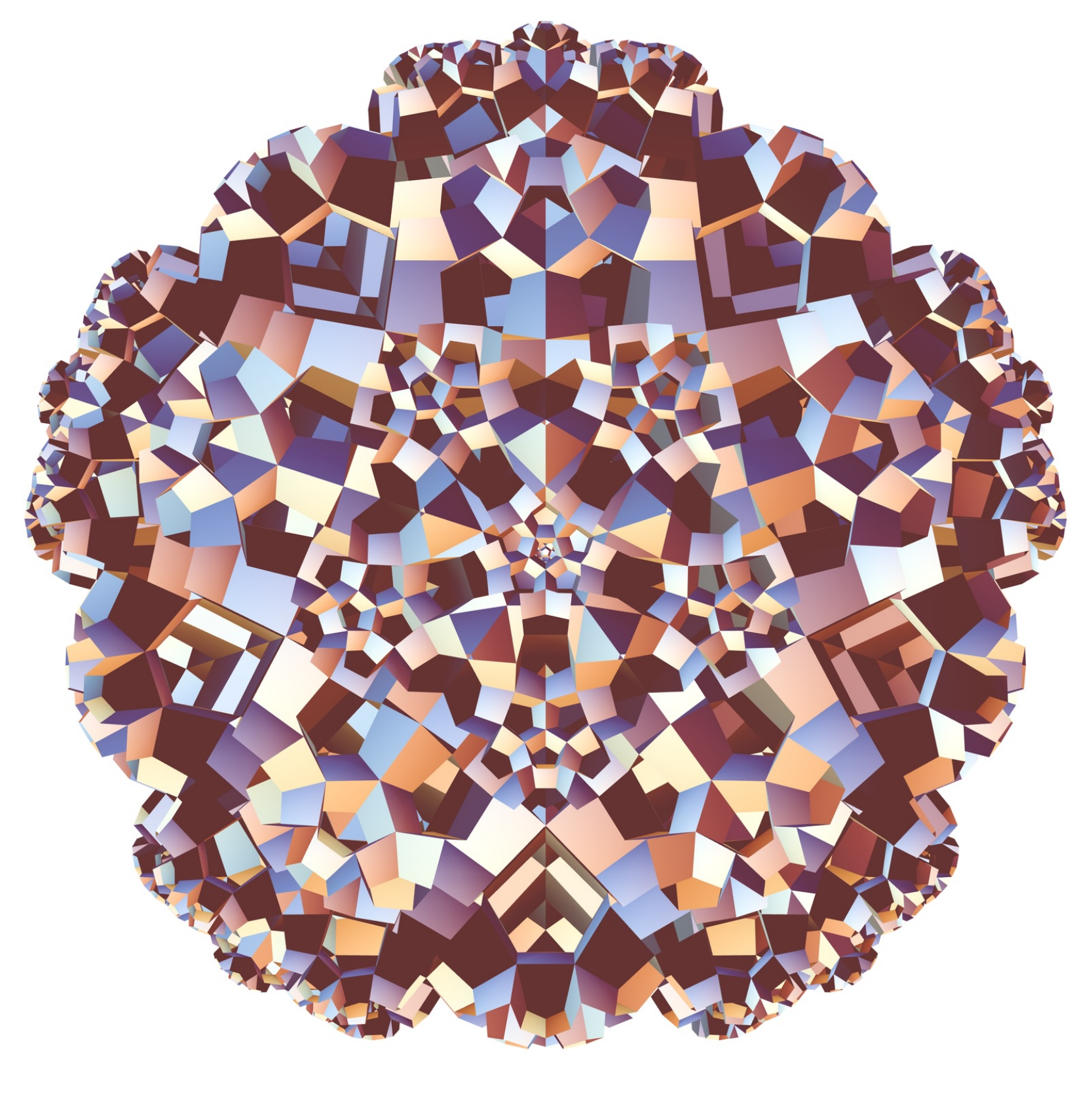}
   \qquad 
   \includegraphics[width=2.5in]{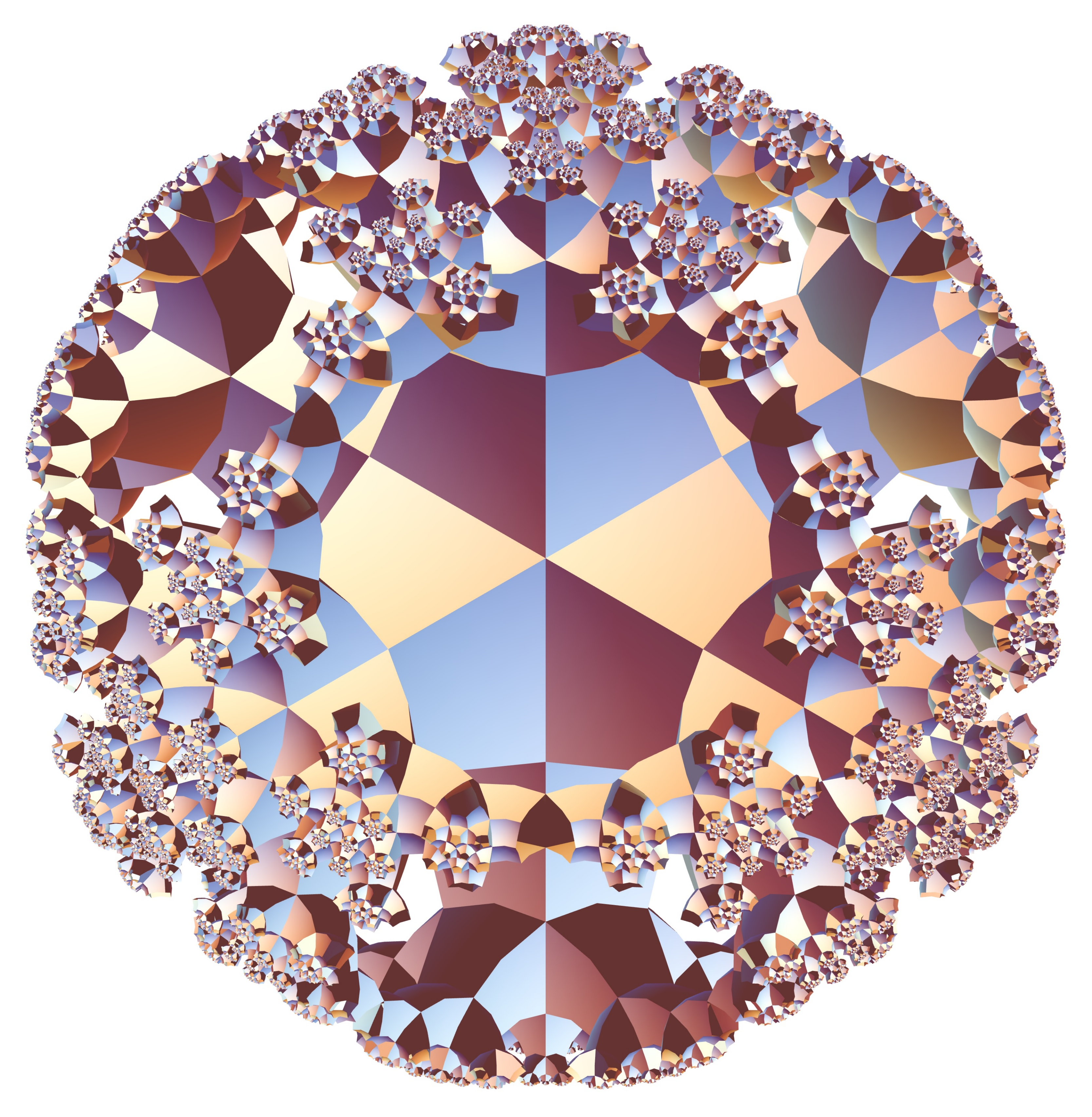}
   \caption{Hyperbolic prismatic dodecahedra for $n=4$ and $n=8$}
   \label{fig:hypprisdodeca}
\end{figure}

\begin{theorem}
For each $n\ge3$ there exists a prismatic dodecahedron $\Pi_{5,3}^n$. For $n=3$ the prismatic dodecahedron resides in $\bS^3$, and its prisms are symmetric with respect to the 2-dimensional faces of the 120-cell. For $n\ge4$ the prismatic dodecahedra reside in hyperbolic space. The quotient surface is a compact polyhedral surface of genus 6 tiled by 30 squares. 
\end{theorem}

Figure \ref{fig:hypprisdodeca} shows portions of the hyperbolic $\Pi_{5,3}^4$ and $\Pi_{5,3}^8$, while in 
figure \ref{fig:sphprisdodeca} you can see a fundamental piece and all of $\Pi_{5,3}^3$.

\begin{figure}[h] 
   \centering
   \includegraphics[width=2.5in]{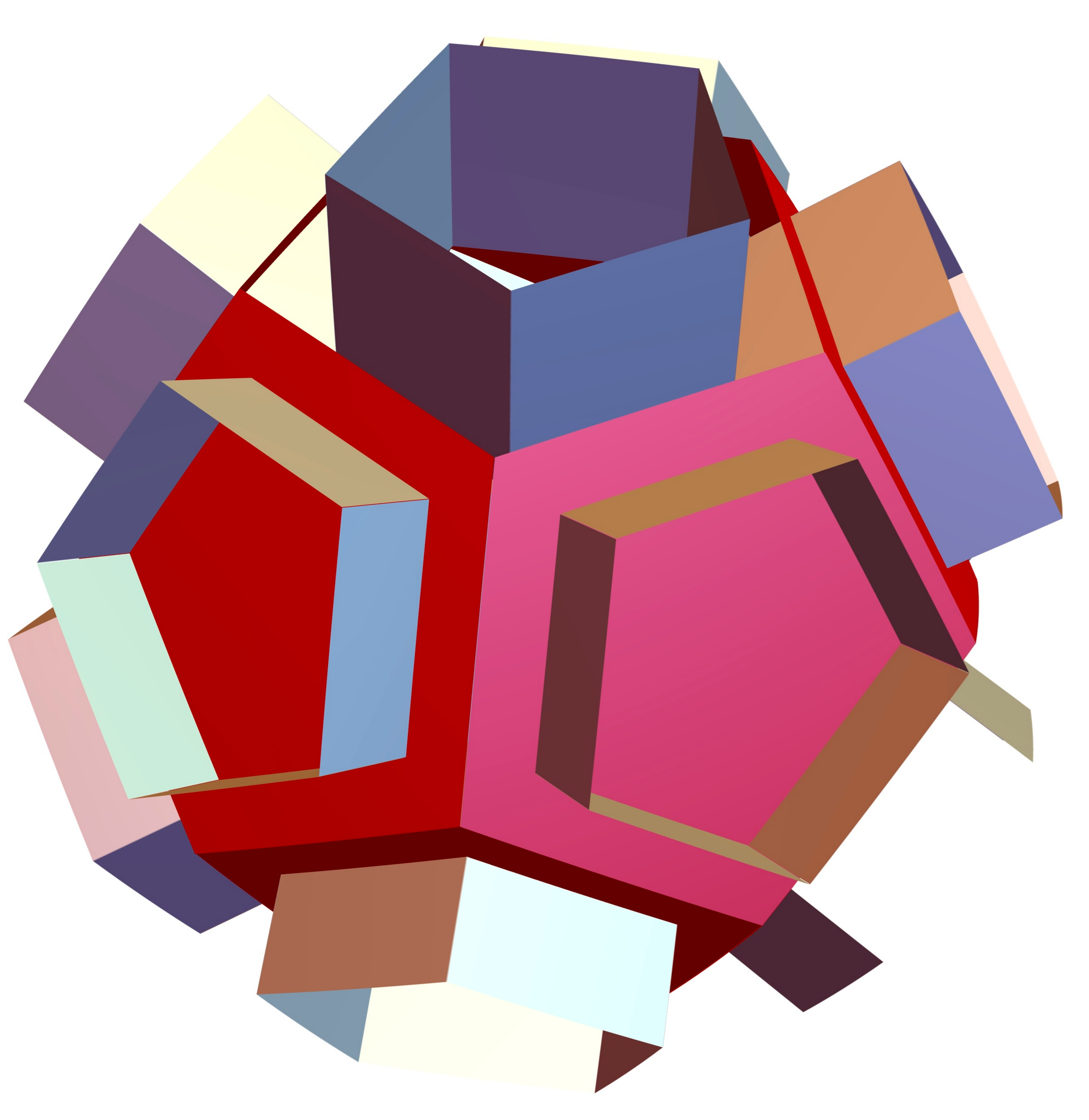}
   \qquad 
   \includegraphics[width=2.5in]{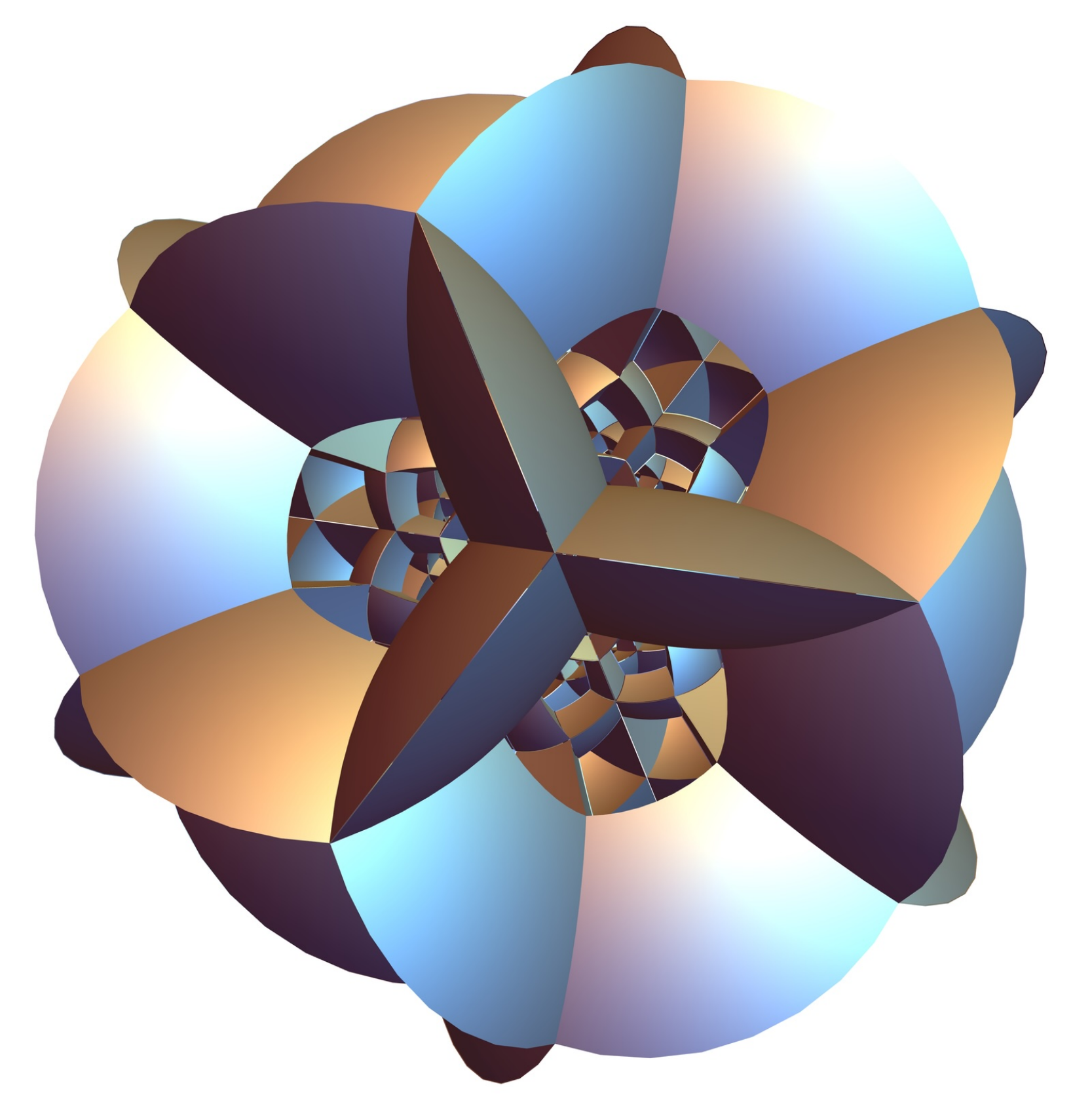}
   \caption{Spherical prismatic dodecahedron}
   \label{fig:sphprisdodeca}
\end{figure}

Similarly, we have an existence theorem for hyperbolic prismatic icosahedra $\Pi_{3,4}^n$:
\begin{theorem}
For each $n\ge3$ there exists a prismatic icosahedron in hyperbolic space. The quotient surface is a compact polyhedral surface of genus 10 tiled by 30 squares with valency 10 at each vertex. 
\end{theorem}

\begin{figure}[h] 
   \centering
   \includegraphics[width=1.8in]{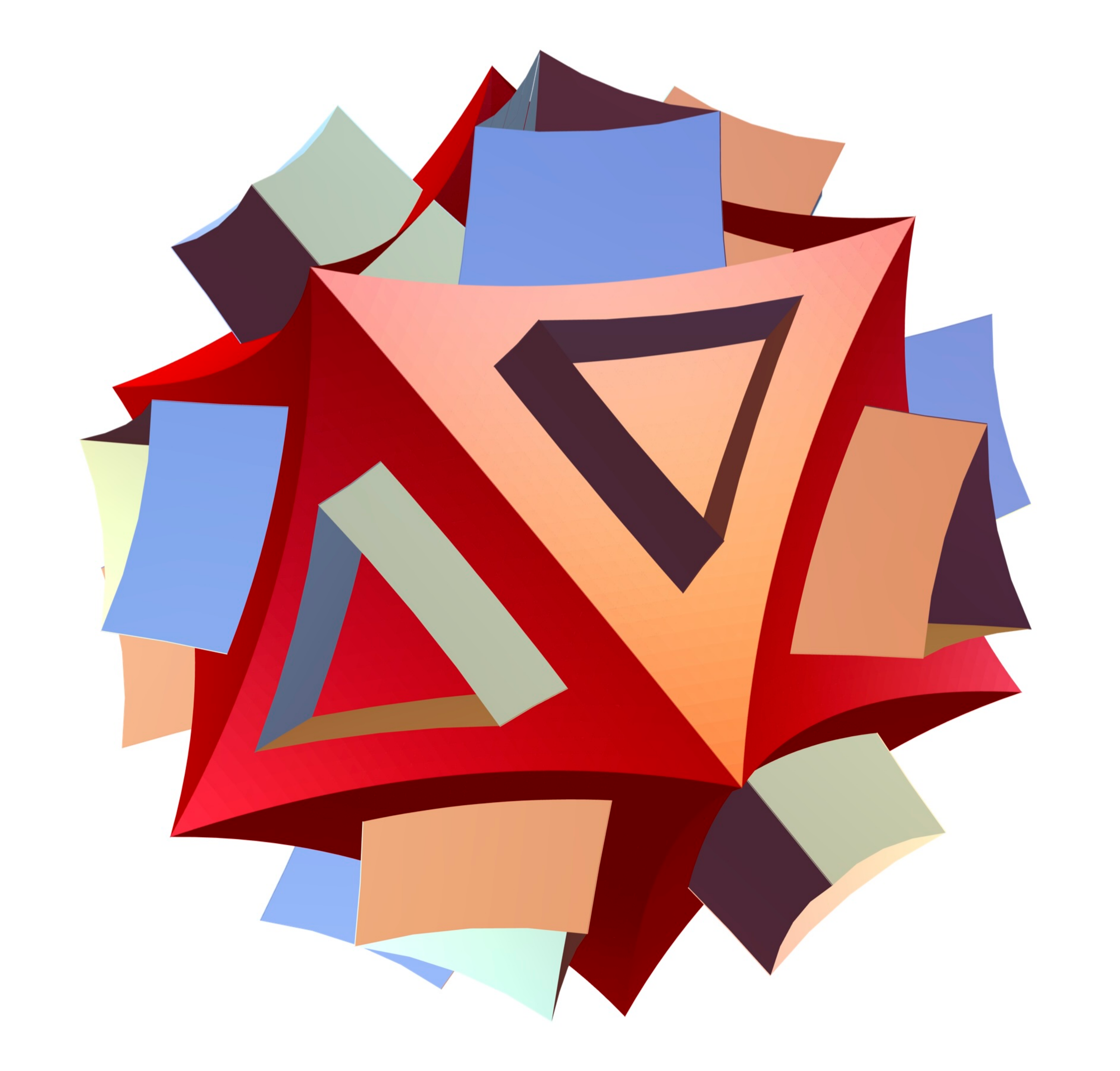}
   \qquad 
   \includegraphics[width=1.8in]{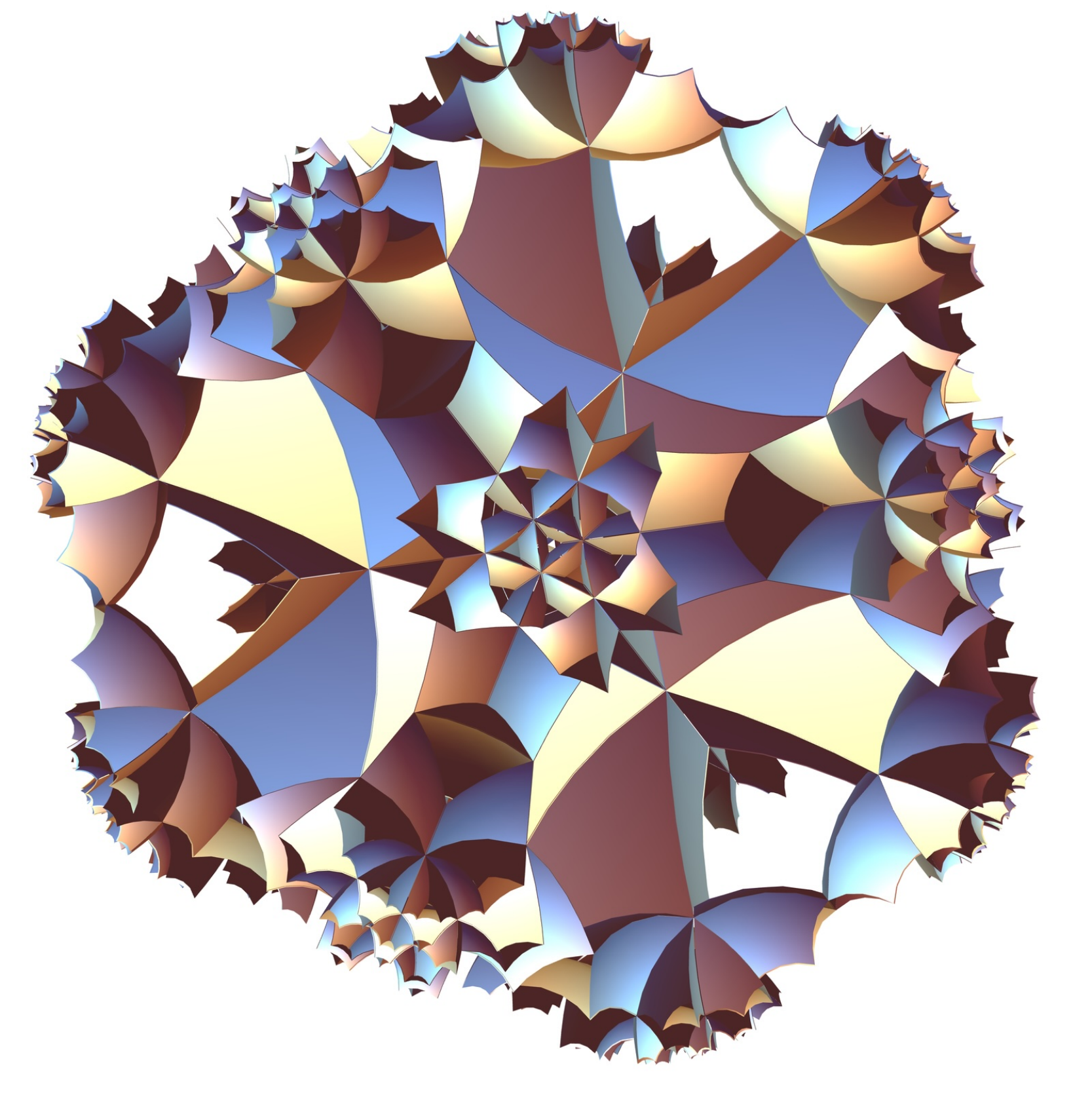}
   \qquad 
   \includegraphics[width=1.8in]{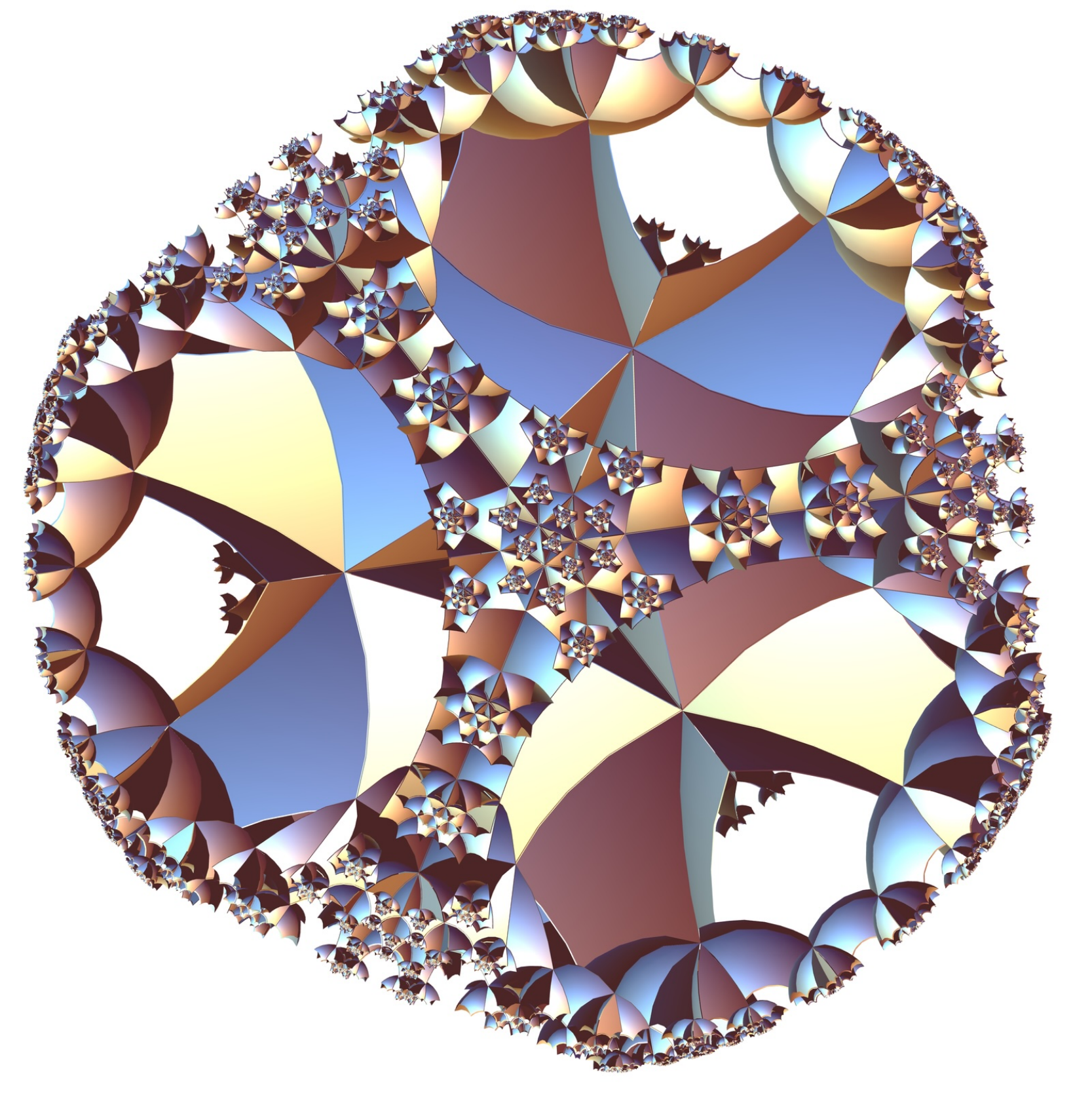}
   \caption{Hyperbolic prismatic icosahedra}
   \label{fig:sphprisicosa}
\end{figure}

In figure \ref{fig:sphprisicosa} you can see  portions of $\Pi_{3,5}^3$, $\Pi_{3,5}^4$ and $\Pi_{3,5}^5$.

\section{Antiprismatic Polyhedra}

\subsection{Introduction}

In this section, we will construct antiprismatic polyhedra  analogously to the prismatic polyhedra: We begin with a Platonic solid, attach regular antiprisms to all of its faces, then attach more Platonic solids to the top faces of the antiprisms, and keep going. Again a condition on the dihedral angle of the Platonic solid will be needed to make sure that this construction produces an embedded infinite polyhedral surface.

Antiprismatic polyhedra will be invariant under a discrete group $R_{p,q}^n$ of isometries that is generated by reflection-rotations at the faces of a Platonic solid $P_{p,q}$, where the  rotation is by the angle $\pi/p$ about an axis perpendicular through the face center. Clearly the Platonic solid itself is not suitable as a fundamental domain for this group action. A better candidate would be the truncation $TP_{p,q}$ of $P_{p,q}$, but attempts to tile with $TP_{p,q}$ usually leave gaps. These can sometimes be filled with Platonic solids, leading to the following:

Fundamental domains for $R_{p,q}^n$ are constructed by adding pyramids $PY_q†^n$ to the truncation $TP_{p,q}$ of the Platonic solid.

\begin{definition}
A generalized pyramid is a polyhedron with a distinguished $q$-gonal face, called the base, and $q$ possibly infinite triangles, called sides, that are adjacent in pairs and to the base.  We denote by  $PY_q^n$  a pyramid   such that
\begin{itemize}
\item the base is a regular $q$-gon;
\item $PY_q^n$ is rotational symmetric by a rotation of angle $2\pi/q$ about an axis through  the center of and perpendicular to the base;
\item the dihedral angle between adjacent sides is $2\pi/n$.
\end{itemize}
\end{definition}

In the simplest case, subdividing a Platonic solid $P_{p,q}$ will result in pyramids $PY_p^q$ in all spaces of constant curvature. We will construct more complicated examples in section \ref{sec:pyramids}.

The resulting polyhedron obtained by gluing  pyramids $PY_q^n$  to the $q$-gonal faces of a truncation $TP_{p,q}$ will be called an  $n$-akis $P_{p,q}$, denoted by $KP_{p,q}^n$.  In order for $P=KP_{p,q}^n$ to tile (and $R_{p,q}^n$ to be discrete), these polyhedra  need to satisfy an angle condition:

\begin{theorem}

Let $\alpha$ be the dihedral angle between two $2p$-gonal faces of a fixed truncated $TP_{p,q}$, $\beta$ the angle between a $q$-gonal and a $2p$-gonal face, and $\gamma$ the dihedral angle between the  base and a side face of the attached pyramid  $PY_q^n$ . If these angles satisfy
\[
\alpha +2\beta +2\gamma=2\pi  \ ,
\] 
then $P=KP_{p,q}^n$ is the fundamental domain of the group $R_{p,q}^n$ generated by the transformations that first reflect at the $2p$-gonal faces of $TP_{p,q}$ and then rotate by $\pi/p$ about an axis perpendicular to and through the center of the $2p$-gonal face.
\end{theorem}

\begin{proof}
This follows from the Poincaré polyhedron theorem. The transformations of $R_{p,q}^n$ contribute to two matching conditions: $n$ copies of $PY_q^n$ will fit around a common edge of two sides of $PY_q^n$. Secondly, after attaching two copies of  
$KP_{p,q}^n$ along a common $2p$-gon with a $\pi/p$ rotation and using angle $\alpha$), two of the attached pyramids will meet along a side face (using twice the angle $\beta$ and $\gamma$) of the pyramids.
\end{proof}

We refer the reader to section \ref{sec:tetra} for pictures and a more detailed description in the tetrahedral Euclidean case.

Note that for $n=2$, the pyramids $PY_q^2$ degenerate to $q$-gons, so that no pyramids are needed.

\subsection{Construction of the Pyramids $PY_{q}^n$}
\label{sec:pyramids}

\begin{theorem}
Pyramids $PY_{q}^n$ exist for any $q\ge3$ and $n\ge3$, but not necessarily in all spaces of constant curvature, and not necessarily of all sizes. 
\end{theorem}

Note that by the Poincaré polyhedron theorem, continued reflections at the sides of a pyramid $PY_{q}^n$ will create a possibly infinite polyhedron with $q$-gons as faces and valency $n$. The construction implies that this polyhedron will be Platonic, so serve as a model of a generalized $P_{q,n}$.

The rest of this  section is devoted to proving this theorem for $n=4$, i.e. for pyramids over a square. The proof for $n\ne4$ is very similar.

\begin{proof}

For $p=4$ and $n=3$, we  take any cube in $\bS^3$, $\bR^3$ or $\bH^3$ and subdivide it into pyramids that have the cube faces as a base and the center of the cube as an apex. Each of these pyramids is a pyramid of type $PY_{4}^3$.

As there is no finite $P_{n,4}$ for any $n\ge4$, there is no finite $PY_{4}^4$ either, and a $PY_{4}^4$ can only exist in Euclidean or hyperbolic space.

For $n=4$, the simplest case is $PY_{4}^4$ in Euclidean space. It simply is a half-column over a square.

 \begin{figure}[h] 
   \centering
   \includegraphics[width=3in]{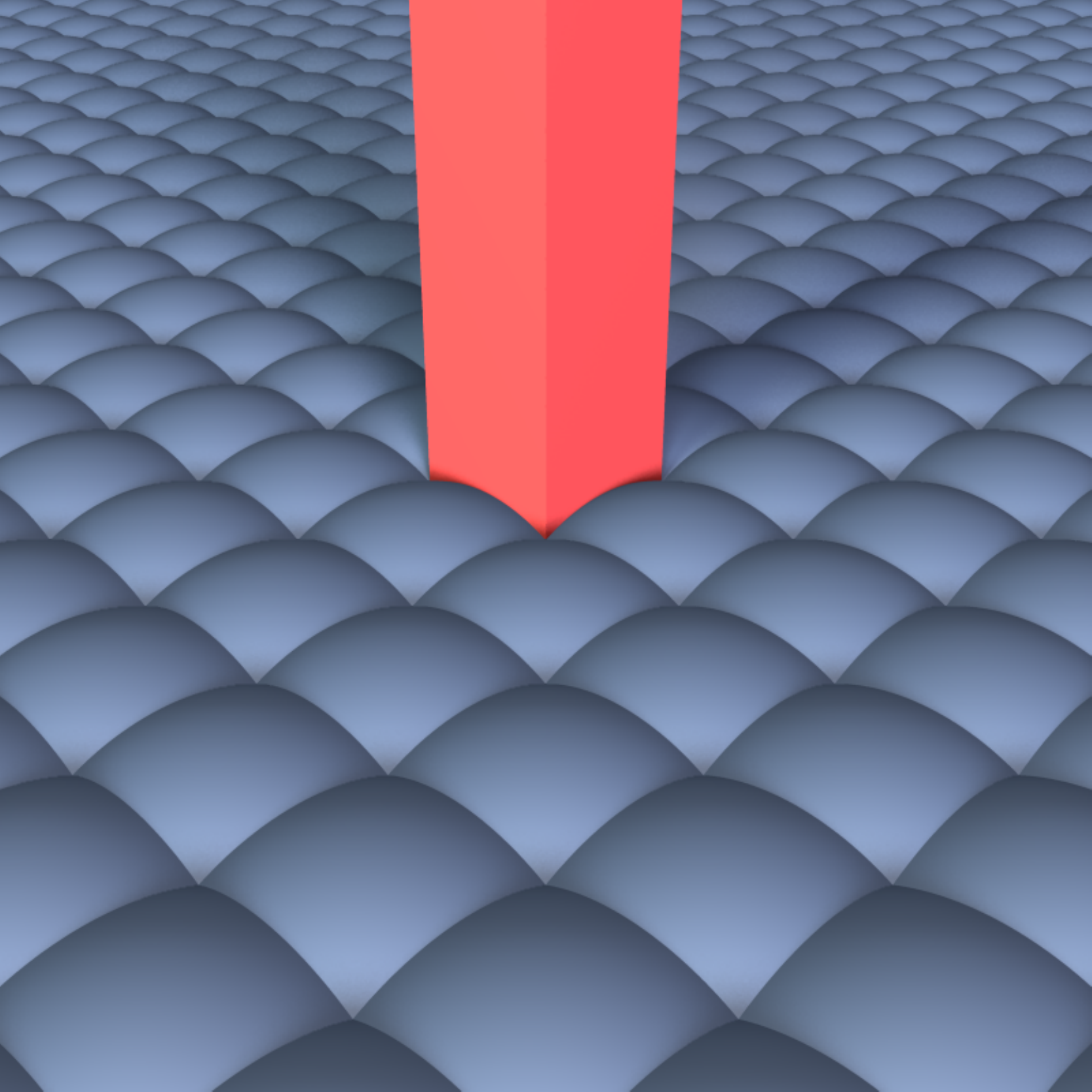}
   \caption{A doubly periodic hyperbolic cube $P_{4,4}$}
   \label{fig:hypcube4}
\end{figure}

For $n=4$ and in hyperbolic space, we will use the upper half space model for the construction of a pyramid $PY_4^4$ (see figure \ref{fig:hypcube4}). Place upper hemispheres with radius $r$ centered at the points $(i,j,0)$ for $i,j\in\Z$. These represent hyperbolic geodesic planes in $\bH^3$. For $r>\sqrt2/2$, four of them meet at points with coordinates $(i+1/2,j+1/2,\sqrt{r^2-1/2})$ so that the region above the spheres is an infinite doubly periodic hyperbolic solid polyhedron  $P_{4,4}$ bounded by hyperbolic squares with vertices in the horosphere $z=\sqrt{r^2-1/2}$. For large $r$, the squares become smaller, and the dihedral angle between them approaches $\pi$. For $r\to \sqrt2/2$ the squares become ideal with dihedral angle $\pi/2$.

We divide $P_{4,4}$ periodically into ideal hyperbolic pyramids $PY_{4}^4$ over each of the square faces, using vertical half strips as sides. Evidently, the vertical sides have dihedral angle $\pi/2$.

\medskip

We will now construct hyperbolic generalized Platonic polyhedra $P_{4,n}$ for $n\ge 5$ in several steps.

First, we construct infinite columns $PY_4^n$ with a square base and dihedral angle $2\pi/n$ in the ball model of hyperbolic space. To this end, we consider four spheres centered at points $(\pm r,0,0)$  and $(0,\pm r,0)$ and radius $\sqrt{r^2-1}$ for any $r>1$. 

\begin{lemma}\label{lem:rn}
These four spheres cut the unit ball in hyperbolic planes  that make a dihedral angle of $\arccos(1/(r^2-1))$. This dihedral angle becomes $2\pi/n$ when
\[
r=r_n=\sqrt{1+\sec(2\pi/n)} \ .
\]
\end{lemma}

Next we cut these infinite columns by using a horizontal hyperbolic plane, i.e. a sphere centered at $(0,0,-s)$ for $s>1$ with radius  $\sqrt{s^2-1}$. The intersection of this plane with the column will be a hyperbolic square (the base) and a face of  the generalized Platonic $P_{4,n}$ that we want to construct. 
This means that we need to ensure that the plane actually intersects the column.

 \begin{figure}[h] 
   \centering
   \includegraphics[width=2.5in]{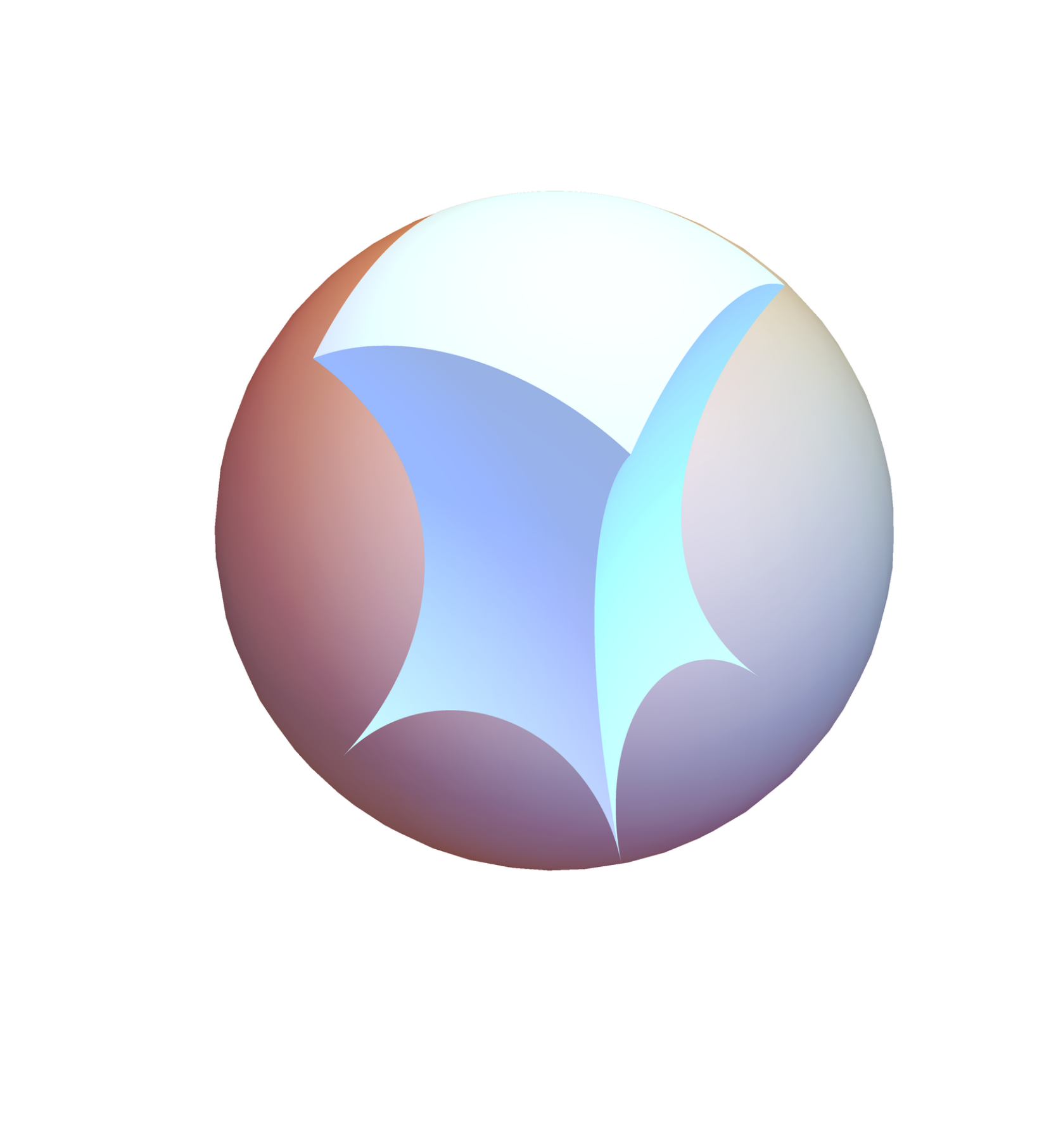}
   \quad
   \includegraphics[width=2.5in]{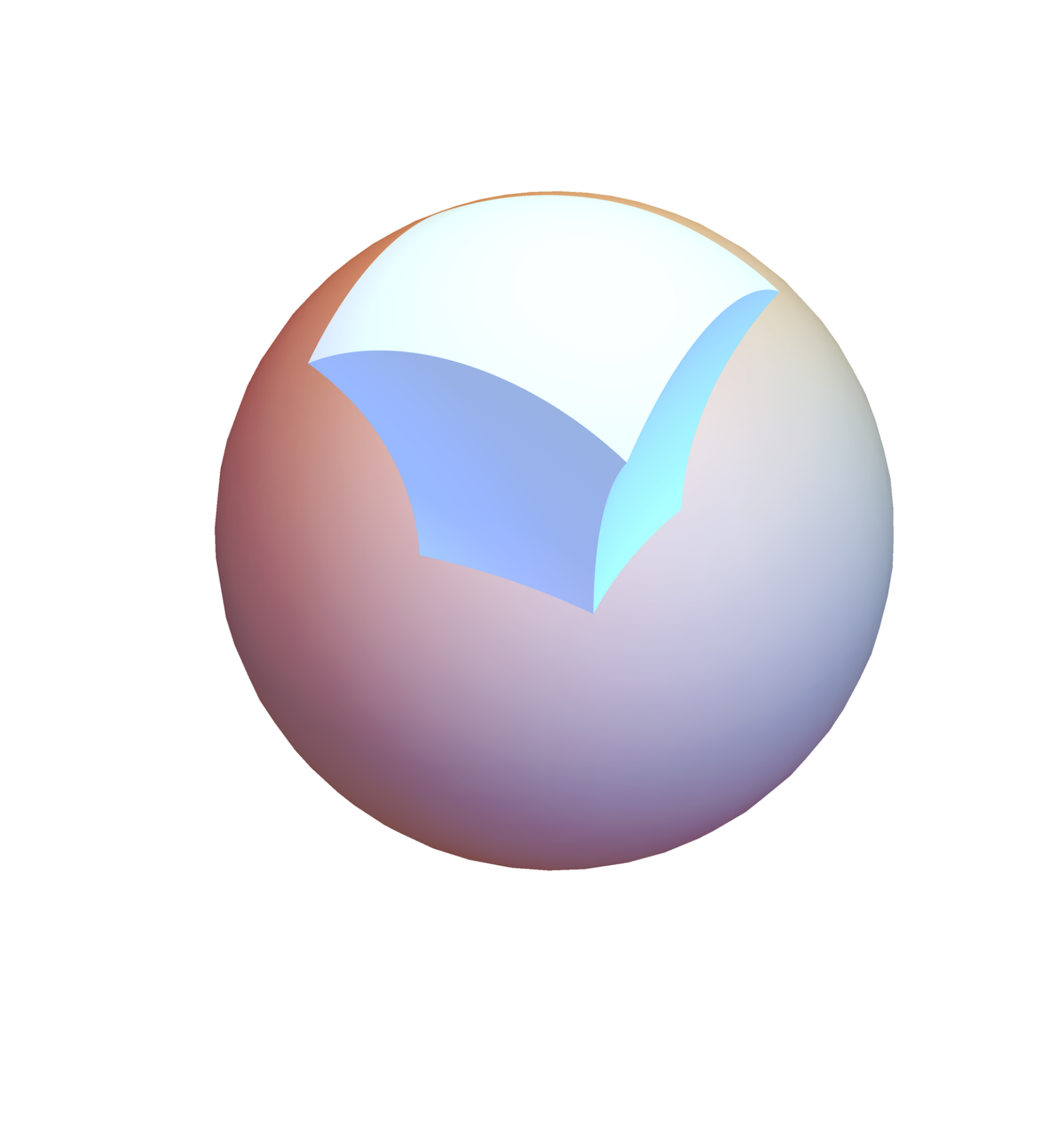}
      \caption{Half-columns for a $P_{4,5}$}
   \label{fig:halfcolumn45}
\end{figure}

In the extreme case, the plane will intersect the column in its ideal vertices, rendering the base square intersection  ideal as well. This happens for $s=-\cot(\pi/n)$. The dihedral angle $\gamma$ between the base square and a side of the column is given by 

\[
\cos \gamma=\frac {\sqrt{\cos(\pi/n)}}{\sqrt{s^2-1}}
\]

For $s\to\infty$, the base square will lie in the $xy$-plane, so that its interior angle will be $2\pi/n$ and the dihedral angle between base and face will be $\pi/2$. For $s=\cot(\pi/n)$, the dihedral angle between base and face becomes 
\[
\gamma = \pi/2-\pi/n \ .
\]

This can be seen without computation by switching to the upper half space model and placing one of the ideal vertices of the column at $\infty$. Then the faces of the column meeting at this ideal vertex are half strips over an isosceles Euclidean triangle in the $xy$-plane with angles $2\pi/n$, $\gamma$, $\gamma$.

Finally we repeatedly reflect this half column at its sides to obtain the (infinite) $P_{4,n}$. In contrast to the finite Platonic solids and the infinite $P_{4,4}$, the edge lengths of the square faces are bounded below by the edge length of a hyperbolic square with interior angle $\pi/n$ (but still can become arbitrarily large). This will significantly complicate the intermediate value argument in the proof of Theorem \ref{thm:anglecond}.

\end{proof}

For later use, we note:

\begin{lemma}\label{lem:gamma}
The dihedral angle $\gamma$ between base and one side of $PY_4^n$ is given by
\[
\cos\gamma = \frac{1}{\sqrt{r^2-1} \sqrt{s^2-1}} \ .
\]
\end{lemma}

\begin{lemma}\label{lem:edgepy}
The edge length $l$ of the base square of $PY_4^n$  is given by
\[
\cosh l = 1+\frac{2s^2}{\left(r^2-2\right) s^2-r^2} \ .
\]
\end{lemma}
\begin{proof}
A direct computation shows that the vertices of the base square are given by

\[
\frac{1}{\sqrt{\left(r^2-2\right) s^2-r^2}+r s}\left(\pm s, \pm s, r\right) \ .
\]

We now use the distance formula in hyperbolic space from section \ref{sec:formulas}. 
\end{proof}

In general, the pyramids $PY_{q}^n$ exist as follows:

\begin{itemize}
\item
In $\bS^3$ (and also $\bR^3$ and $\bH^3$) only  pyramids $PY_{q}^n$ that come from the Platonic solids by subdivision (i.e. for $(q,n)$ from the set $\{(3,3), (3,4), (3,5), (4,3), (5,3)\}$) exist . 
\item
For $(q,n)$ from the set $\{(3,6), (4,4), (6,3)\}$, $PY_{q}^n$ exists in $\bR^3$ and $\bH^3$ only.
\item
For all other values of $(q,n)$, $PY_{q}^n$ exists only in hyperbolic space.
\end{itemize}

The details of the construction for all such pyramids are very similar to what we did earlier in this section.

\subsection{Geography of Antiprismatic Polyhedra}

The following table gives the values for the angle condition for Euclidean truncated Platonic solids for the cases $n=2$ and $n=3$ The angles $\alpha$ and $\beta$ are the two dihedral angles of the truncated solid, while $\gamma_n$ is the dihedral angle between base and side of a pyramid that has dihedral side angle $2\pi/n$. For $=2$ this means $\gamma_2=0$, and no pyramid will be attached. 

 This allows us to determine for which values of $n$ and in which spaces we can expect a tiling with $n$-akis polyhedra $KP_{p,q}^n$. For instance, for the cube the angle sum is less than $360^\circ$ for $n=2$ in Euclidean space. As in $\bS^3$ the angles will become larger, there is hope that the case $n=2$ can be realized in $\bS^3$. On the other hand, for $n=3$ the angle sum is larger than $360^\circ$, and it will become even larger for $n>3$. Therefore we need hyperbolic space to make the angle sum smaller for all $n\ge 3$.

 This information is given in the last three columns of the table.

\begin{table}[h]
   \centering
   \begin{tabular}{@{} ll | ll | lll @{}} 
      \toprule
          &   & Angle Sums in   &Euclidean Space  & Values of & $n$ possible  & in \\
        symbol   & name  &$\alpha+2\beta+2\gamma_2$  &$\alpha+2\beta+2\gamma_3$ & $\bS^3$ & $\bR^3$ & $\bH^3$ \\
      \midrule
      $\{3,3\} $     & tetrahedron & $289.47^\circ$ & $360^\circ$ & 2 & 3 & $\ge 4$ \\
      $ \{4,3\}$      & cube     &  $340.53^\circ$ & $411.06^\circ$ & 2 & - & $\ge3$\\
      $ \{3,4\}$       & octahedron  & $360^\circ$ & $450^\circ$ & - & 2 & $\ge3$\\
      $\{5,3\} $       & dodecahedron  & $401.81^\circ$ & $472.34^\circ$ & - & - & $\ge2$ \\
      $\{3,5\} $       & icosahedron  & $423.44^\circ$ & $540^\circ$ & - & - & $\ge2$ \\
      \bottomrule
   \end{tabular}
   \label{tab:plato2}
\end{table}

\section{The Antiprismatic Octahedron}

We will use antiprismatic octahedra to give full details of the existence proof of antiprismatic polyhedra.  The other cases will be very similar.

\subsection{The Euclidean Antiprismatic Octahedron $AP_{3,4}^2$}

We begin by constructing a fundamental domain for a group of reflection-rotations $R_{3,4}^2$ at the faces of a Euclidean octahedron:

The truncated octahedron tiles space by reflecting it across its hexagonal faces with a $60^\circ$ rotational twist. This tiling is called the {\em bitruncated cubic honeycomb} or {\em truncoctahedrille}, see \cite{cbg} page 295,  and figure \ref{fig:TruncatedOctahedronTiling}.

In this tiling, three truncated octahedra fit around an edge. For one of these three truncated octahedra two of its hexagonal faces meet at that edge, for the two others a hexagonal and a square face meet. This means that the angle condition $\alpha+2\beta$ is satisfied without the use of pyramids ($n=2$). 

\begin{figure}[h] 
   \centering
   \includegraphics[width=3.0in]{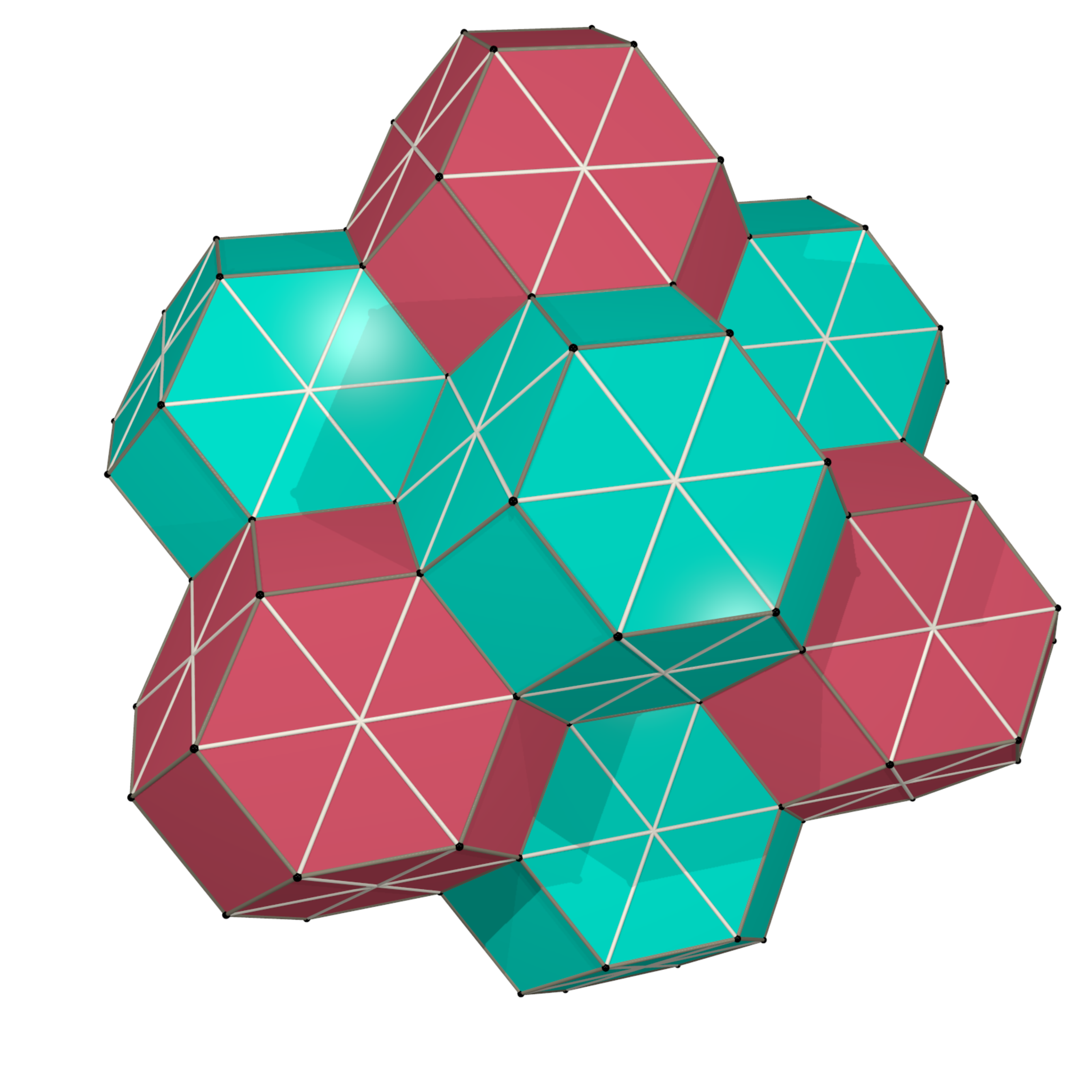}
   \caption{The bitruncated cubic honeycomb}
   \label{fig:TruncatedOctahedronTiling}
\end{figure}

We now center an octahedron inside one of the truncated octahedra and attach regular antiprisms to its faces.
The size of the octahedron is chosen so that the hexagonal faces of the truncated octahedron cut the antiprisms in half. This way the group $R_{3,4}^2$ generated by the rotation-reflections at the hexagonal faces of the truncated octahedron extend the prisms to a Euclidean antiprismatic octahedron $AP_{3,4}^2$, see figure \ref{fig:apoct3}.

 \begin{figure}[h] 
   \centering
   \includegraphics[width=2.5in]{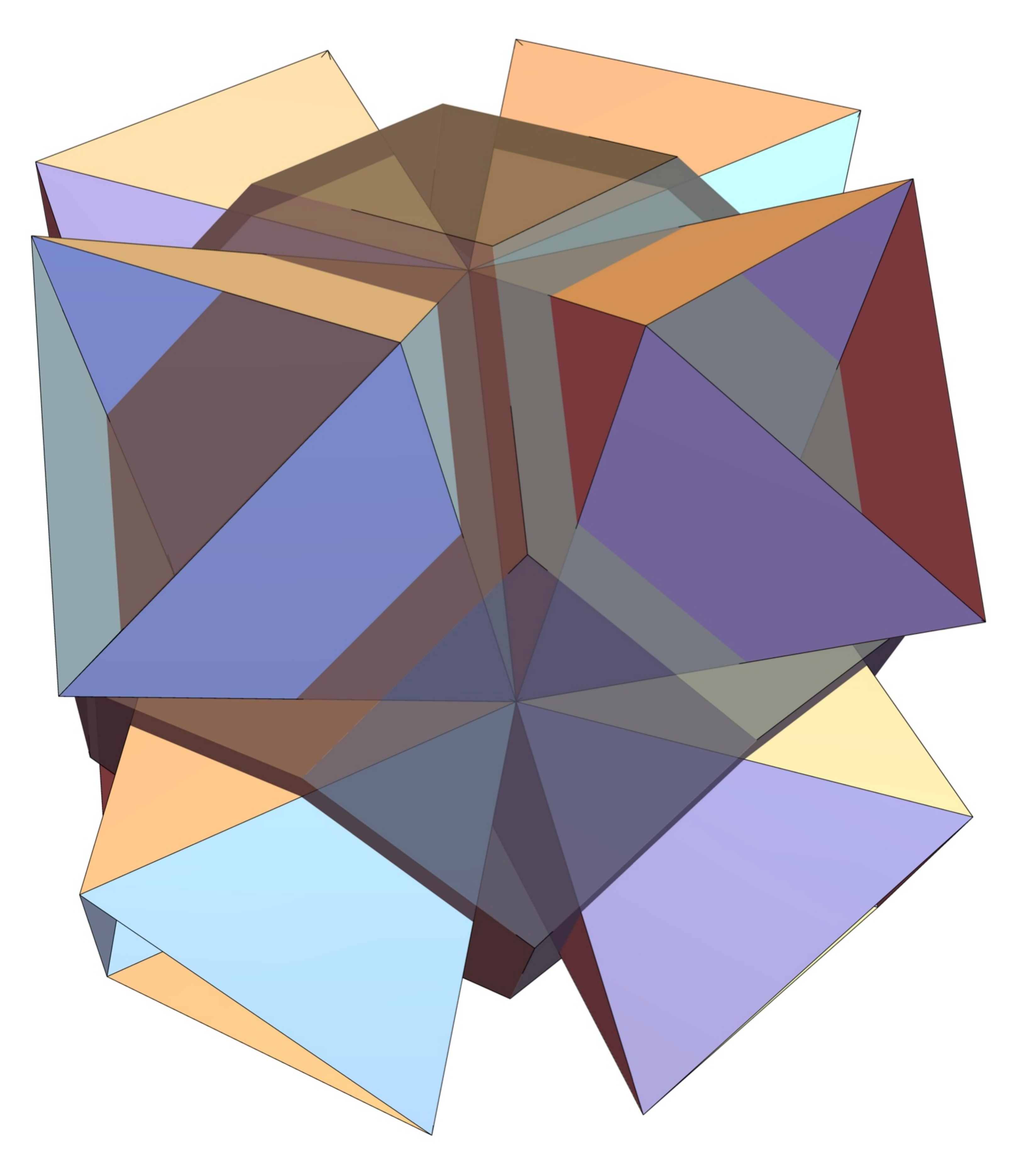}
   \qquad
   \includegraphics[width=2.5in]{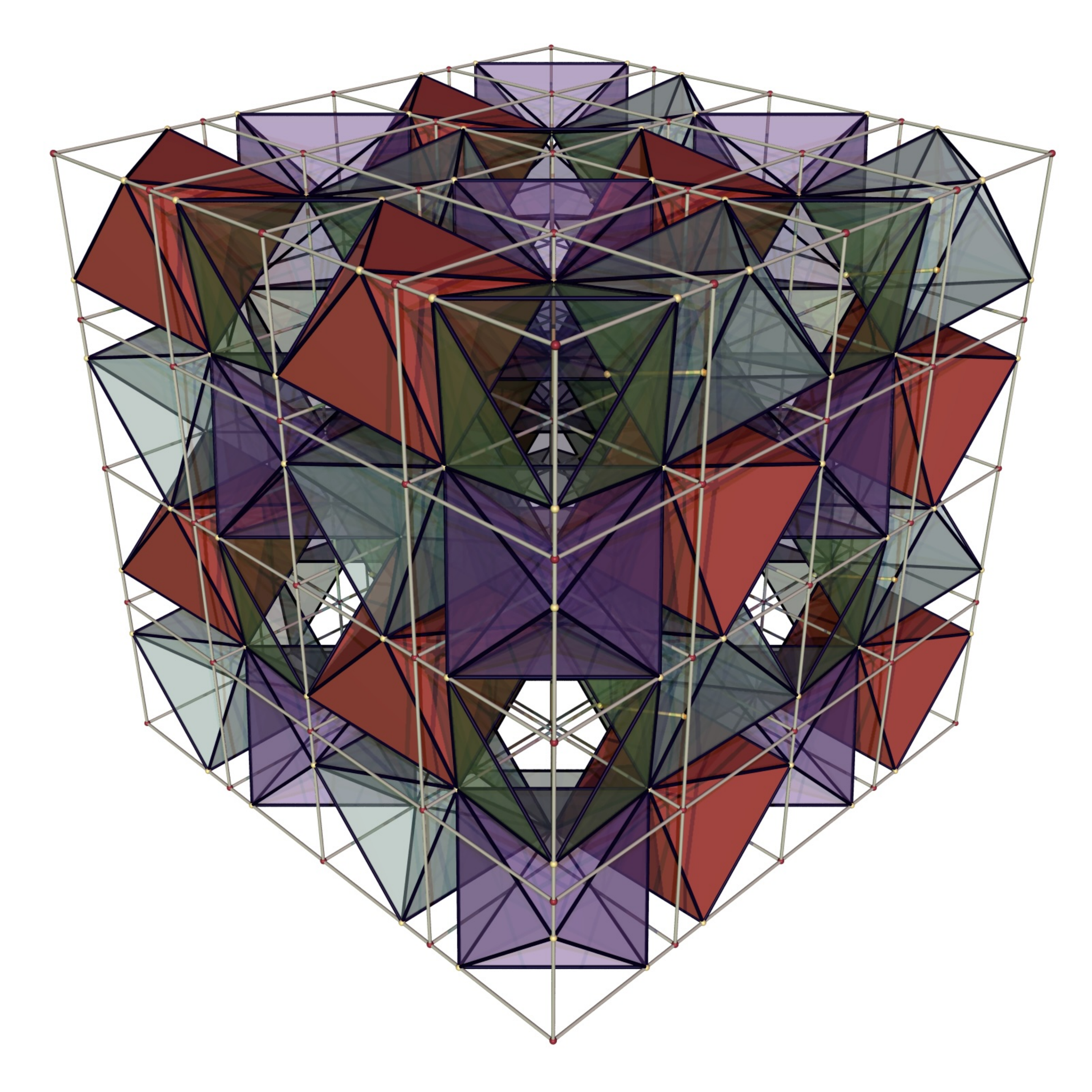}
   \caption{The Euclidean antiprismatic octahedron $AP_{3,4}^2$}
   \label{fig:apoct3}
\end{figure}

The quotient surface under translational symmetries has genus 4 and is quite interesting. It has a conformal hidden symmetry of order 12 which can be used to represent the surface as a 12-fold cyclically branched cover over the thrice punctured sphere and serves as a conformally correct simplicial approximation of Alan Schoen's I-WP surface, see \cite{lwy} for details.

\subsection{Hyperbolic Antiprismatic Octahedra}\label{sec:octa}

In the antiprismatic hyperbolic case, we need more complicated fundamental domains than truncated octahedra, because 
hyperbolic truncated octahedra don't tile hyperbolic space, regardless of the chosen dihedral angles. 

If we arrange three of them around an edge as in Euclidean space, there will be a gap that sometimes can be filled by a generalized and possibly infinite Platonic solid with square faces.

In the simplest case, these filling Platonic solids will be hyperbolic cubes.

We begin with a hyperbolic truncated octahedron of edge length $l$ and attach pyramids to its square faces that have dihedral angle $2\pi/3$ at their triangular sides, i.e. that can be obtained from a regular hyperbolic  cube by dividing it into six pyramids that have their apex at the center of the cube.

This hyperbolic {\em tetrakis truncated octahedron} $KP_{3,4}^3$ comes in a 1-parameter family with the edge length $l$ of the truncated octahedron varying between 0 and $\infty$. We need to find the correct parameter so that $KP_{3,4}^3$ can be used to tile hyperbolic space.

\begin{theorem}\label{thm:anglecond}
There is a value of $l$ so that the  tetrakis truncated octahedron $KP_{3,4}^3(l)$ satisfies the angle condition
\[
\alpha +2\beta +2\gamma=2\pi  \ .
\] 
We will denote this polyhedron simply by $KP_{3,4}^3$.
\end{theorem}
\begin{proof}
For $l\to 0$ we approach the Euclidean case. For $l>0$ all dihedral angles are smaller than their Euclidean counterparts and $\alpha +2\beta +2\gamma\to 2\pi+\pi/2>2\pi$.
For $l\to\infty$ both the truncated octahedron and the cube become ideal, and we can compute their dihedral angles explicitly as 
\[
\alpha=\arccos(2/3), \beta=\arccos(1/\sqrt6), \gamma= \arccos(1/2)
\]
so that $\alpha +2\beta +2\gamma=4\pi/3<2\pi$. By the intermediate value theorem, there is a value of $l$ so that $\alpha +2\beta +2\gamma=2\pi$.
\end{proof}

 \begin{figure}[h] 
   \centering
   \includegraphics[width=2in]{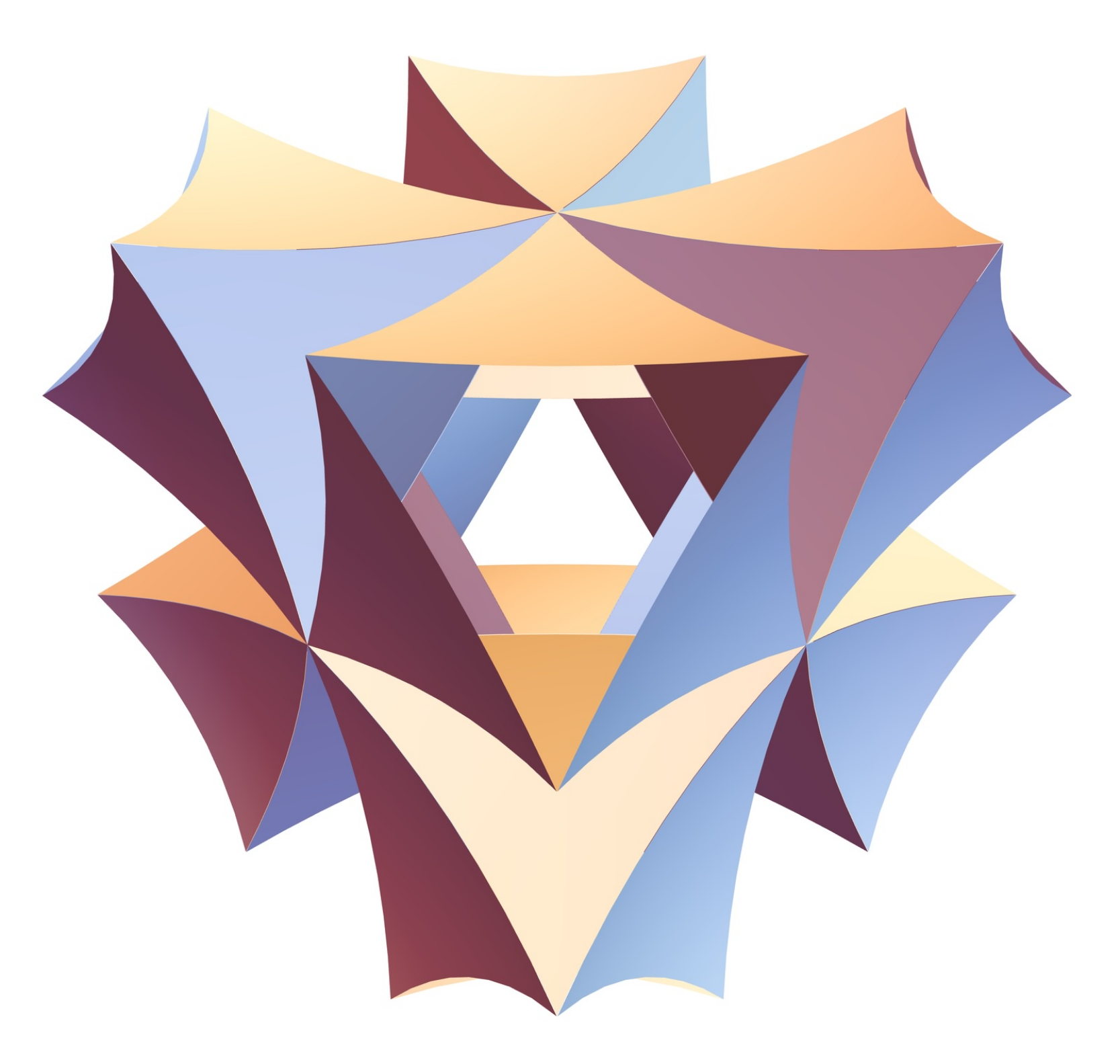}
   \quad
   \includegraphics[width=2in]{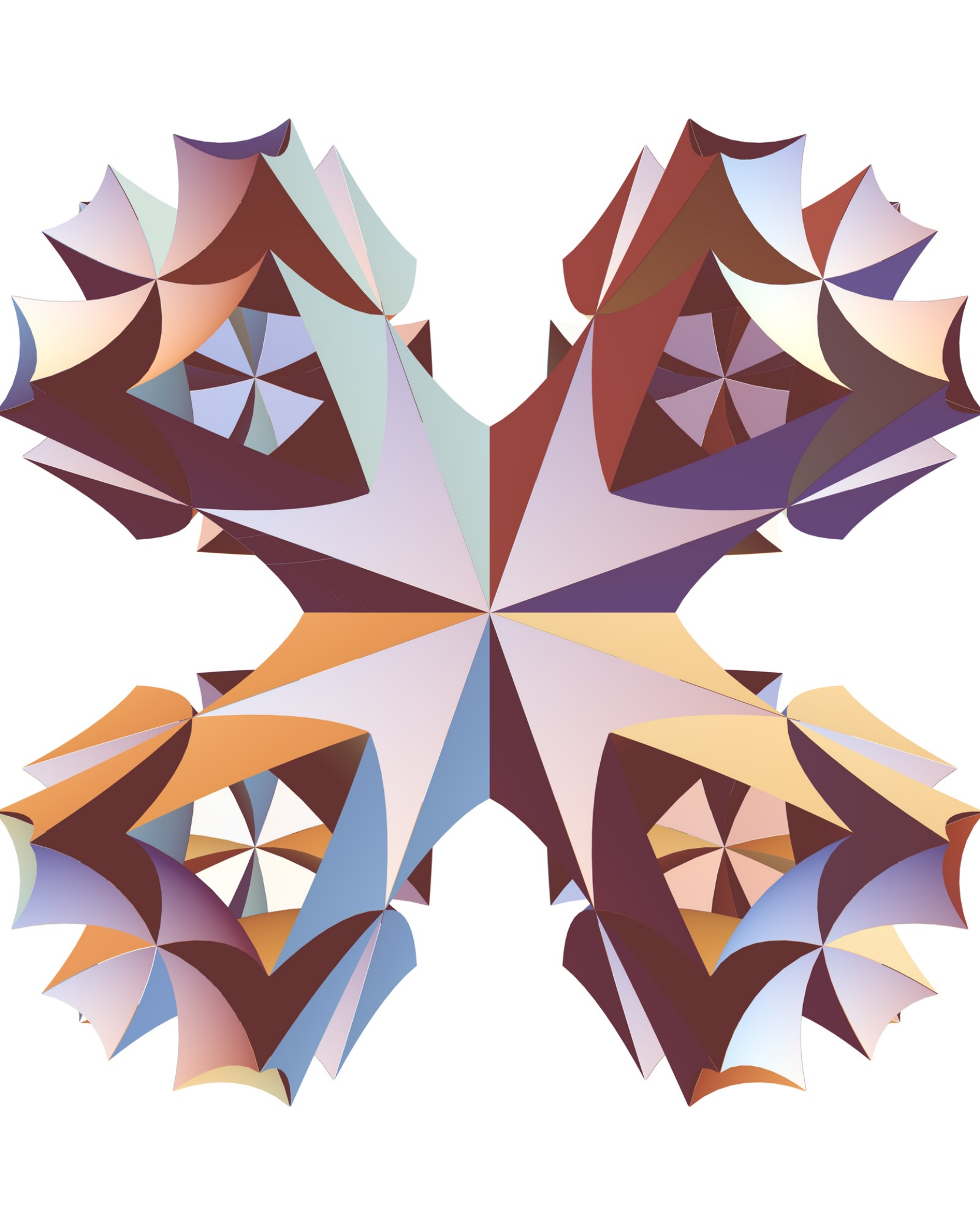}
   \quad
   \includegraphics[width=2in]{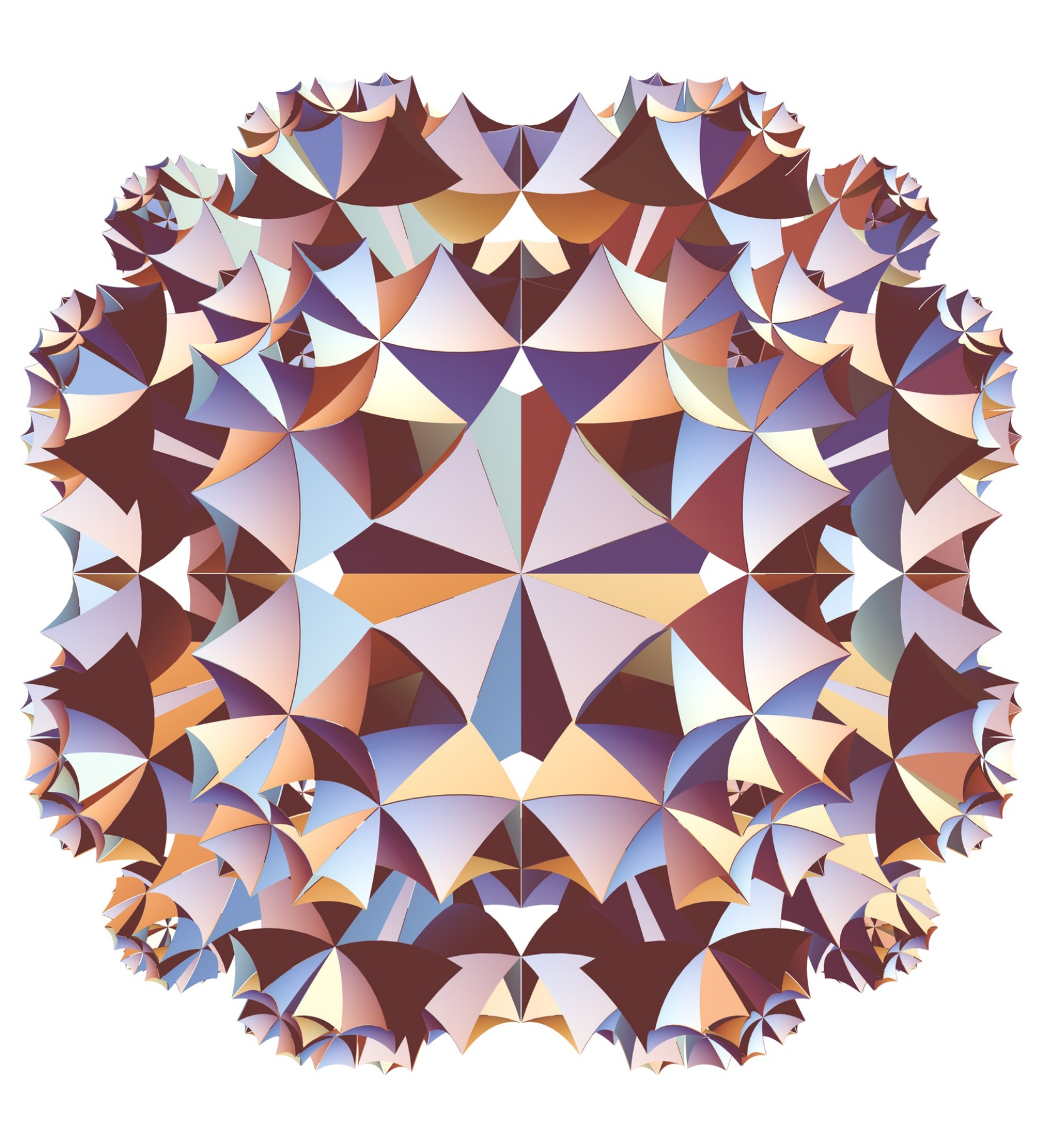}
   \caption{Antiprismatic Octahedra for $n=3$}
   \label{fig:AntiprismaticOcta3}
\end{figure}

 By the Poincaré polyhedron theorem $KP_{3,4}^3$ is a fundamental domain for the group $R_{3,4}^3$ generated by rotation-reflections at its hexagonal faces.

 \subsection{$KP_{3,4}^n$ for $n\ge 5$}

Our goal is to complete the proof of the following theorem for $n\ge 5$

\begin{theorem}
For every $n\ge3$ there exists a $n$-akis truncated octahedron $KP_{3,4}^n$ that tiles hyperbolic space.
\end{theorem}
\begin{proof}

We will again use  an intermediate value argument. In the  first extreme case we have  $a=1$ and the truncated octahedron $TP_{3,4}(a)$ becomes ideal. To evaluate the angle sum $\alpha+2\beta$, we place one of its vertices at $\infty$ in the upper half space model of hyperbolic space. The faces of $TP_{3,4}(1)$ that meet at this vertex become infinite half-strips over a triangle in the $xy$-plane with angles $\alpha$, $\beta$, $\beta$ so that for $a=1$ we have $\alpha+2\beta=\pi$. 
 As noted in the proof of Lemma \ref{lem:rn}, for an ideal $PY_4^n$ we have  $2\gamma=\pi-2\pi/n$, so that 
 \[
 \alpha+2\beta+2\gamma = 2\pi -2\pi/n<2\pi \ .
 \]
 
 The other extreme case occurs when the edge length of the pyramid becomes minimal. The explicit evaluation of $\alpha+2\beta$ in this case is difficult, but all we will need is that $\alpha+2\beta>\pi$. This follows because the dihedral angles of any finite  $TP_{3,4}(a)$ will be smaller than  that of the ideal  $TP_{3,4}(1)$. Again, as noted in the proof of Lemma \ref{lem:rn}, in this extreme case we have $\gamma=\pi/2$ so that 
  \[
 \alpha+2\beta+2\gamma > 2\pi \ .
 \]
 
 By the intermediate value theorem, for any $n\ge5$ there is a value of $a$ such that the angle condition is satisfied.
  \end{proof}

\begin{theorem}
For $n\ge4$, there is an antiprismatic octahedron $AP_{3,4}^n$ in hyperbolic space.
\end{theorem}
\begin{proof}
To  construct a closed hyperbolic antiprismatic octahedron $AP_{3,4}^n$, we need to find an octahedron  whose attached antiprisms are symmetric  with respect to rotation-reflections at the hexagonal faces of $KP_{3,4}^n$. 

This is done by once more applying the  intermediate value theorem. We begin with a very small $P_{3,4}$ centered inside $KP_{3,4}^n$. Its edge length will be much smaller than the distance to $KP_{3,4}^n$. We now  increase its size. The distance to $KP_{3,4}^n$ will either become very small or its edge length very large, so that  in either case its edge length will become larger than its distance to $KP_{3,4}^n$. By the intermediate value theorem, there is a finite $P_{3,4}$ whose attached antiprisms are symmetric  with respect to rotation-reflections at the hexagonal faces of $KP_{3,4}^n$.

 \begin{figure}[h] 
   \centering
   \includegraphics[width=2in]{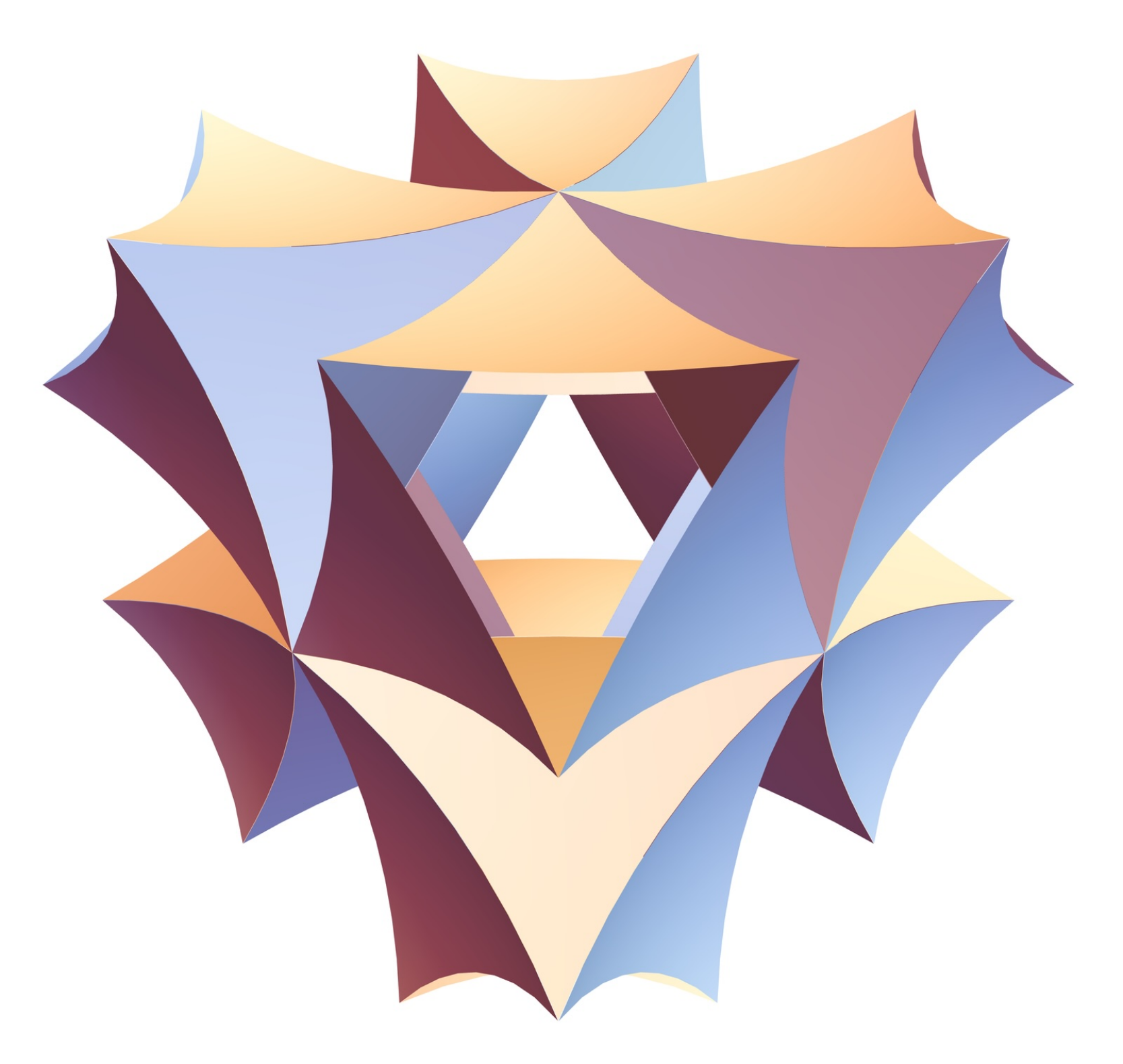}
   \quad
   \includegraphics[width=2in]{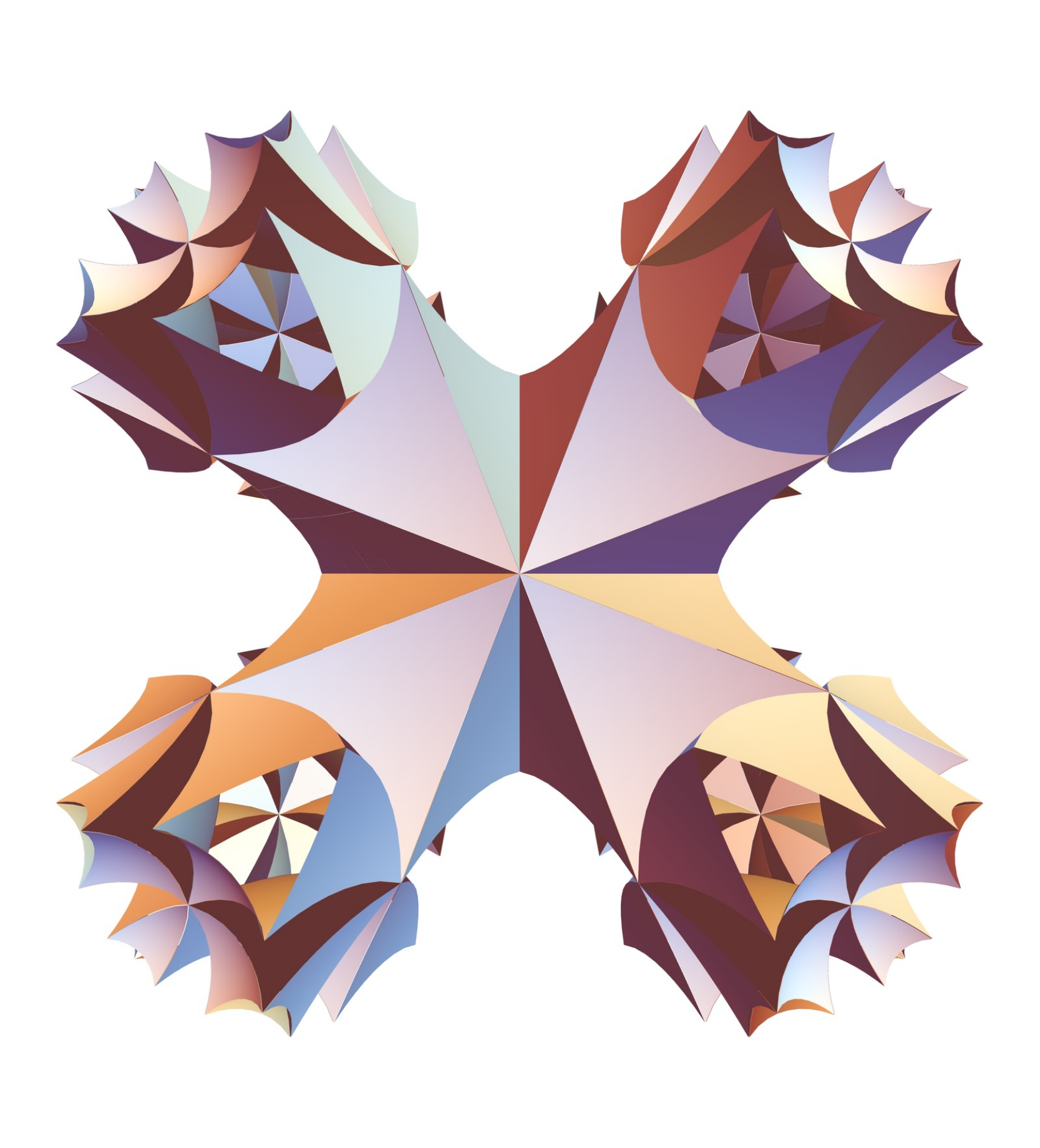}
   \quad
   \includegraphics[width=2in]{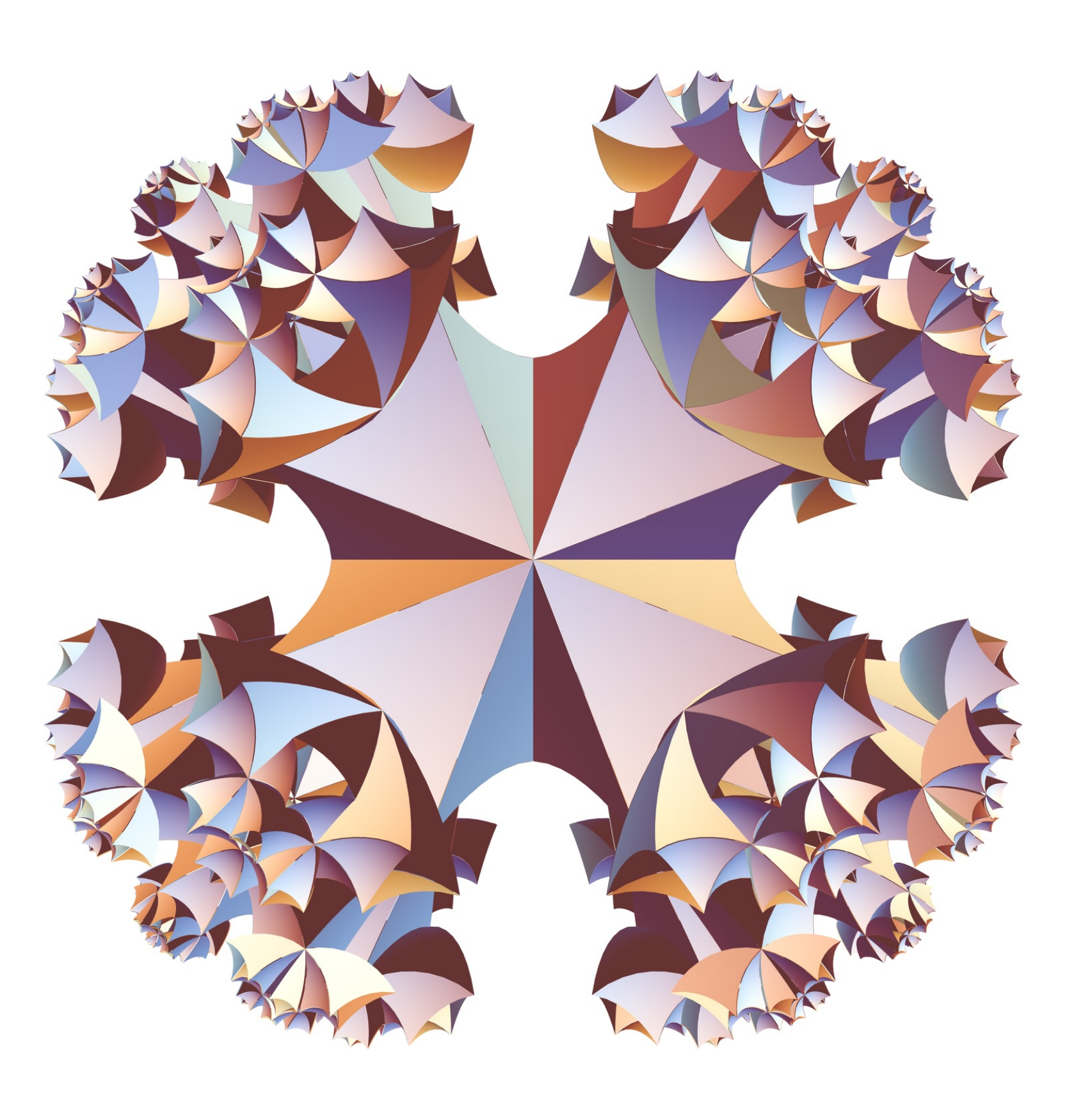}
   \caption{Antiprismatic Octahedra for $n=4$}
   \label{fig:AntiprismaticOcta4}
\end{figure}

\end{proof}

\newpage

\section{Antiprismatic Tetrahedra}
\label{sec:tetra}

Antiprismatic tetrahedra $AP_{3,3}^n$ exist in $\bS^3$ for $n=2$, in $\bR^3$ for $n=3$, and in $\bH^3$ for $n\ge4$.

\subsection{The Antiprismatic Tetrahedron $AP_{3,3}^3$ in $\bR^3$}

In this section, the basic building block is a tetrahedron. To its faces we attach regular antiprisms over a triangle, to them we attach tetrahedra, etc., see figure \ref{fig:antipristet}.

\begin{figure}[h] 
   \centering
   \includegraphics[width=2in]{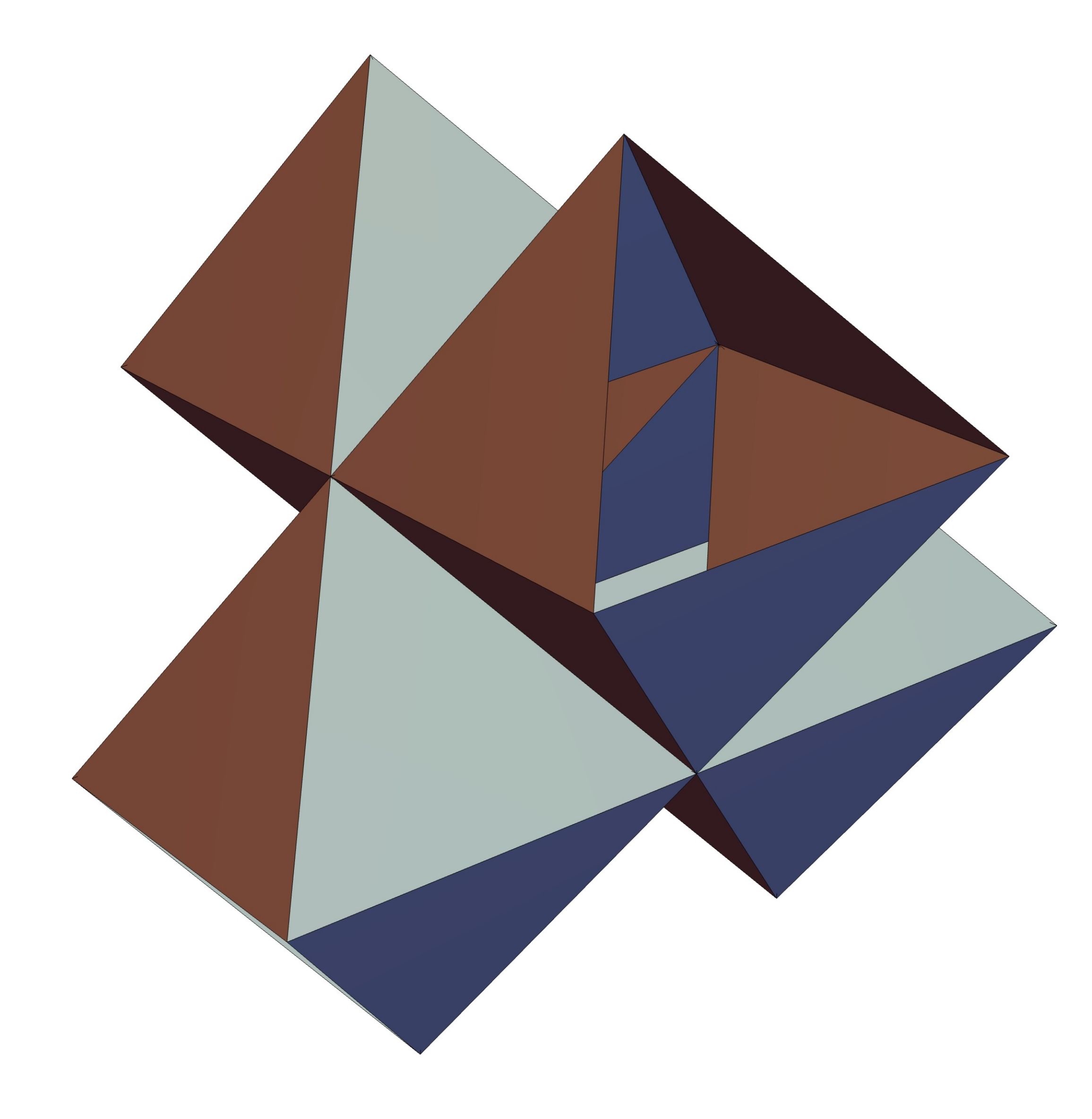}
   \quad 
   \includegraphics[width=2in]{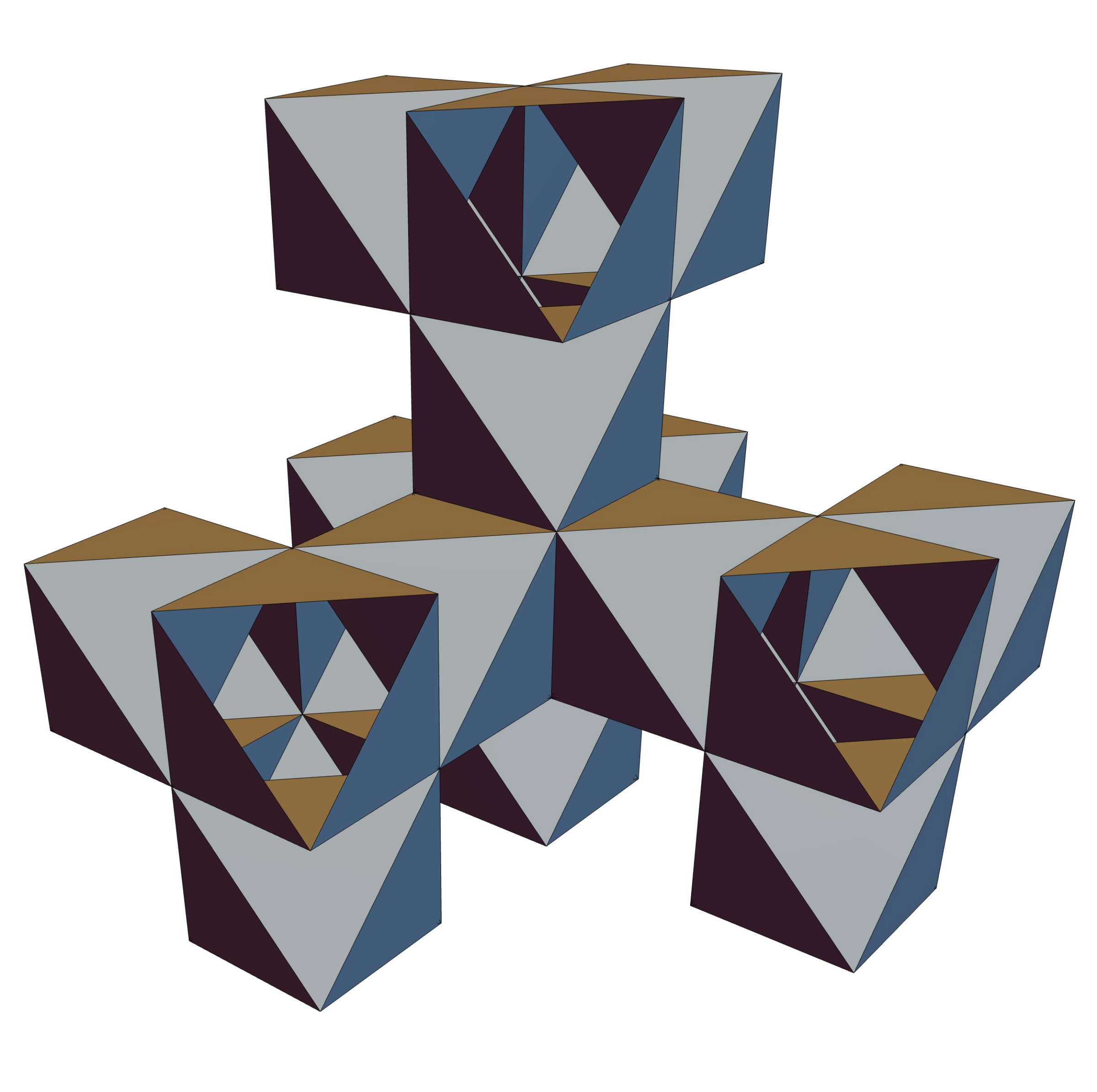} 
   \quad 
   \includegraphics[width=2in]{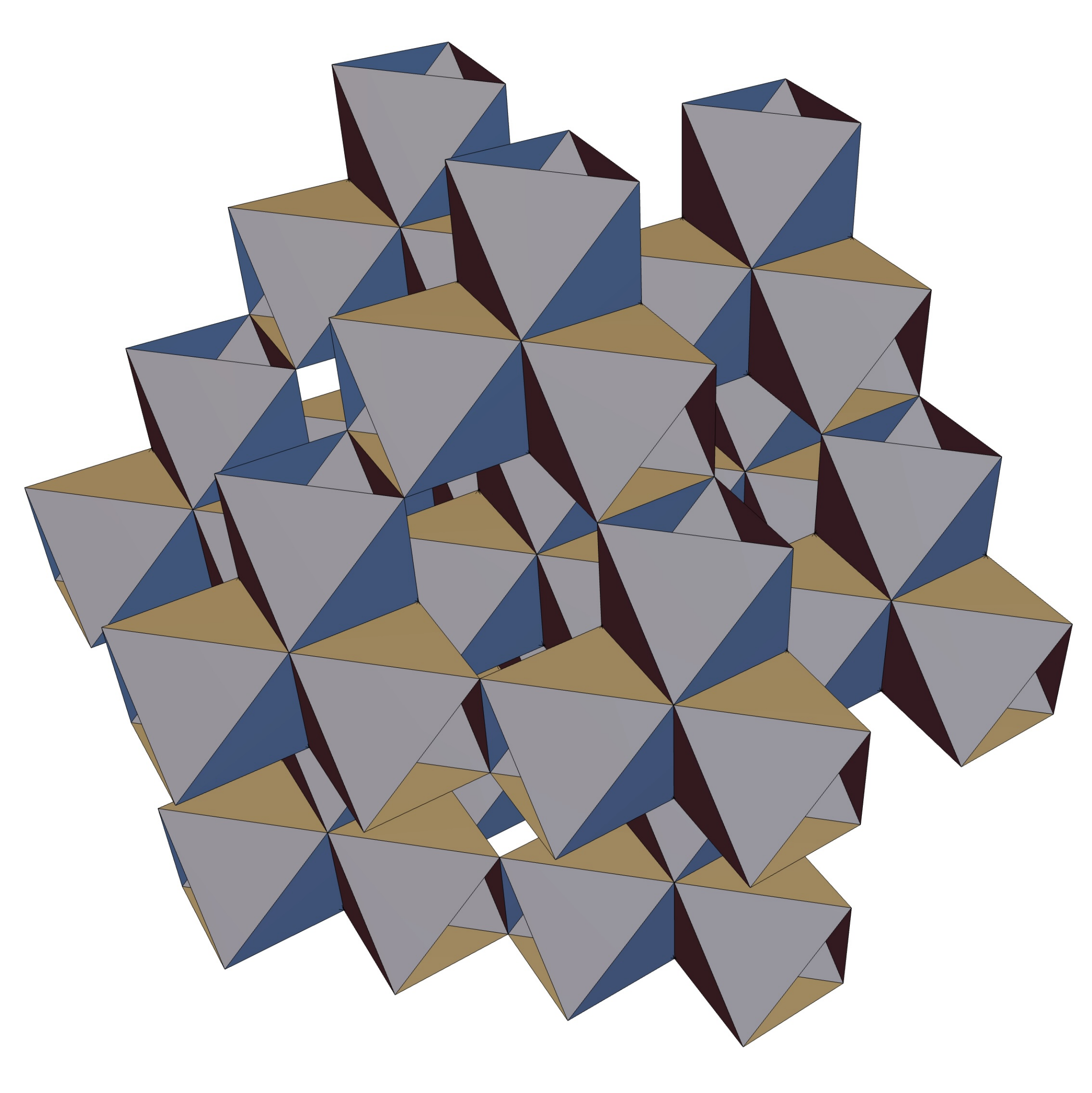} 
   \caption{The first three generations of the Euclidean antiprismatic tetrahedron $AP_{3,3}^3$}
   \label{fig:antipristet}
\end{figure}

The resulting surface is a periodic Euclidean polyhedron tiled with equilateral triangles so that $9$ triangles meet at each vertex. A fundamental domain under the translational symmetries consist of the first generation surface where parallel boundary edges are identified in pairs. This surface is not Platonically tiled: It has closed Petri-polygons of length 6 (the waists of the prisms) and 8 (across the prisms, following the identifications).



\subsection{The Quarter Cubic Honeycomb}
\label{sec:qch}

In this section, we will describe the fundamental domain $KP_{3,3}^3$ for the  group $R_{3,3}^3$ of reflection-rotations that leave the Euclidean antiprismatic tetrahedron invariant. It is the simplest case where one can visualize the angle condition.

For the Euclidean truncated tetrahedron $TP_{3,3}$ two hexagons make a dihedral angle of $\alpha=\arccos (1/3)$ (that of a regular tetrahedron), while a hexagon and triangle have dihedral
angle $\beta=\pi-\arccos (1/3)$. Therefore, $KP^3_{3,3}$ satisfies the angle condition, because $2\gamma_3=\alpha$.


Filling the gaps between the truncated tetrahedra  leads to a tiling of space called the {\em quarter cubic honeycomb}, see figure \ref{fig:qucuhoney}.

\begin{figure}[h] 
   \centering
   \includegraphics[width=1.8in]{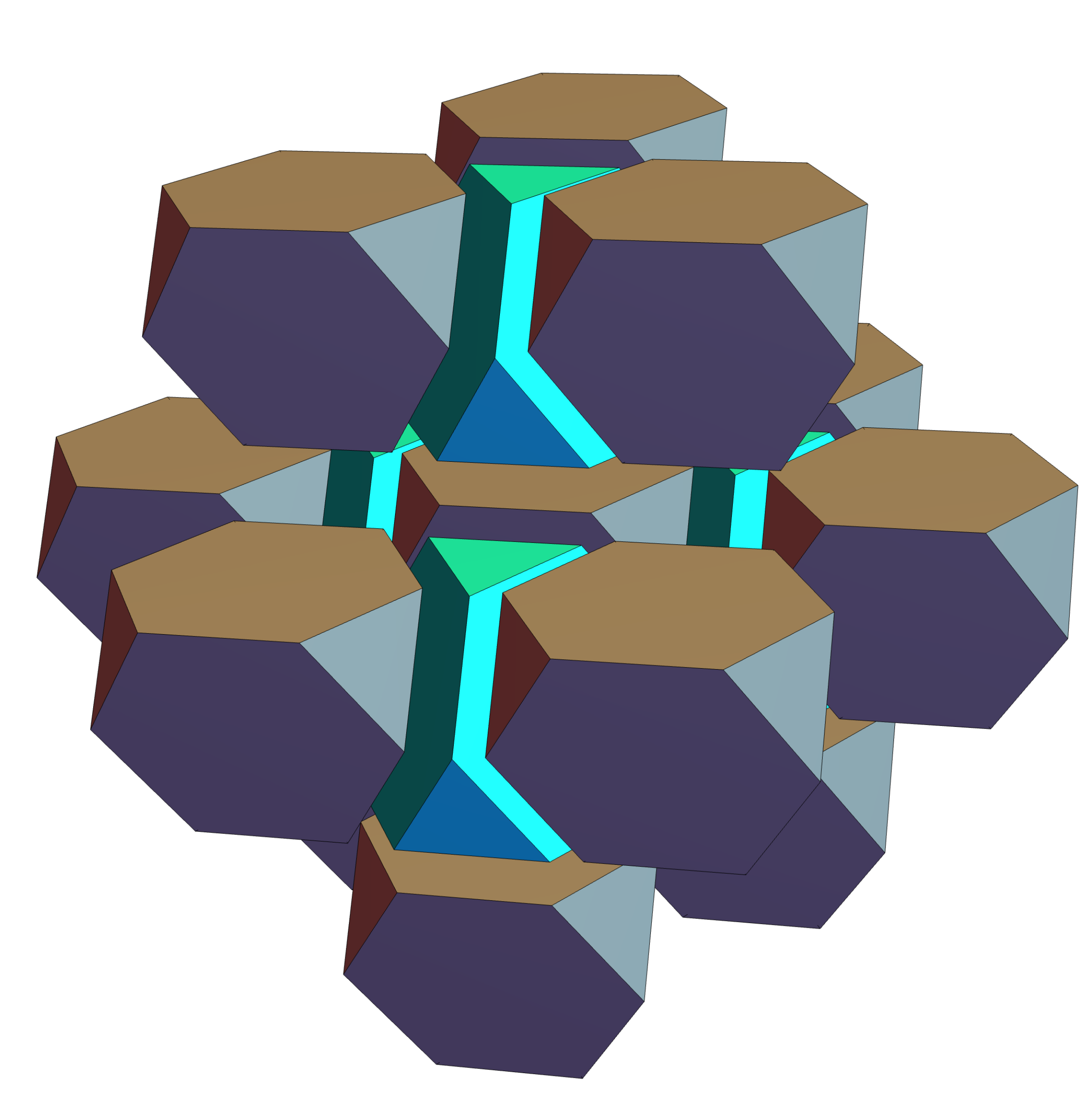}  \qquad
      \includegraphics[width=1.8in]{figures/antiprismatictetrahedron/TriakisTruncatedTetrahedron.pdf}\qquad \includegraphics[width=1.8in]{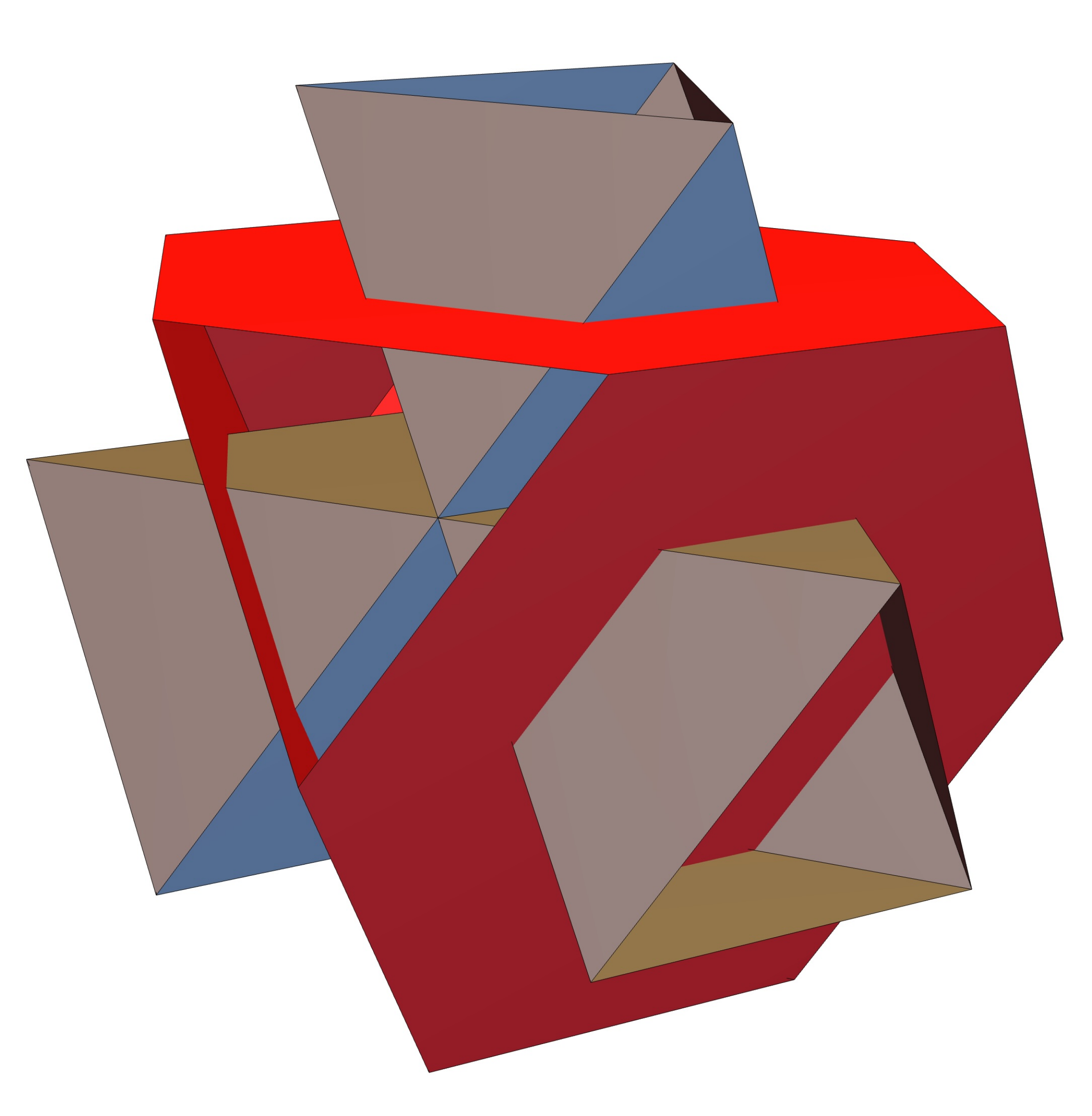}
   \caption{Quarter Cubic Honeycomb and  the Triakis Truncated Tetrahedron}
   \label{fig:qucuhoney}
\end{figure}

The group $R_{3,3}^3$ generated by the four reflection-rotations at the hexagonal faces of the truncated tetrahedron has as a fundamental domain the {\em Triakis Truncated Tetrahedron} $KP_{3,3}^3$.

We can now fit smaller tetrahedra into $KP_{3,3}^3$ so that antiprisms attached to its faces are symmetric with respect to the reflection-rotations at the the hexagonal faces of $TP_{3,3}^3$.

\subsection{The Hyperbolic Triakis Truncated Tetrahedra}

To show the existence of hyperbolic antiprismatic tetrahedra $AP_{3,3}^n$, we will use a hyperbolic version $KP_{3,3}^n$ of the  Euclidean Triakis Truncated Tetrahedron $KP_{3,3}^3$. In $KP_{3,3}^3$, four tetrahedral Triakis pyramids assemble to a regular tetrahedron so that three of them meet along a common edge from the apex to the base of the Triakis pyramids. In hyperbolic space, there will be $n\ge 4$ pyramids meeting along a common edge. 

For $n=4$, we construct a Triakis Truncated Tetrahedron $KP_{3,3}^4$ by taking a hyperbolic truncated tetrahedron and a regular octahedron of the same (hyperbolic) edge lengths. We divide the octahedron into eight pyramids which have as base  the triangular faces of the octahedron and whose apex is the center of the octahedron. We attach four such pyramids to the faces of the triangular faces of the truncated tetrahedron. We call the resulting polyhedron a hyperbolic {\em Triakis Truncated Tetrahedron} $KP_{3,3}^4$.

\begin{figure}[h]
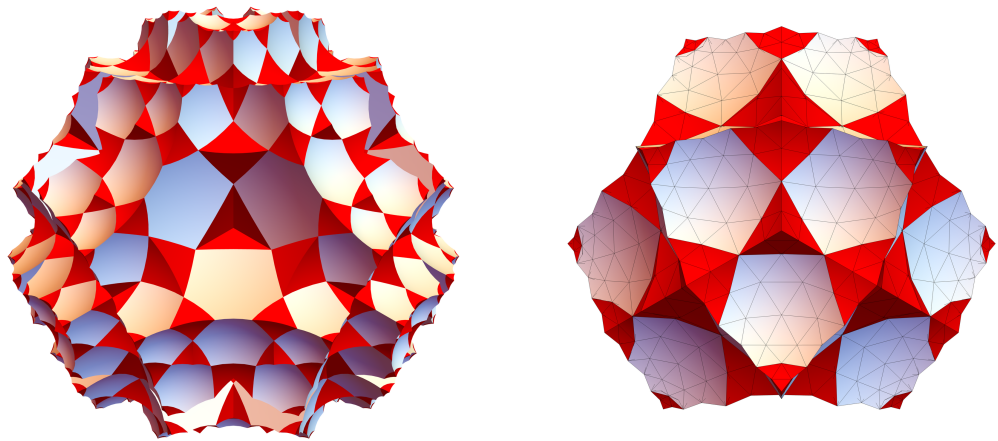
 
   \centering
\includegraphics[width=2.5in]{figures/antiprismatictetrahedron/tetratwist4solid-d.pdf}
\qquad
   \includegraphics[width=2.5in]{figures/antiprismatictetrahedron/tetratwist5solid.pdf}
   \caption{Tiling hyperbolic space with Triakis Truncated Tetrahedra $KP_{3,3}^4$ and }
   \label{fig:tritrutettile}
\end{figure}

In order that  a general $KP_{3,3}^n$ tiles hyperbolic space (see the left figure \ref{fig:tritrutettile}), we again need to satisfy the angle condition $\alpha+2\beta+2\gamma=2\pi$.

For $n=5$, we use an icosahedron instead of an octahedron to construct the pyramids, see the right figure of \ref{fig:tritrutettile}.

To construct Triakis Truncated Tetrahedra $KP_{3,3}^n$ for $n\ge6$  we need Platonic solids $P_{3,n}$ with triangular faces and vertex valency $n$. While there are no such solids that are compact, we can use non-compact versions.  These can be constructed in complete analogy to the solids $P_{4,n}$ for $n\ge 5$ in section \ref{sec:pyramids}.

\begin{theorem}
For any $n\ge4$ there exists a hyperbolic Triakis Truncated Tetrahedron $KP_{3,3}^n$ satisfying the angle condition.
\end{theorem}

\begin{figure}[h] 
   \centering
   \includegraphics[width=2.5in]{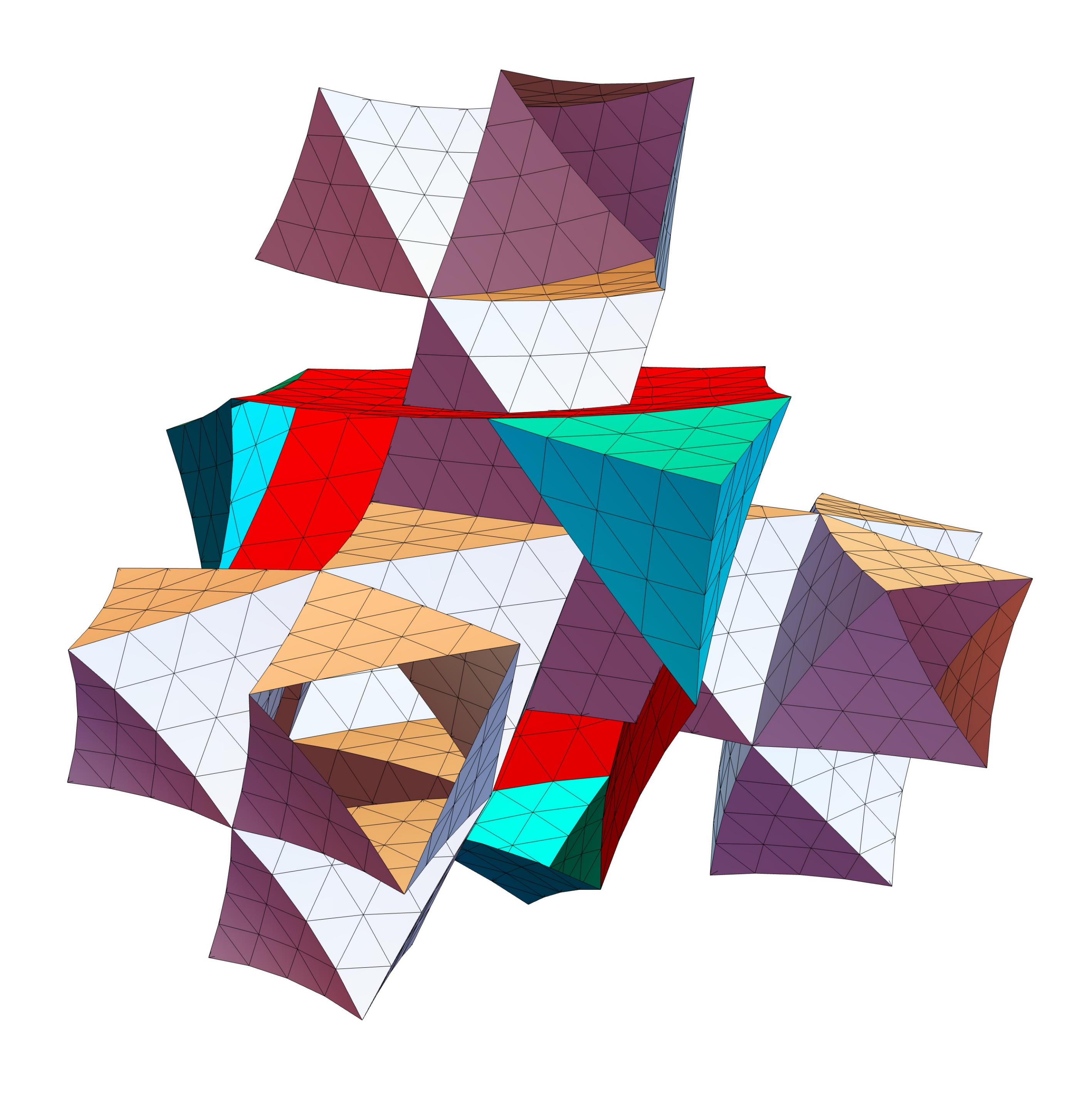}\qquad
   \includegraphics[width=2.5in]{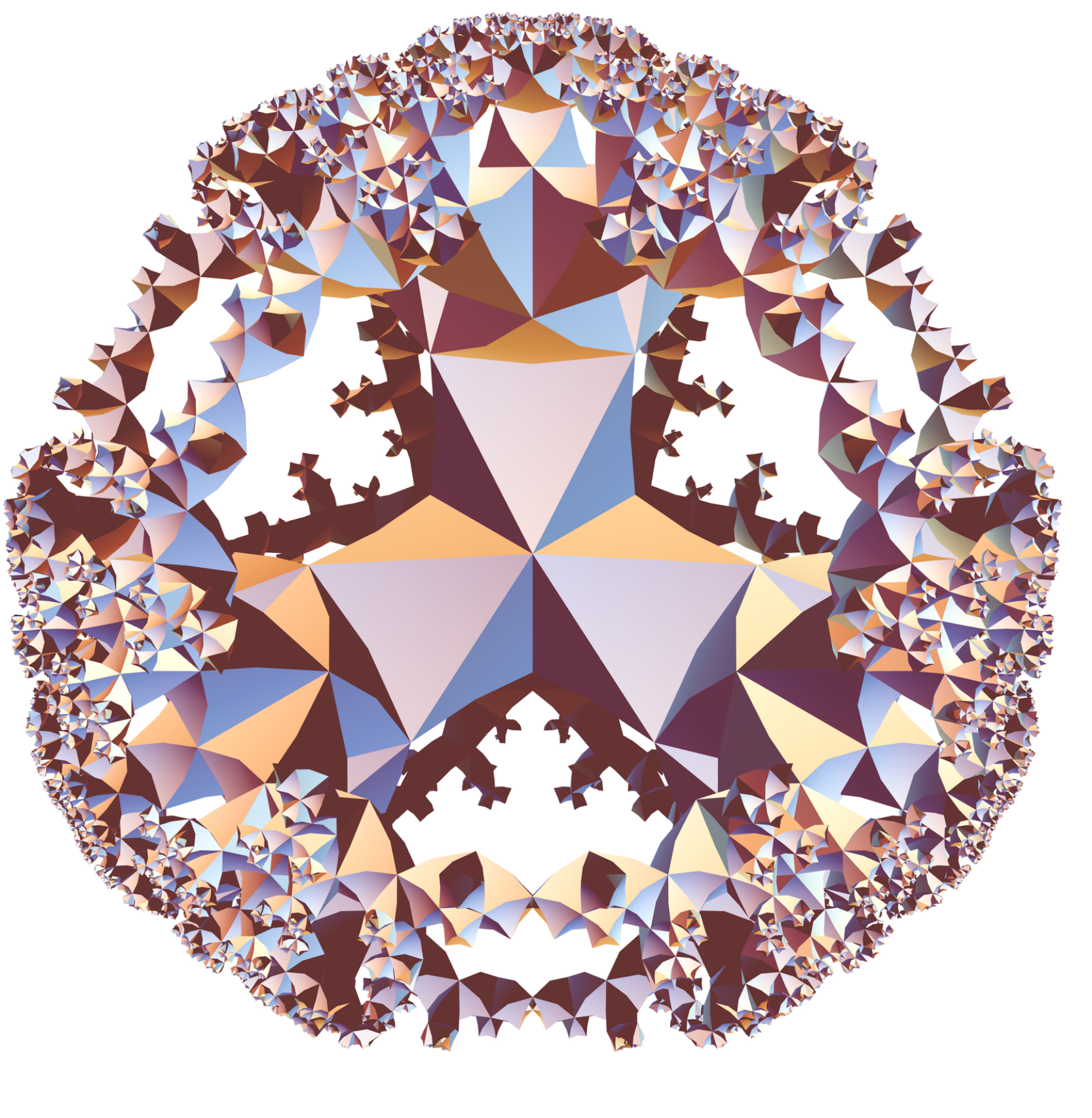}
      \caption{Hyperbolic antiprismatic tetrahedra}
   \label{fig:antitethyp4}
\end{figure}

\subsection{The Spherical Antiprismatic Tetrahedron}

For the construction of the spherical antiprismatic tetrahedron  $AP_{3,3}^2$ we need a spherical triakis truncated tetrahedon $KP_{3,3}^2$. This means that the triakis pyramids are degenerate (i.e. of height 0) and can be ignored: $KP_{3,3}^2$ is just a spherical truncated tetrahedron satisfying the angle condition $\alpha+2\beta=2\pi$ (with $\gamma=0$).  
\begin{theorem}
There exists a spherical truncated tetrahedron that tiles $\bS^3$.
\end{theorem}
\begin{proof}
We use an intermediate value argument: Small spherical truncated tetrahedra are almost Euclidean, and the angle sum $\alpha+2\beta$ is less than $2\pi$. We now increase the size of the truncated tetrahedron so that its vertices lie on a great sphere $\bS^2\subset \bS^3$. Then all dihedral angles are $\pi$, and the angle sum becomes $3\pi$. 

\end{proof}

\begin{figure}[h]
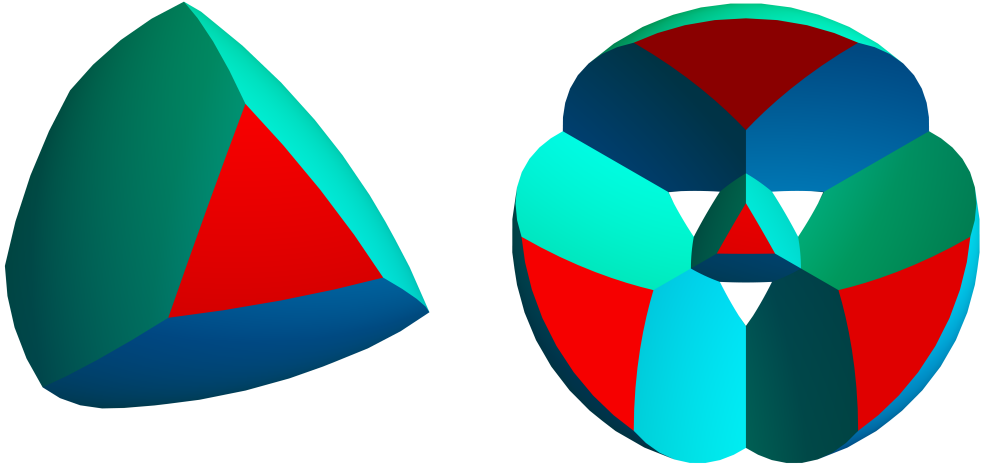
 
   \centering
   \includegraphics[width=2.5in]{figures/antiprismatictetrahedron/tetra-twist-sphere-tile.pdf}\qquad \includegraphics[width=2.5in]{figures/antiprismatictetrahedron/tetra-twist-sphere-tile3.pdf}
   \caption{Tiling  $\bS^3$ with truncated tetrahedra $KP_{3,3}^2$}
   \label{fig:tritrutettile2}
\end{figure}

Ten of these spherical truncated tetrahedra are needed to tile all of $\bS^3$, see figure \ref{fig:tritrutettile2}. The corresponding 4-polytope is called the {\em bitruncated 5-cell} (\cite{elte}, \S 22), see figure \ref{fig:sphantitet}.

\begin{figure}[h] 
   \centering
   \includegraphics[width=2.5in]{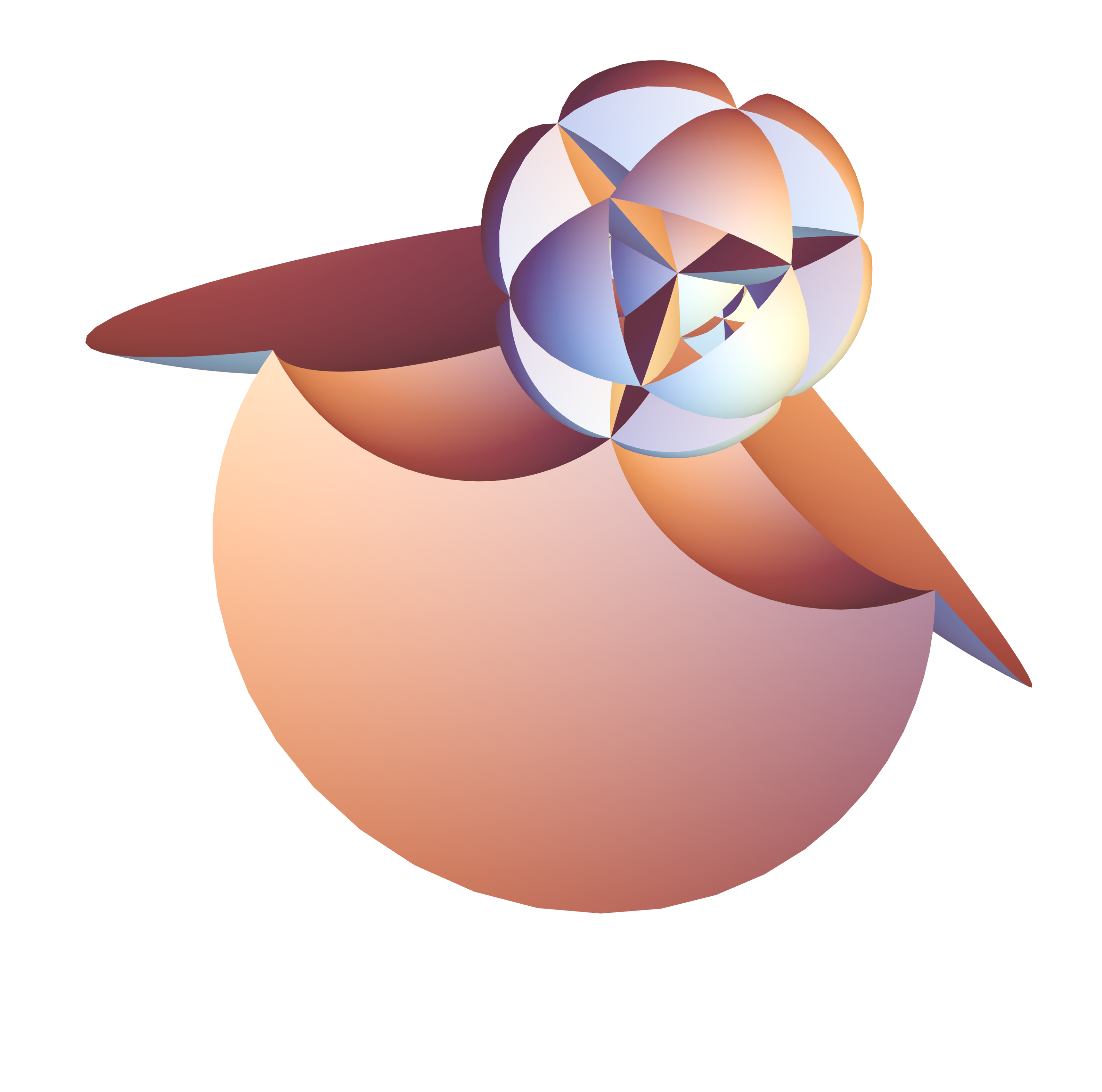}\qquad \includegraphics[width=2.5in]{figures/antiprismatictetrahedron/tetra-twist-sphere-cross4.pdf}
   \caption{The Spherical Antiprismatic Tetrahedron $AP_{3,3}^2$}
   \label{fig:sphantitet}
\end{figure}


\subsection{Antiprismatic Rectified Tetrahedra}

Using the same Triakis Truncated Tetrahedra as fundamental domains, we can vary the constructions of this section by using instead of prisms over a smaller tetrahedron prisms over a rectified tetrahedron (i.e. over every other face of an octahedron). These surfaces are tiled with equilateral triangles and have valency 8, see figure \ref{fig:euantirectet}.

The Euclidean version has been discussed in detail in \cite{Lee15}. It has a hidden conformal symmetry which makes it conformally Platonic. It can be represented as an eight-fold cyclically branched cover over the thrice punctured sphere and is conformally isomorphic to the Fermat quartic.

\begin{figure}[h] 
   \centering
   \includegraphics[width=2.5in]{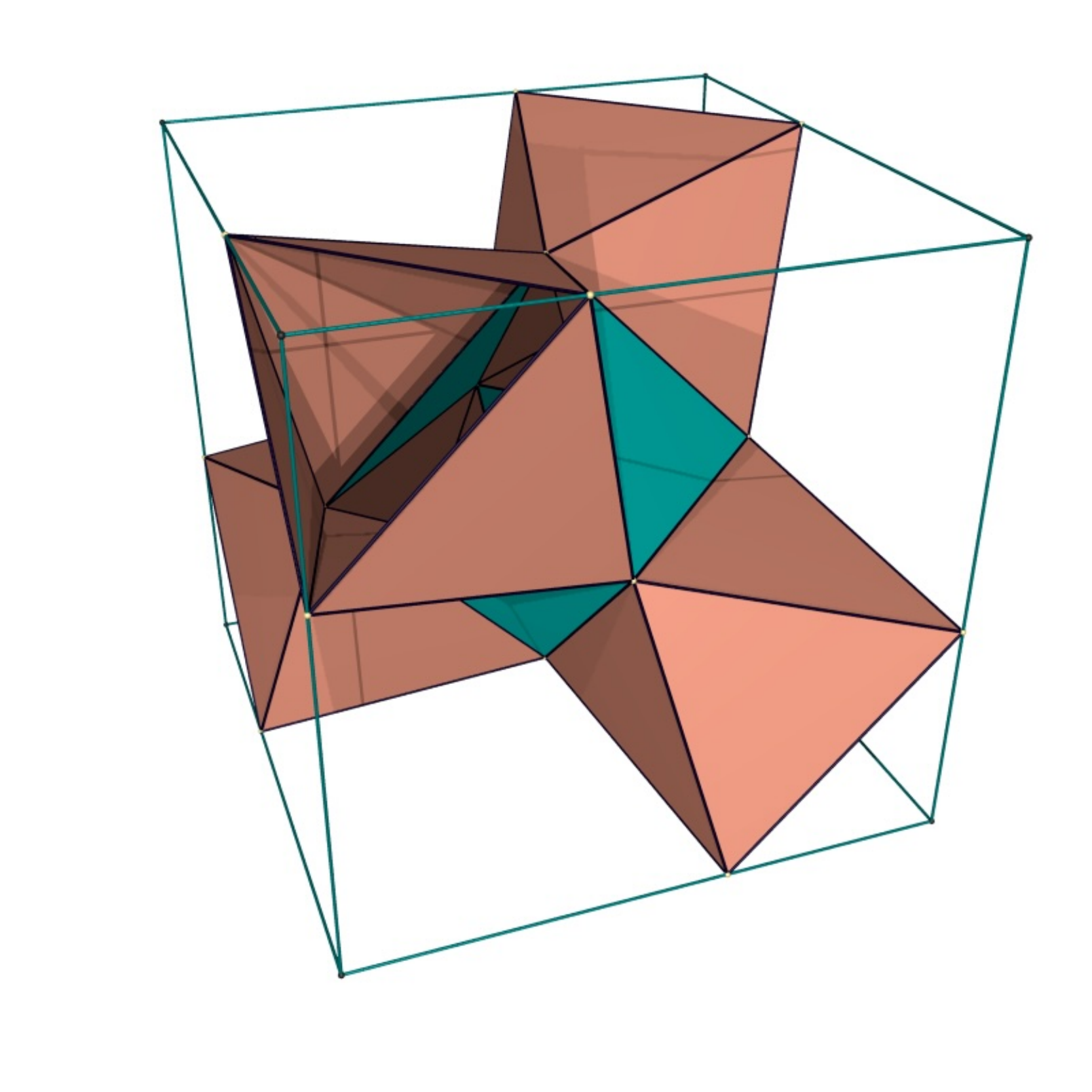}\qquad \includegraphics[width=2.5in]{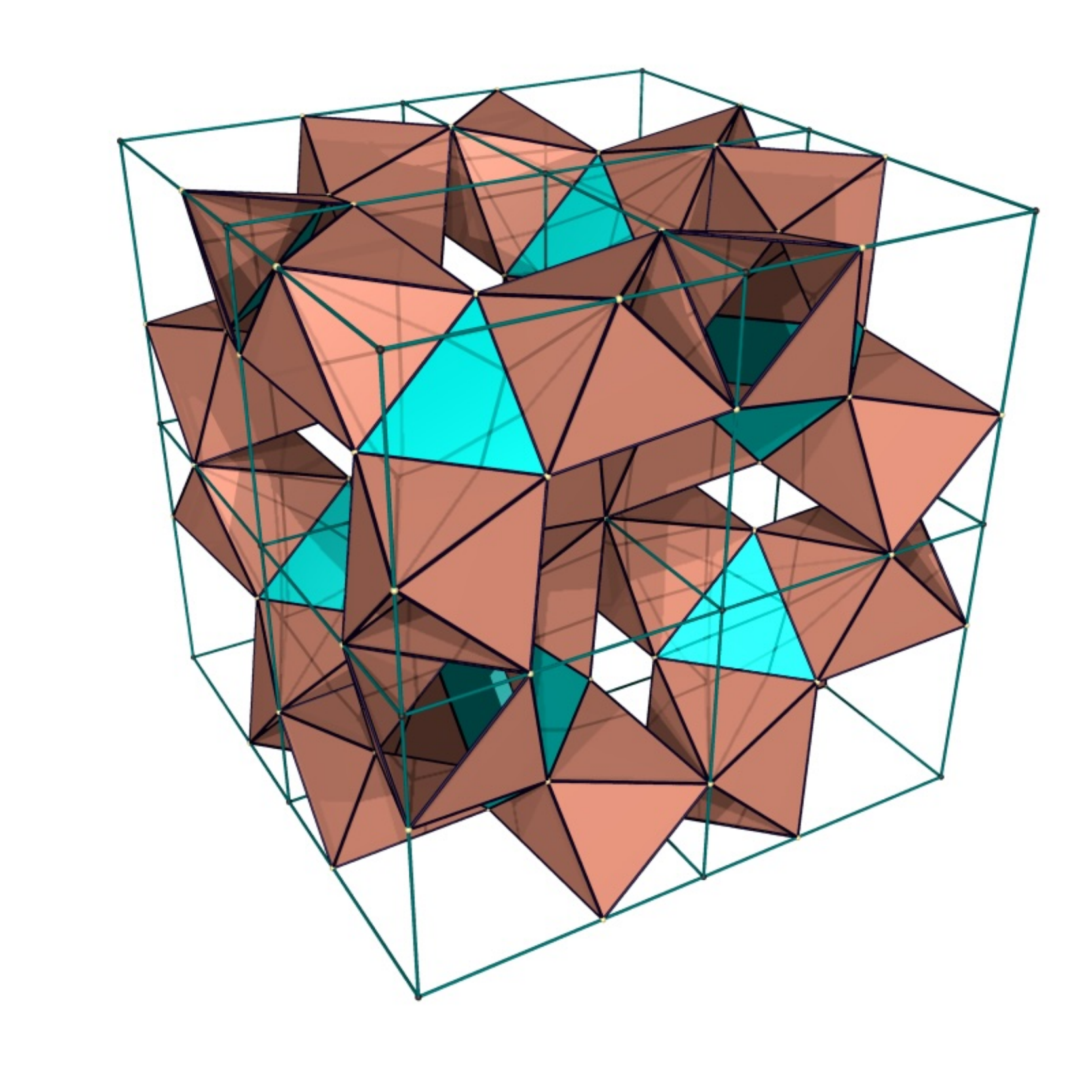}
   \caption{The Euclidean Antiprismatic Rectified Tetrahedron $ARP_{3,3}^3$}
   \label{fig:euantirectet}
\end{figure}

There also exist a spherical version and hyperbolic versions of this polyhedral surface. Suitable quotients are all conformally equivalent.


\section{The Antiprismatic Cube}

Antiprismatic cubes $KP_{4,3}^n$ exist in $\bS^3$ and $\bH^3$. They are obtained by attaching regular antiprisms to the faces of a cube. The resulting surfaces (after suitable identifications) have genus 3 and are tiled with 24 triangles and have valency 9 at each vertex.


\subsection{The Spherical Antiprismatic Cube}

48 copies of a slightly inflated truncated cube $KP_{4,3}^2$ will tile $\bS^3$, see figure \ref{fig:sphpristruncocta}. The identification space is  called the {\em truncated cube space} (\cite{mont}).

\begin{figure}[h] 
   \centering
   \includegraphics[width=2.5in]{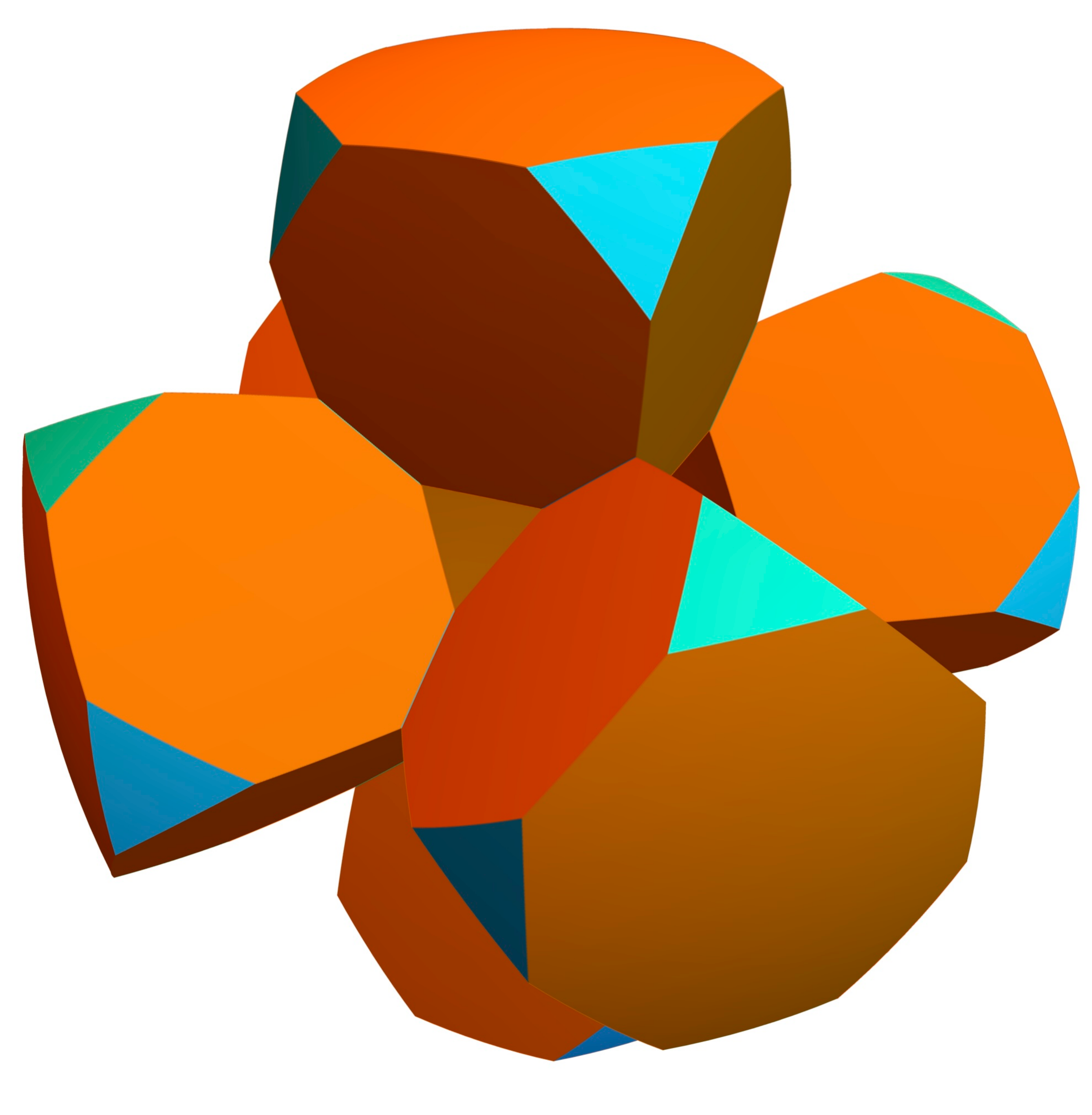}
   \qquad 
   \includegraphics[width=2.5in]{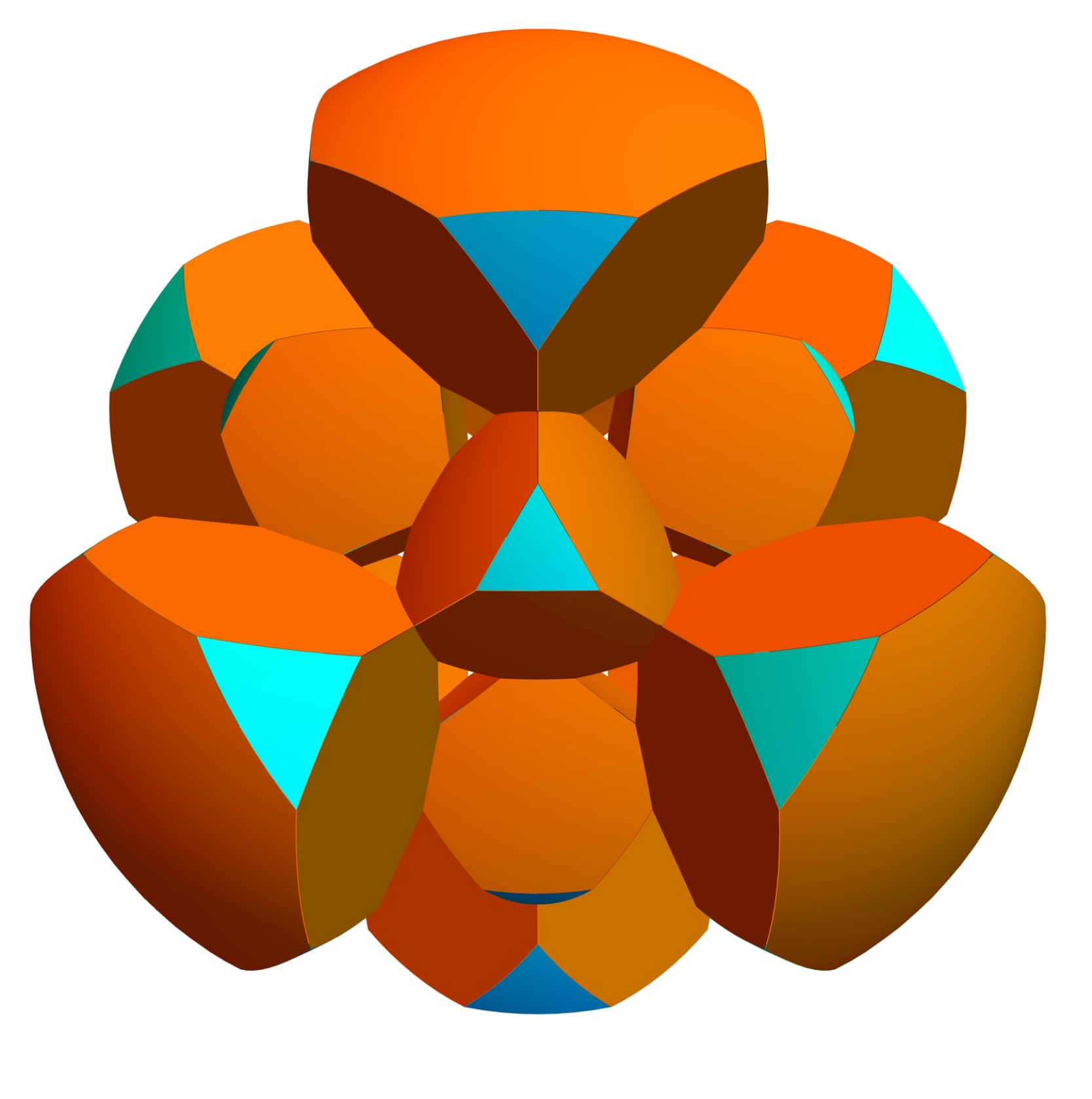}
   \caption{Spherical Truncated Cube Space}
   \label{fig:sphpristruncocta}
\end{figure}

The usual intermediate value argument allows to find a cube  inside the truncated cube so that the octagonal faces of the truncated cube bisect regular antiprisms over the faces of the cube, creating a spherical antiprismatic cube $AP_{4,3}^2$ see figure \ref{fig:sphericalanticube}.

\begin{figure}[h] 
   \centering
   \includegraphics[width=2.5in]{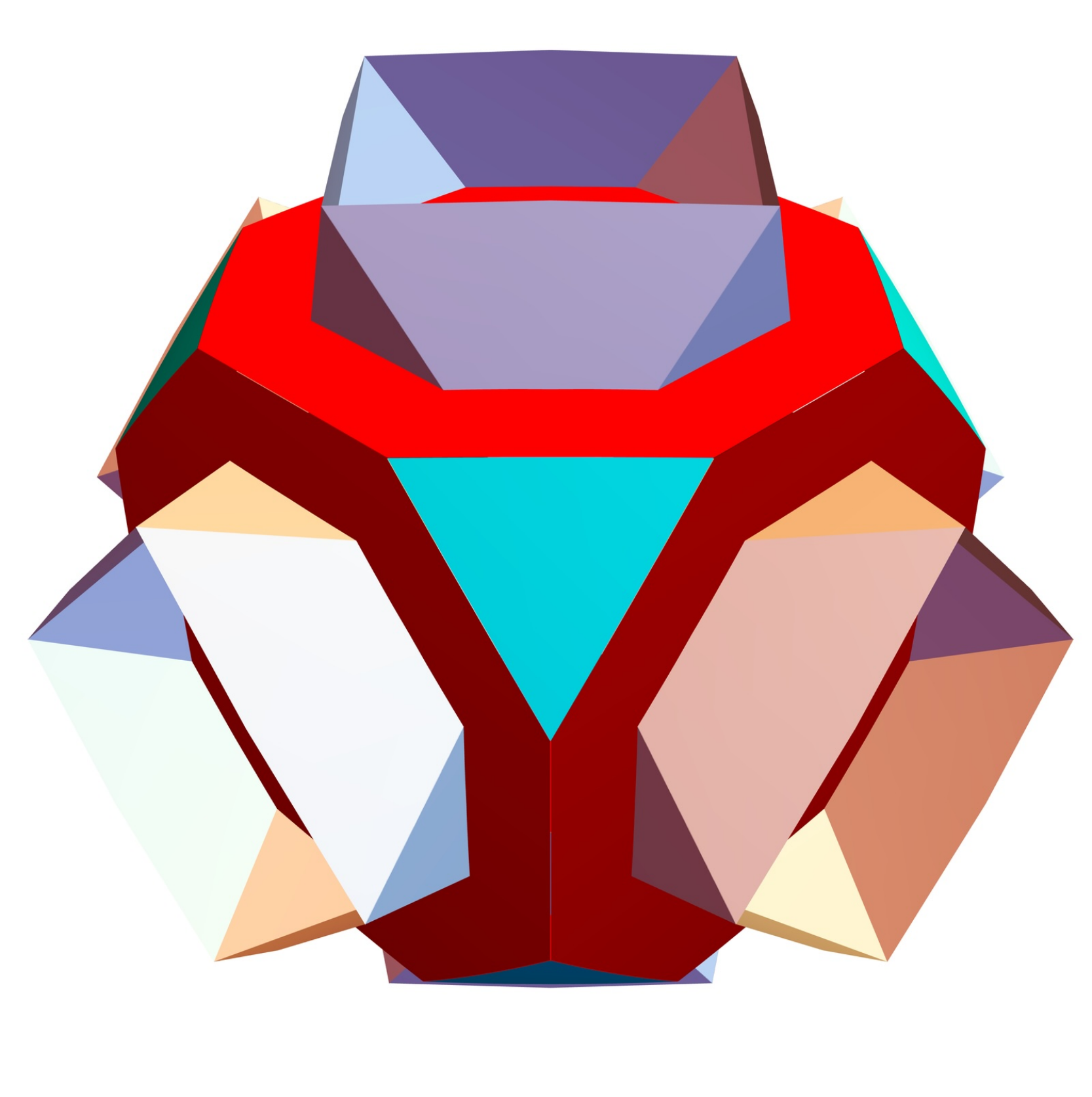}
   \qquad 
   \includegraphics[width=2.5in]{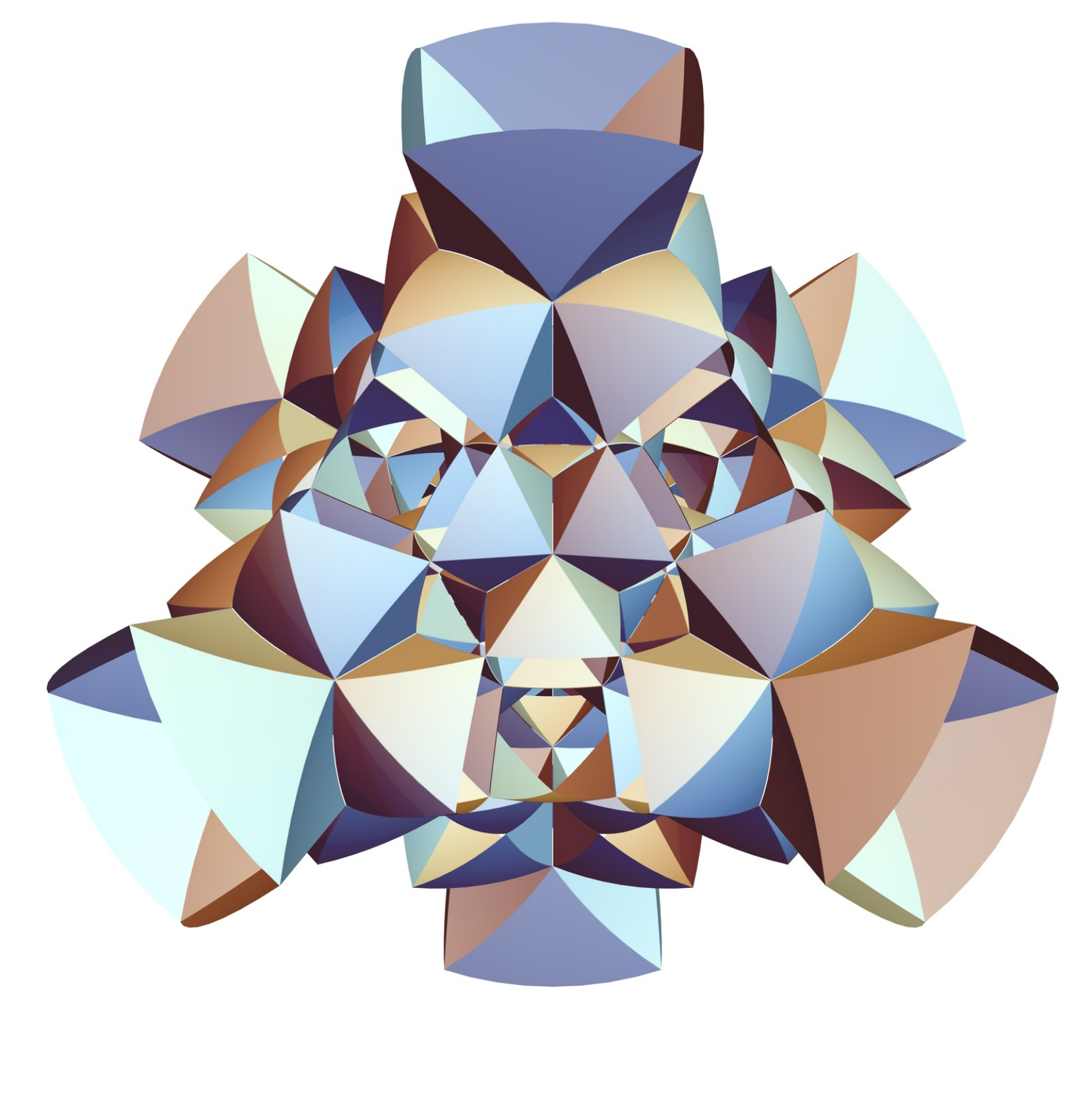}
   \caption{The Spherical Antiprismatic Cube $AP_{4,3}^2$}
   \label{fig:sphericalanticube}
\end{figure}

 \subsection{The Hyperbolic Antiprismatic Cubes}
 
 To construct hyperbolic antiprismatic cubes $AP_{4,3}^n$, we augment hyperbolic truncated cubes with pyramids  over their triangular faces to obtain $n$-akis truncated cubes $KP_{4,3}^n$. The pyramids are chosen so that their  triangular face angles are equal to $2\pi/n$ for $n\ge3$, see figure \ref{fig:hyperanticube}

We have:
 
 \begin{theorem}
 For each $n\ge3$ there exists an $n$-akis  truncated cube $AP_{4,3}^n$ that is a fundamental domain for a discrete group $R_{4,3}^n$ generated by  reflection-rotations at its octagonal faces. Inside $KP_{4,3}^n$  we can find a smaller cube so that regular antiprisms attached to its faces are bisected by the octagonal faces of $KP_{4,3}^n$. Replicating the antiprisms under the group elements of $R_{4,3}^n$ generates a hyperbolic antiprismatic cube $AP_{4,3}^n$.
 \end{theorem}
 
 The proof is similar to the one given for octahedra in section \ref{sec:octa}.
 
 \begin{figure}[h] 
   \centering
   \includegraphics[width=2.5in]{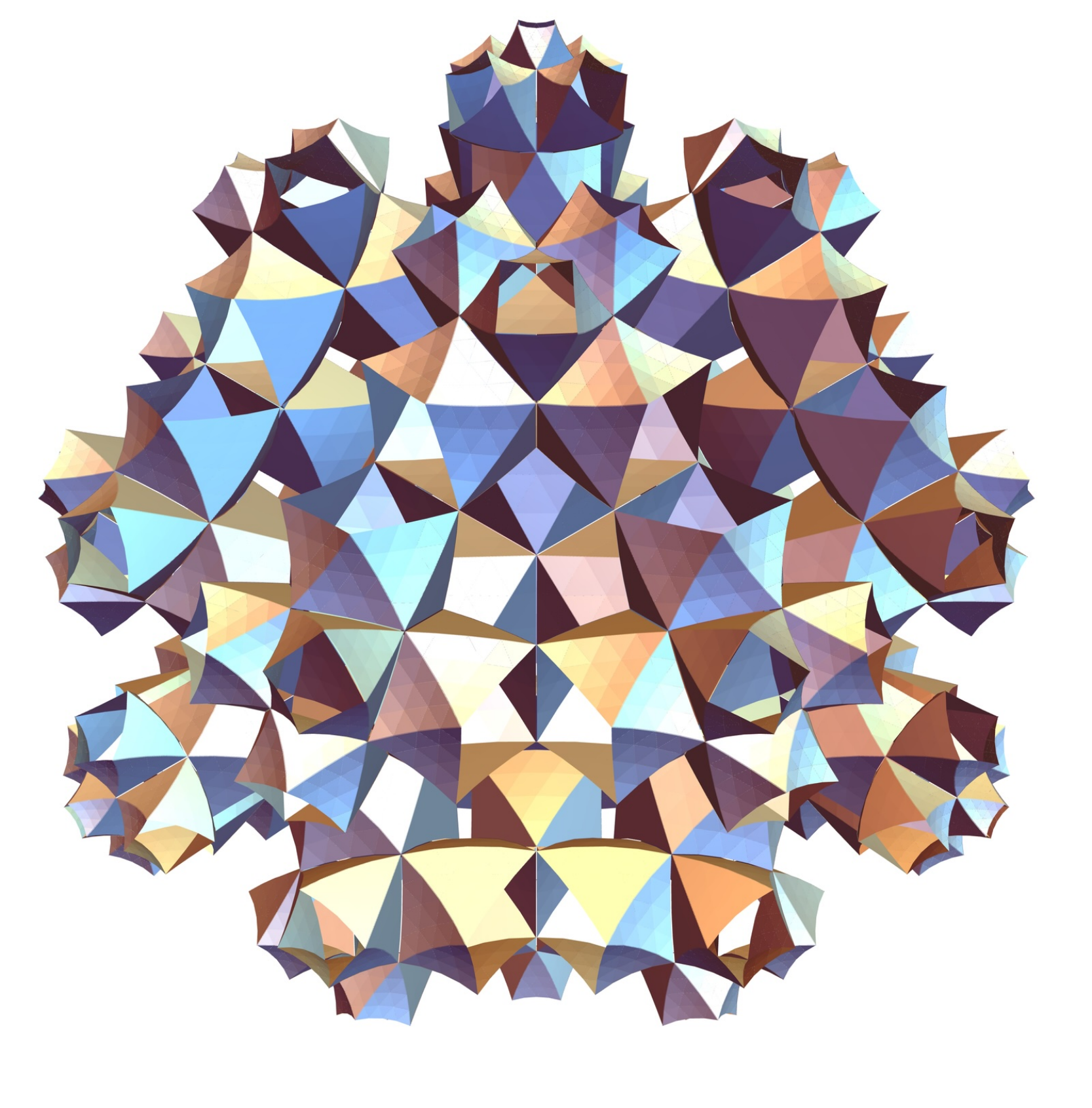}
   \qquad 
   \includegraphics[width=2.5in]{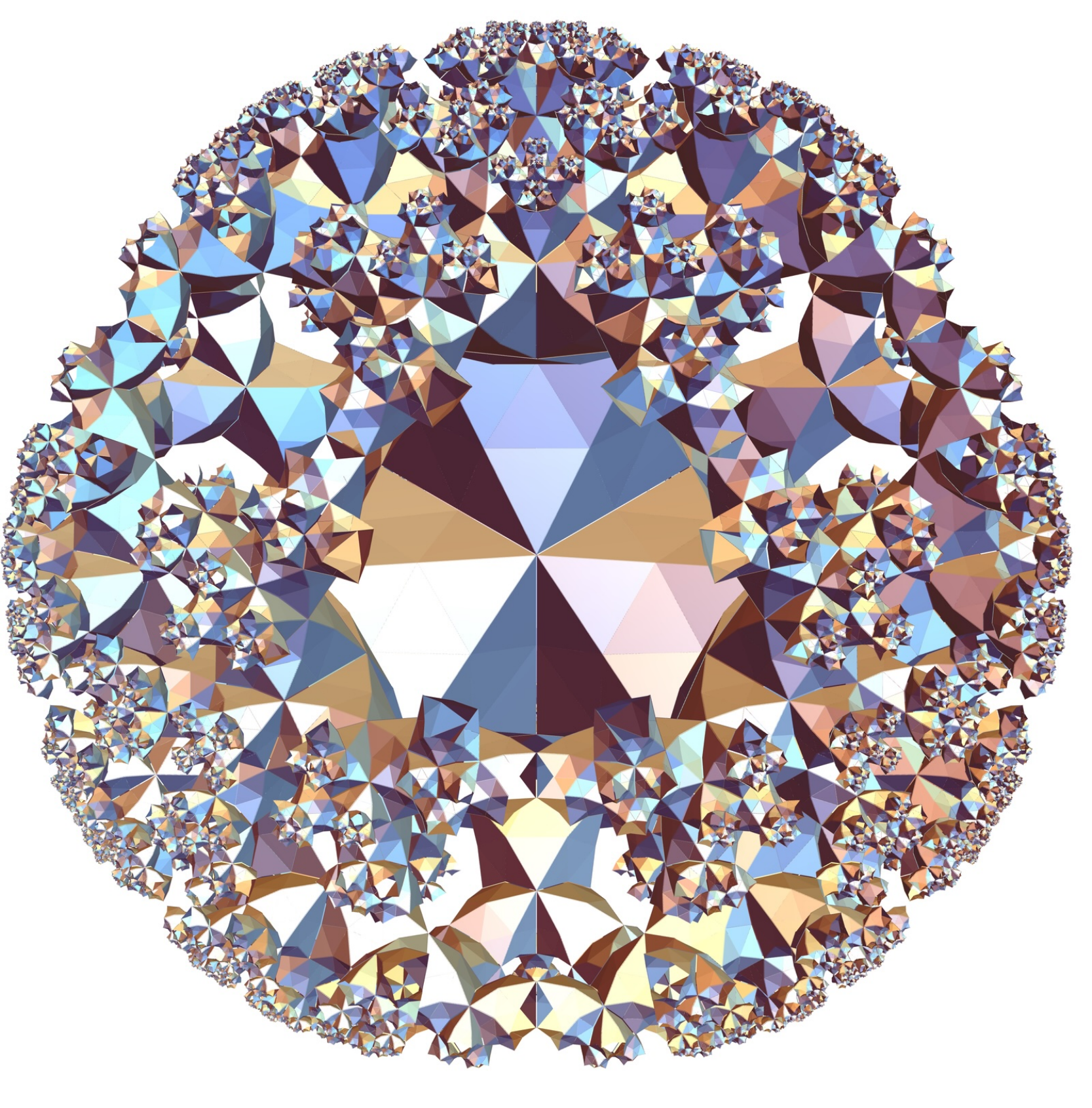}
   \caption{Hyperbolic Antiprismatic Cubes $AP_{4,3}^n$ for $n=3$ and $n=5$}
   \label{fig:hyperanticube}
\end{figure}

\subsection{Antiprismatic Cuboctahedra}\label{sec:apricuboc}

Using the same $n$-akis truncated cubes $KP_{4,3}^n$, we can also construct antiprismatic cuboctahedra $ATP_{4,3}^n$, see figure \ref{fig:hyperanticuboc}. For $n=2$ there is one in $\bS^3$, and for $n\ge 3$ they exist in $\bH^3$. The valency of the vertices is 8.

 \begin{figure}[h] 
   \centering
   \includegraphics[width=2.5in]{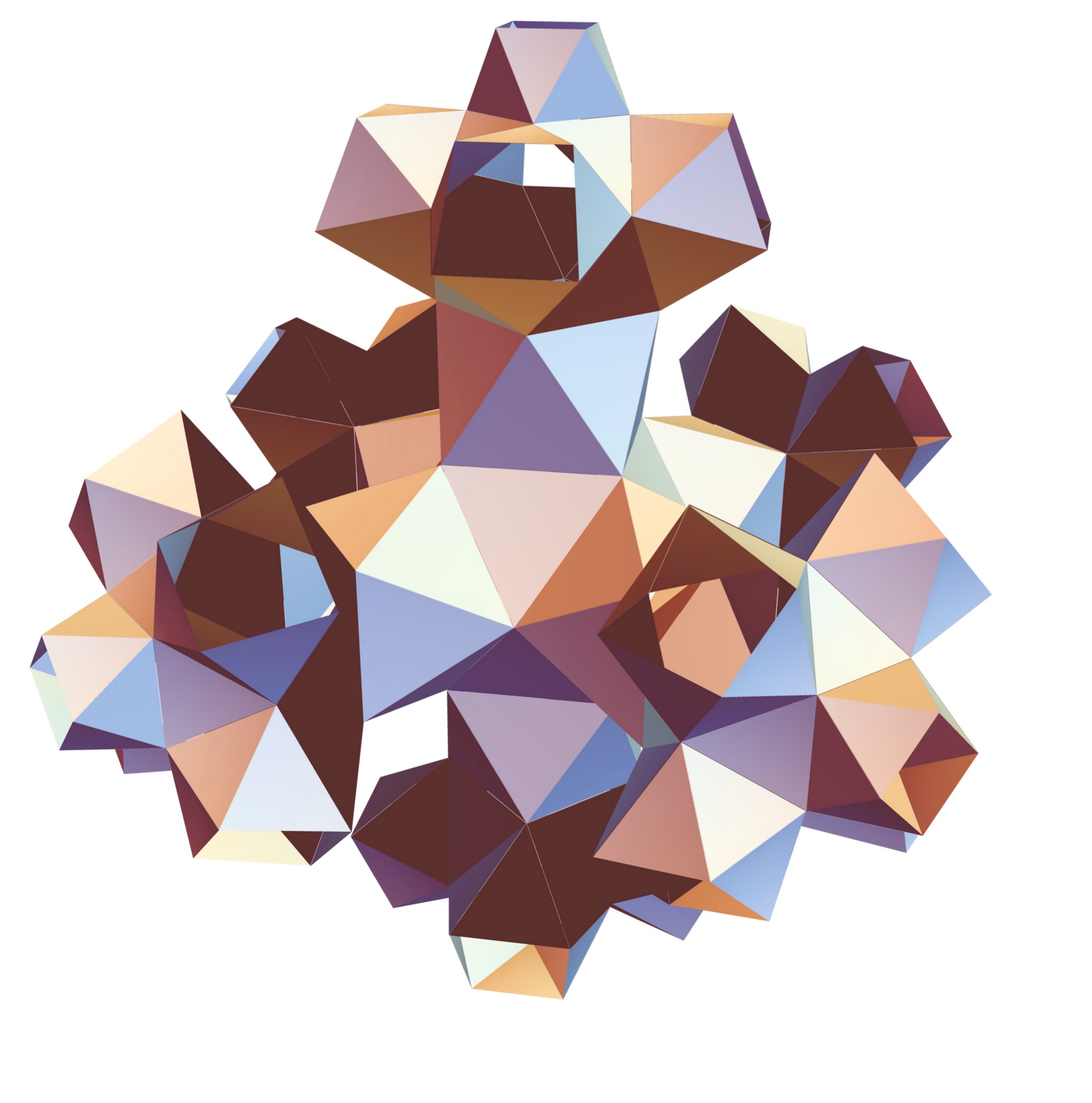}
   \qquad 
   \includegraphics[width=2.5in]{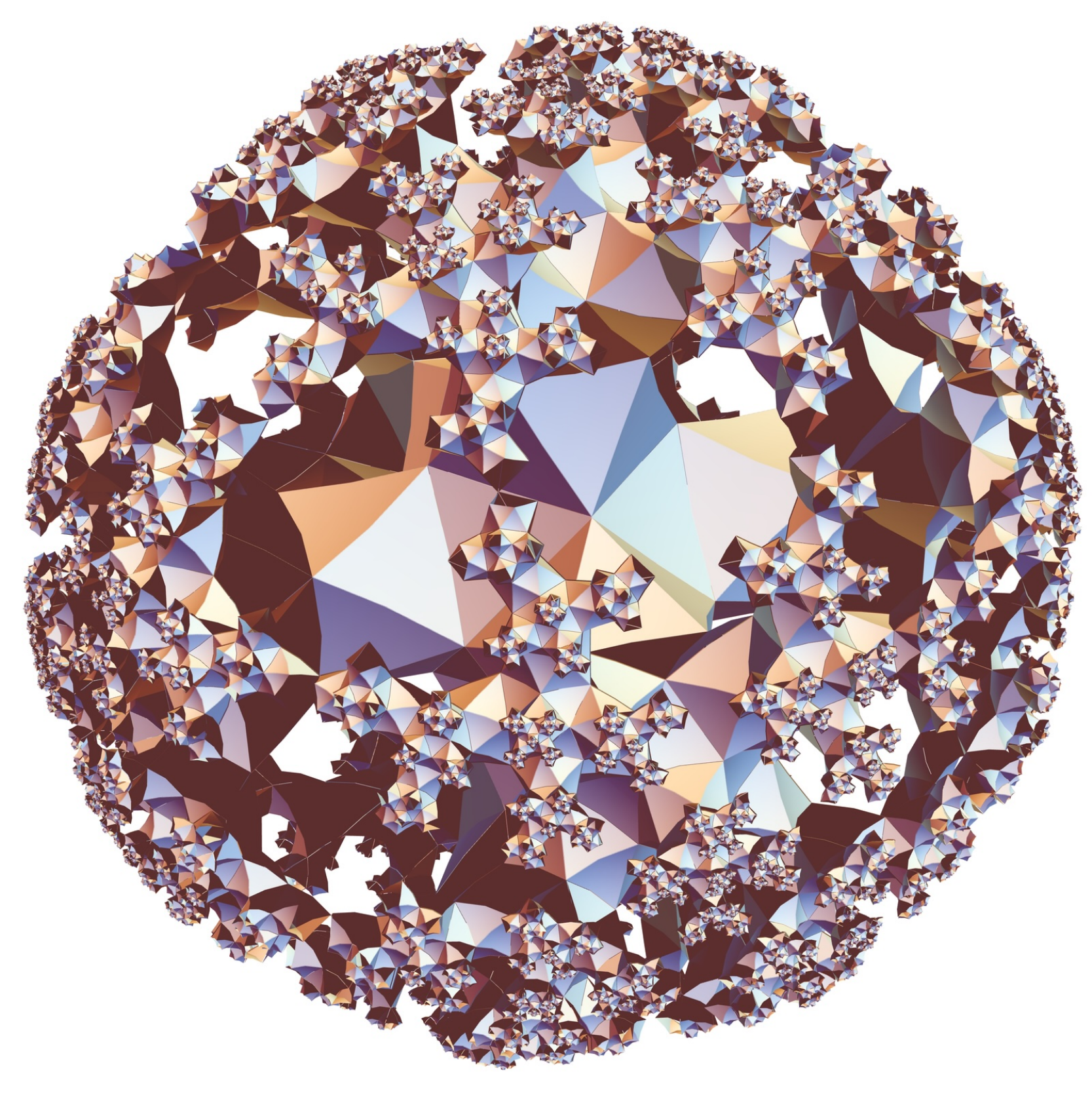}
   \caption{Hyperbolic Antiprismatic Cuboctahedra $ATP_{4,3}^n$ for $n=3$ and $n=5$}
   \label{fig:hyperanticuboc}
\end{figure}

\section{Antiprismatic Hyperbolic Dodecahedra and Icosahedra}

Antiprismatic dodecahedra $AP_{5,3}^n$ (figures \ref{fig:AntiprismaticDodeca2}, \ref{fig:AntiprismaticDodeca3}) and icosahedra  $AP_{3,5}^n$ (figure \ref{fig:AntiprismaticIcosa3}) can only exist in hyperbolic space, as in Euclidean space the angle sum exceeds $360^\circ$ already for $n=2$. 

In the simplest case, for $n=2$, the fundamental domain for the group $R_{5,3}^2$ is a truncated dodecahedron without added pyramids, see Figure \ref{fig:AntiprismaticDodeca2}.

 \begin{figure}[h] 
   \centering
   \includegraphics[width=2in]{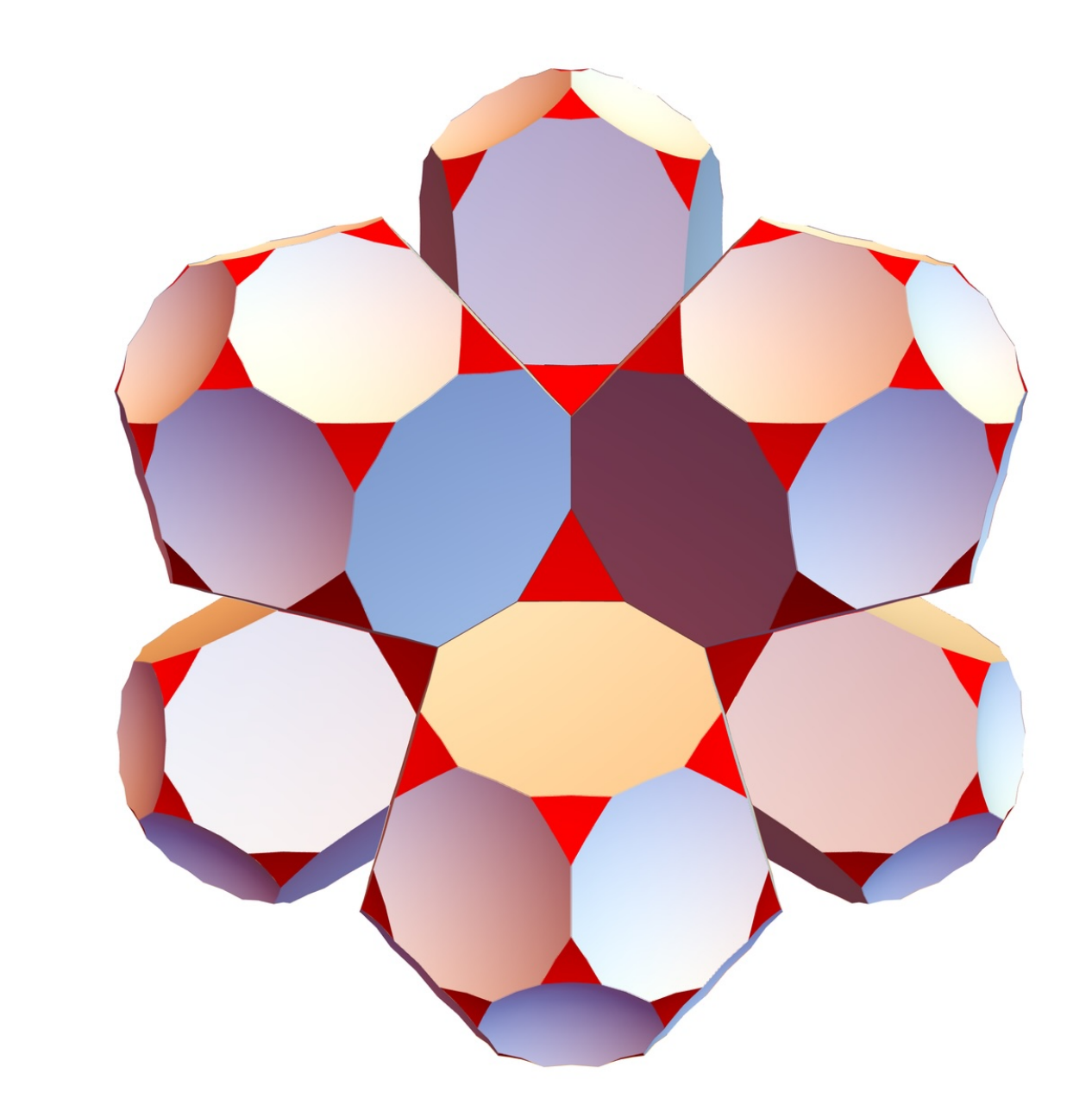}
   \quad
   \includegraphics[width=2in]{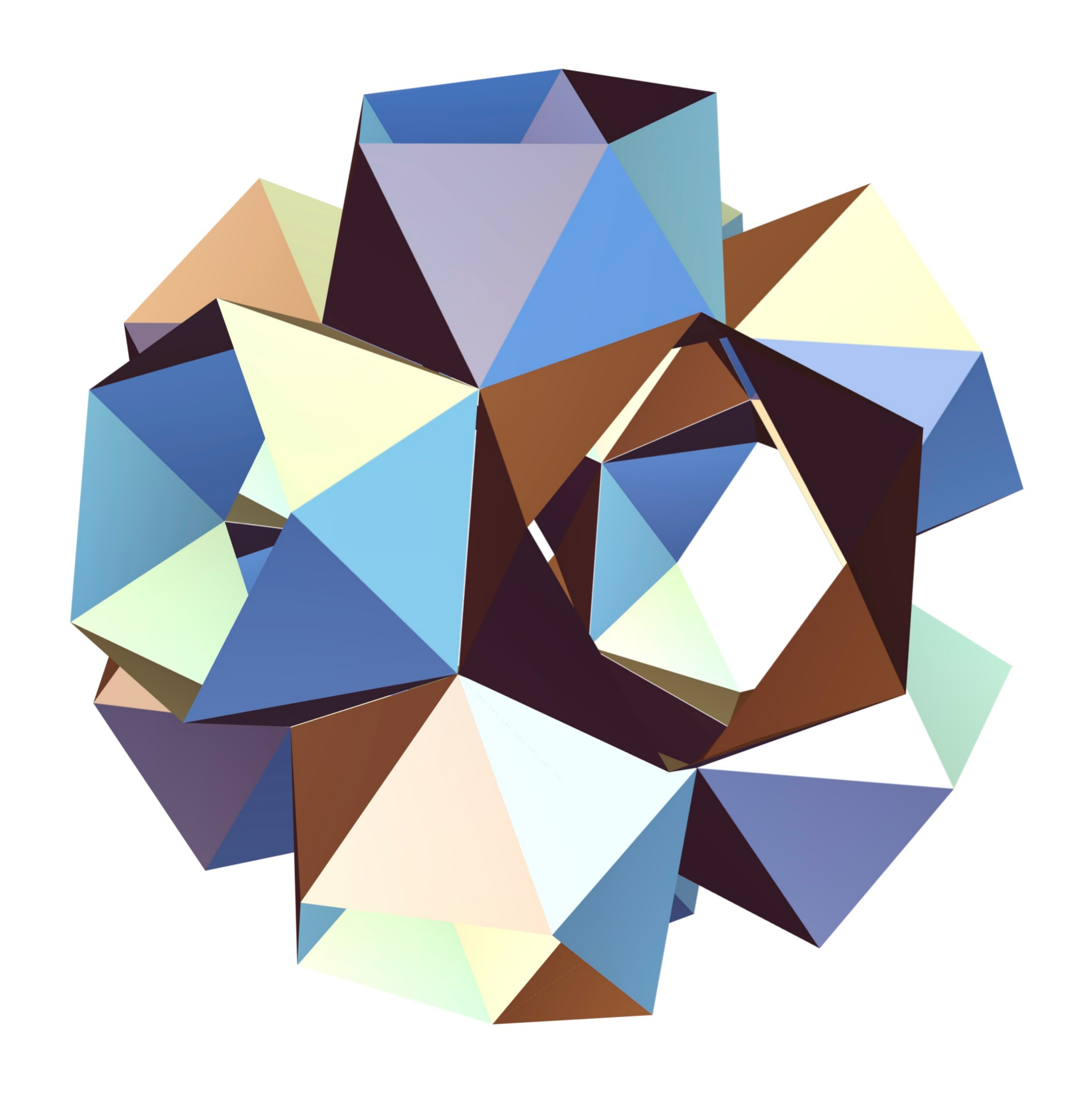}
      \caption{Tiling truncated dodecahedra and an antiprismatic dodecahedron $AP_{5,3}^2$}
   \label{fig:AntiprismaticDodeca2}
\end{figure}

We have

\begin{theorem}
For any $n\ge 2$, there are an antiprismatic dodecahedron $AP_{5,3}^n$ and icosahedron $AP_{3,5}^n$ in hyperbolic space.
\end{theorem}

The identification space of the antiprismatic dodecahedron has genus 6, is tiled by 60 triangles with valency 9.

The proofs are very similar to the other cases but computationally more involved and left to the reader.

 \begin{figure}[h] 
   \centering
   \includegraphics[width=1.8in]{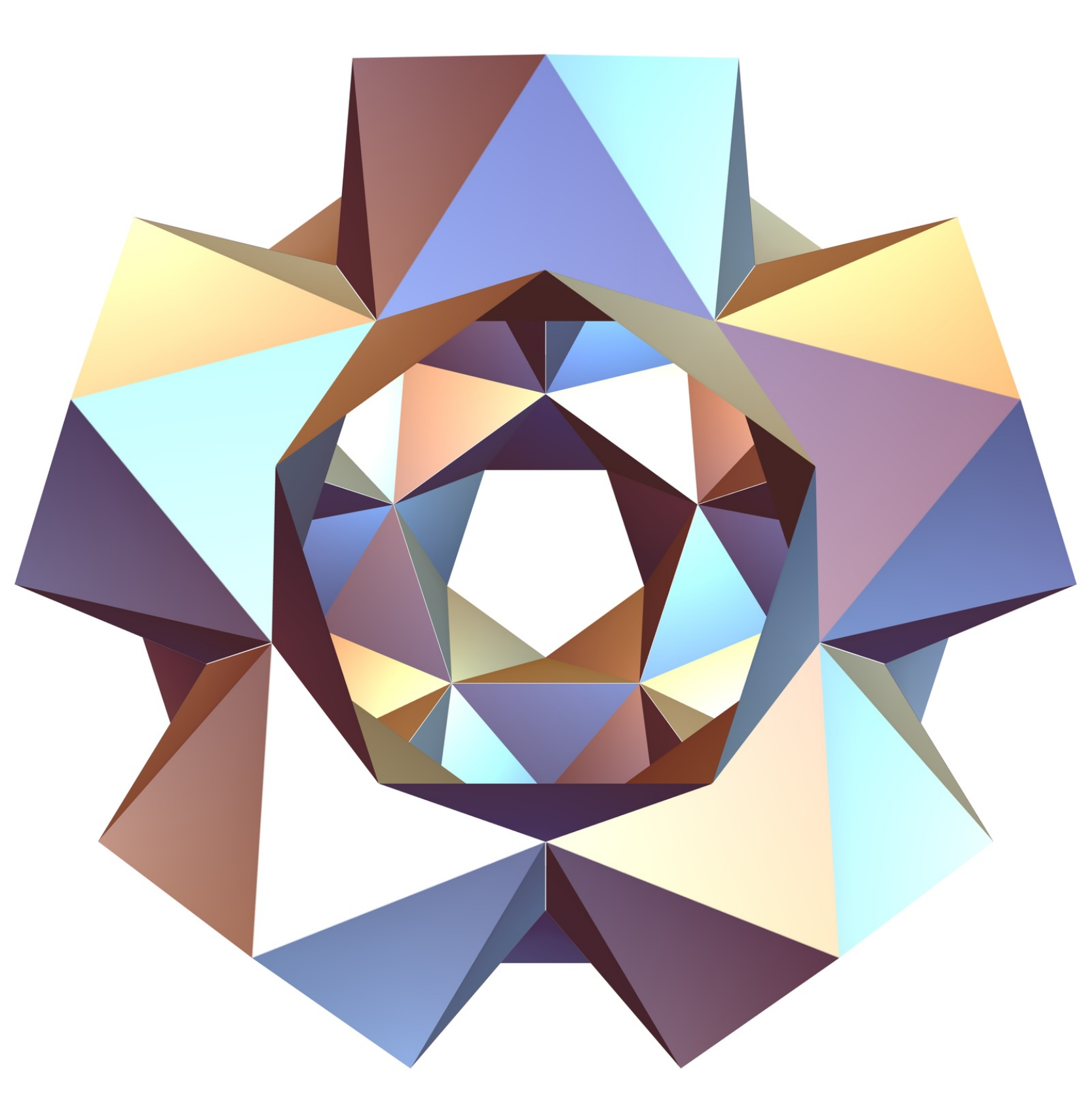}
   \quad
   \includegraphics[width=1.8in]{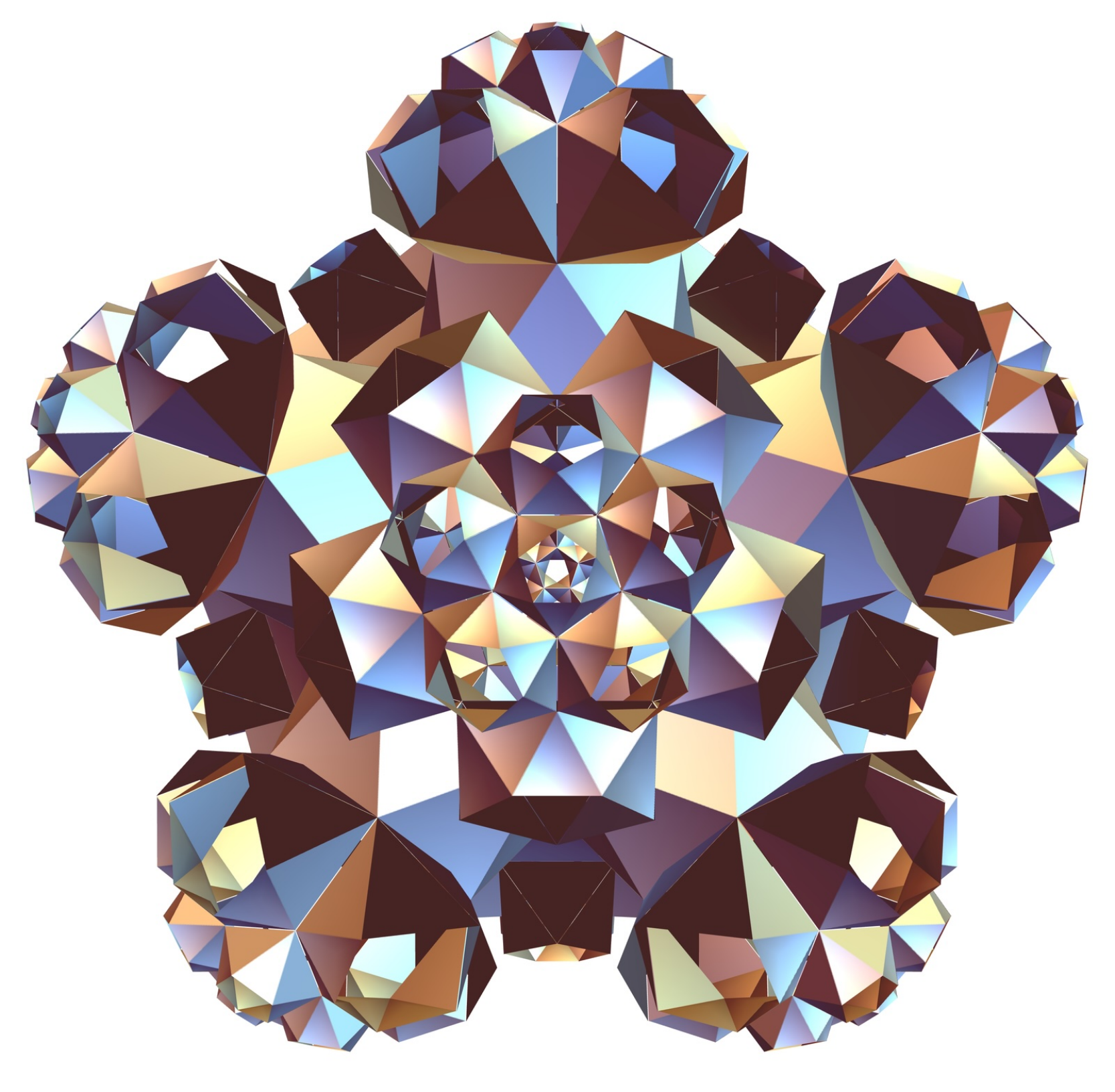}
   \quad
      \includegraphics[width=1.8in]{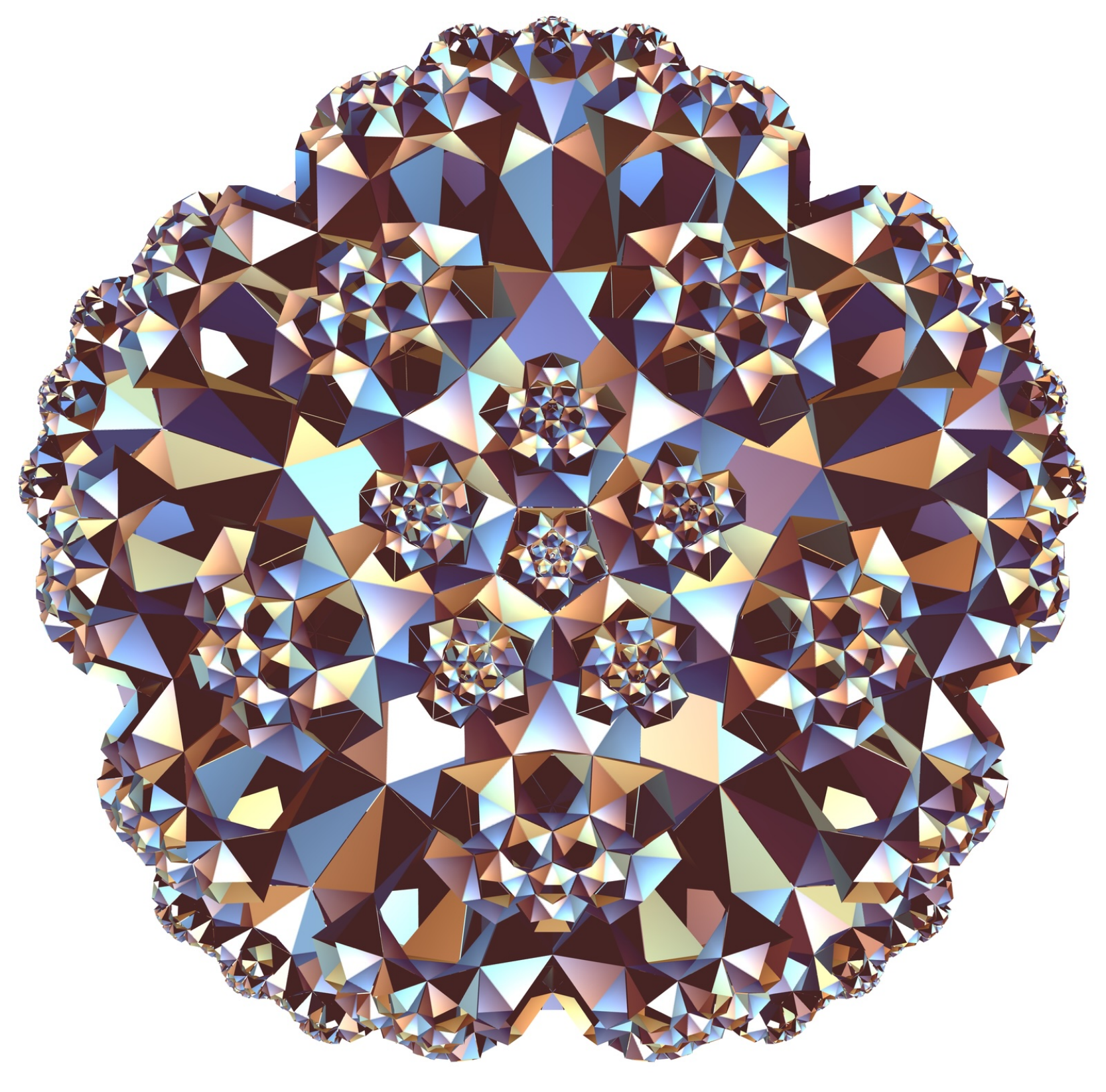}
      \caption{Antiprismatic Dodecahedra for $n=3$}
   \label{fig:AntiprismaticDodeca3}
\end{figure}

There is also the possibility to base the construction on an icosadodecahedron, the rectification of a dodecahedron, using its pentagonal faces for attaching antiprisms and keeping  its triangular faces. The quotient of the antiprismatic icosidodecahedron $ARP_{5,3}^n$ has genus 6, is tiled by 80 triangles with valency 8.

 \begin{figure}[h] 
   \centering
   \includegraphics[width=1.8in]{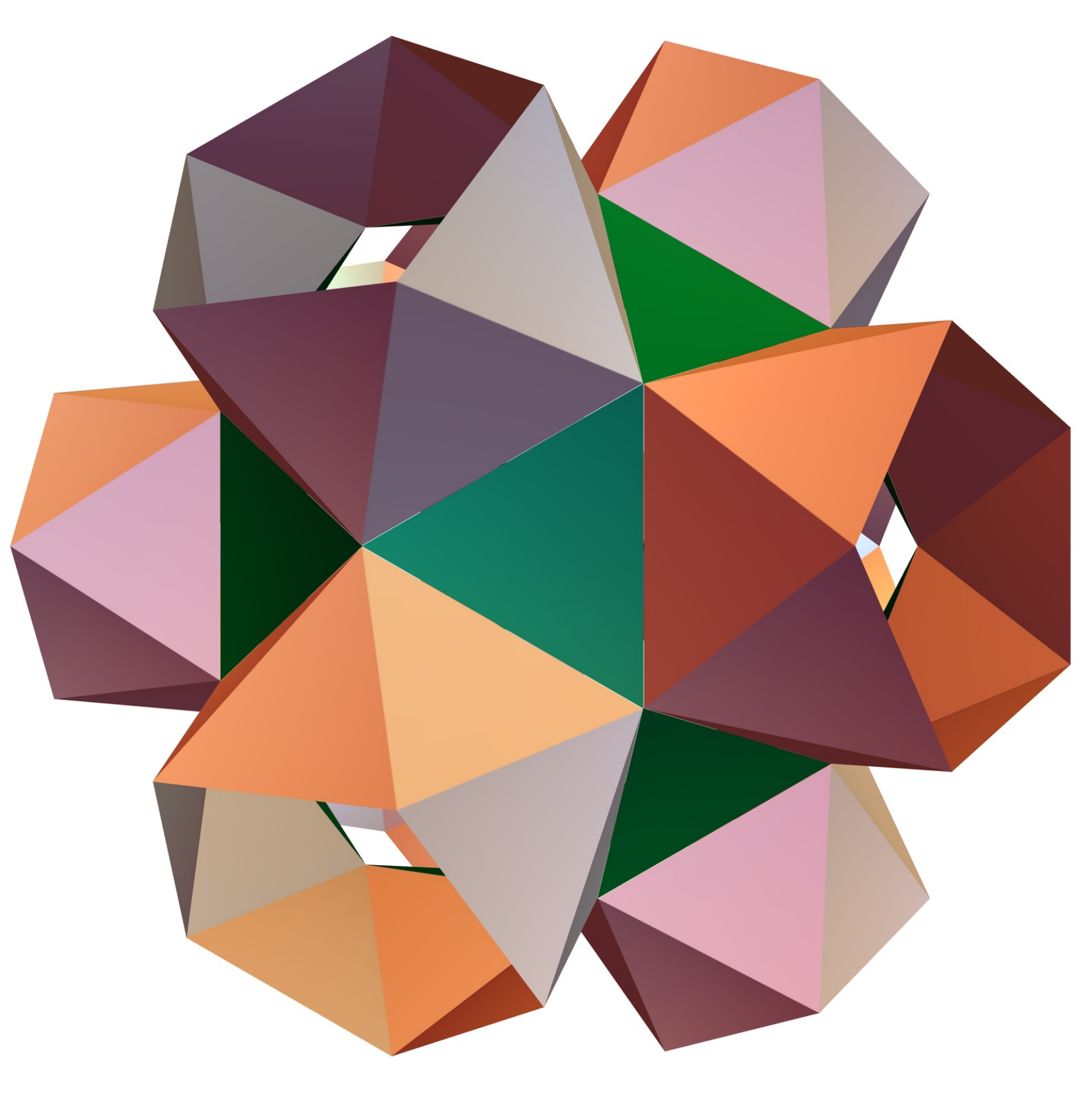}
   \quad
   \includegraphics[width=1.8in]{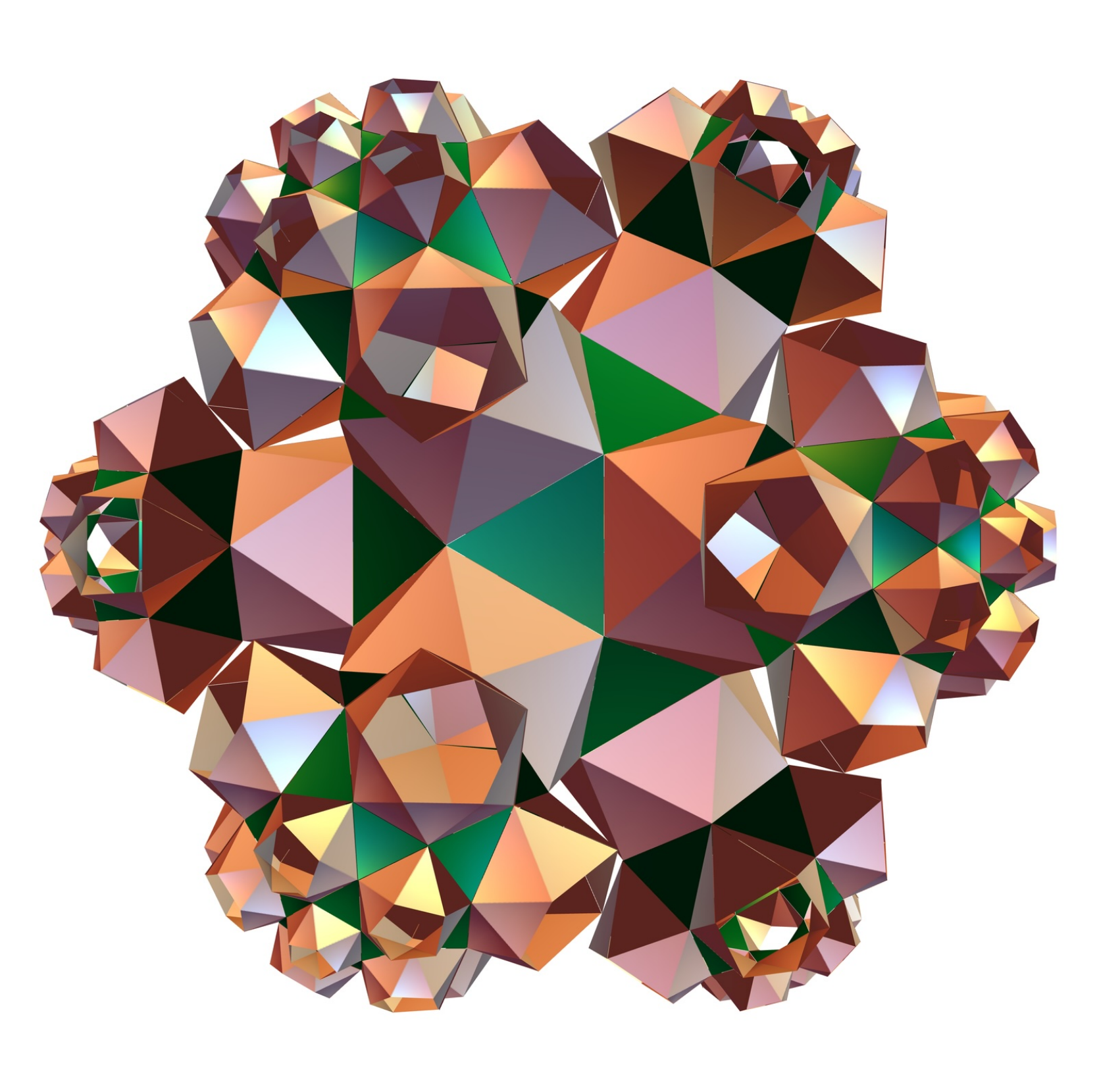}
   \quad
      \includegraphics[width=1.8in]{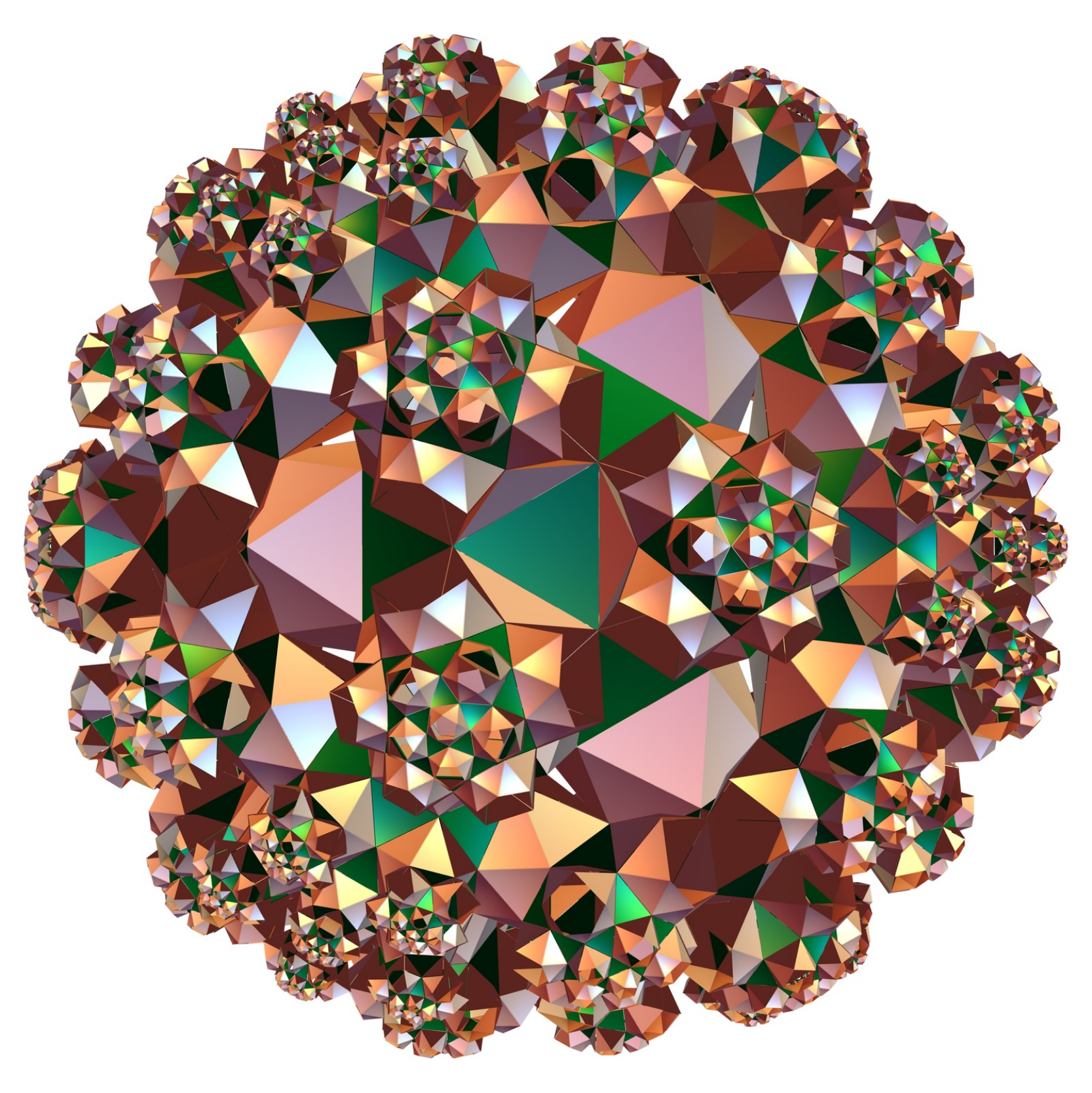}
      \caption{Antiprismatic Icosidodecahedra $ARP_{5,3}^n$ for $n=3$}
   \label{fig:AntiprismaticIcosa3}
\end{figure}

To construct antiprismatic icosahedra in hyperbolic space, we need pyramids $PY_5^n$ for $n\ge 2$. For $n=2$ these are just pentagons, for $n=3$ they can be obtained by subdividing a Platonic dodecahedron, and for $n\ge 4$ we can follow the same  construction as in section \ref{sec:pyramids}.

 \begin{figure}[h] 
   \centering
   \includegraphics[width=1.8in]{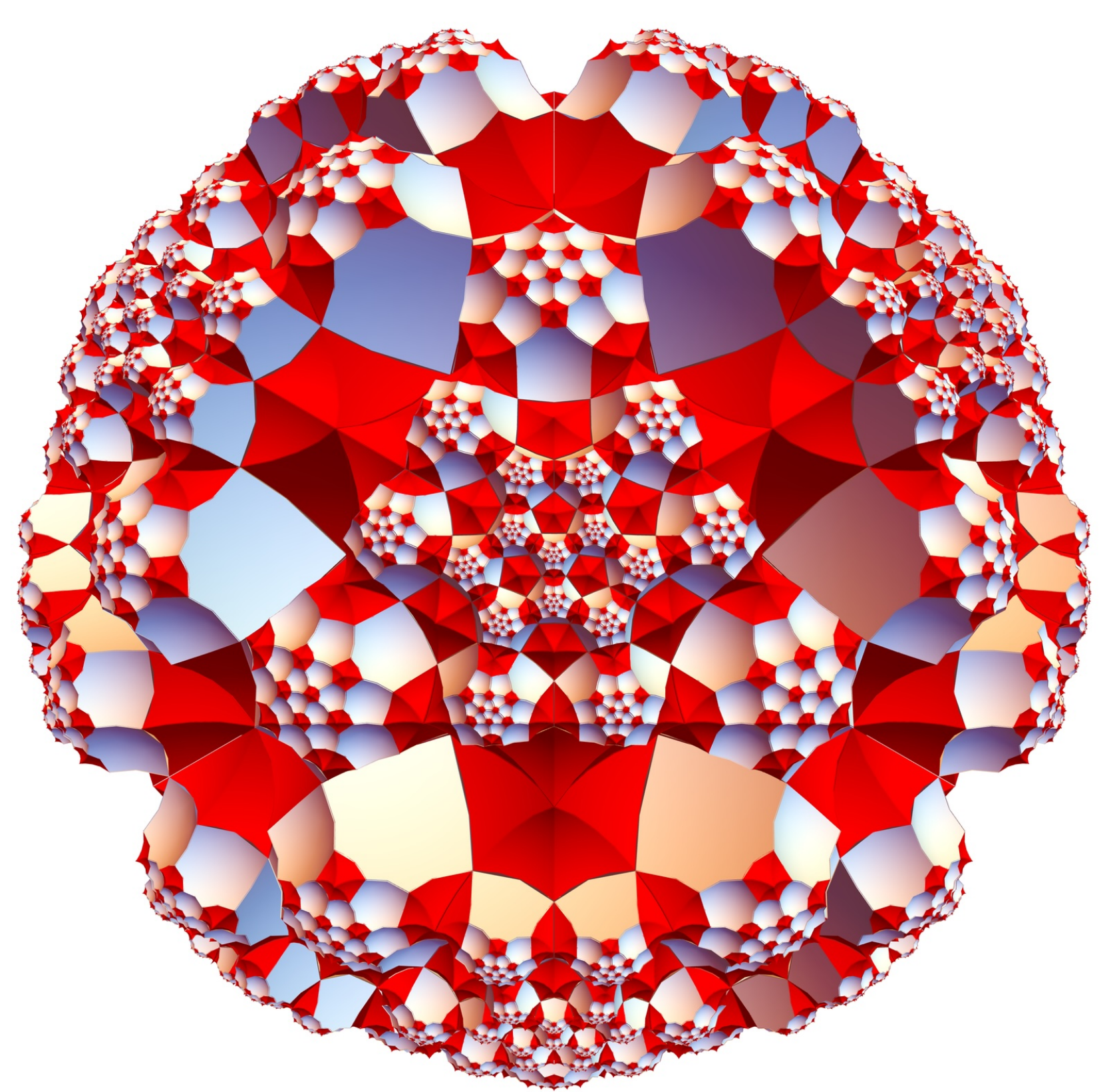}
   \quad
   \includegraphics[width=1.8in]{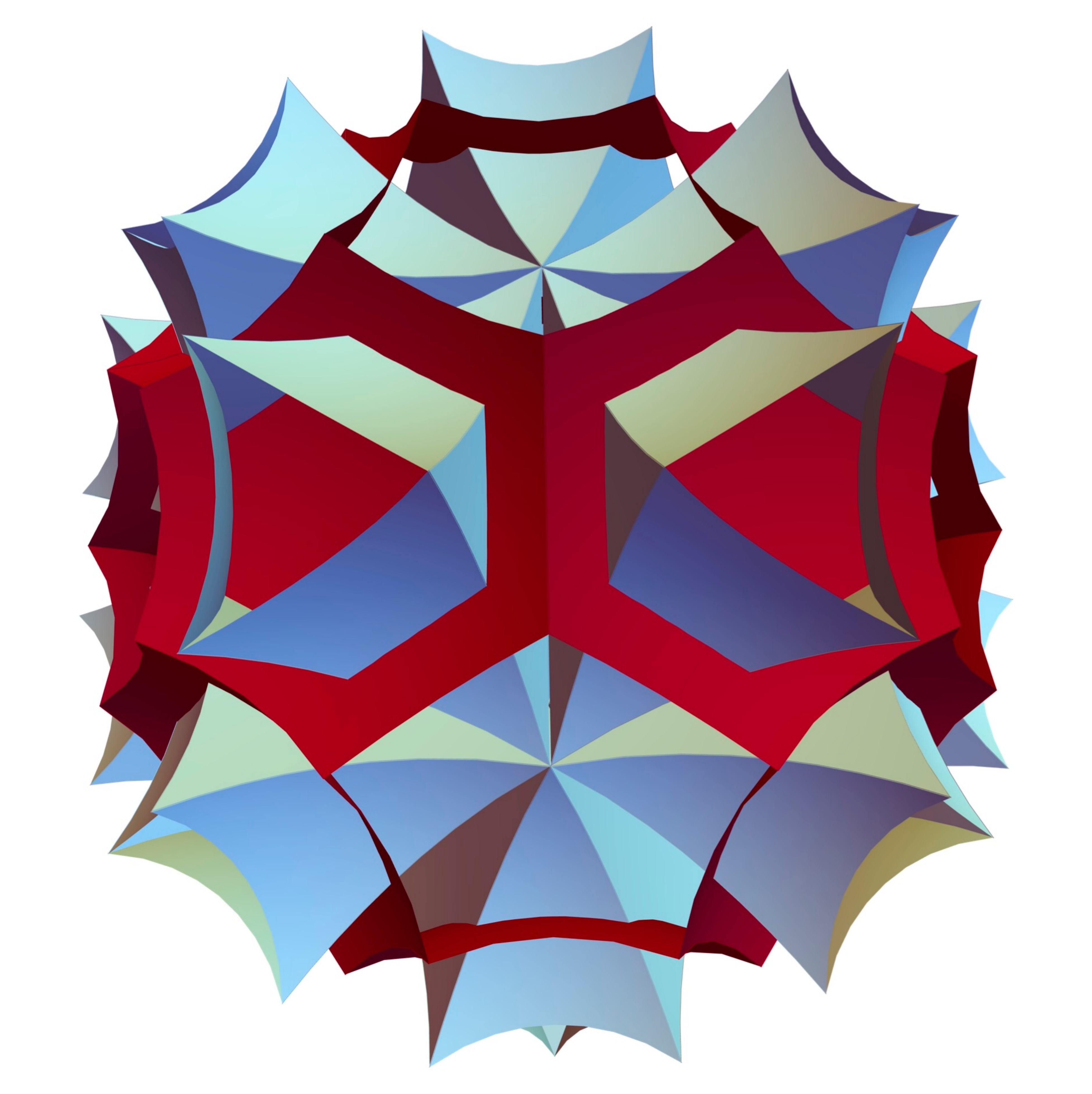}
    \quad
   \includegraphics[width=1.8in]{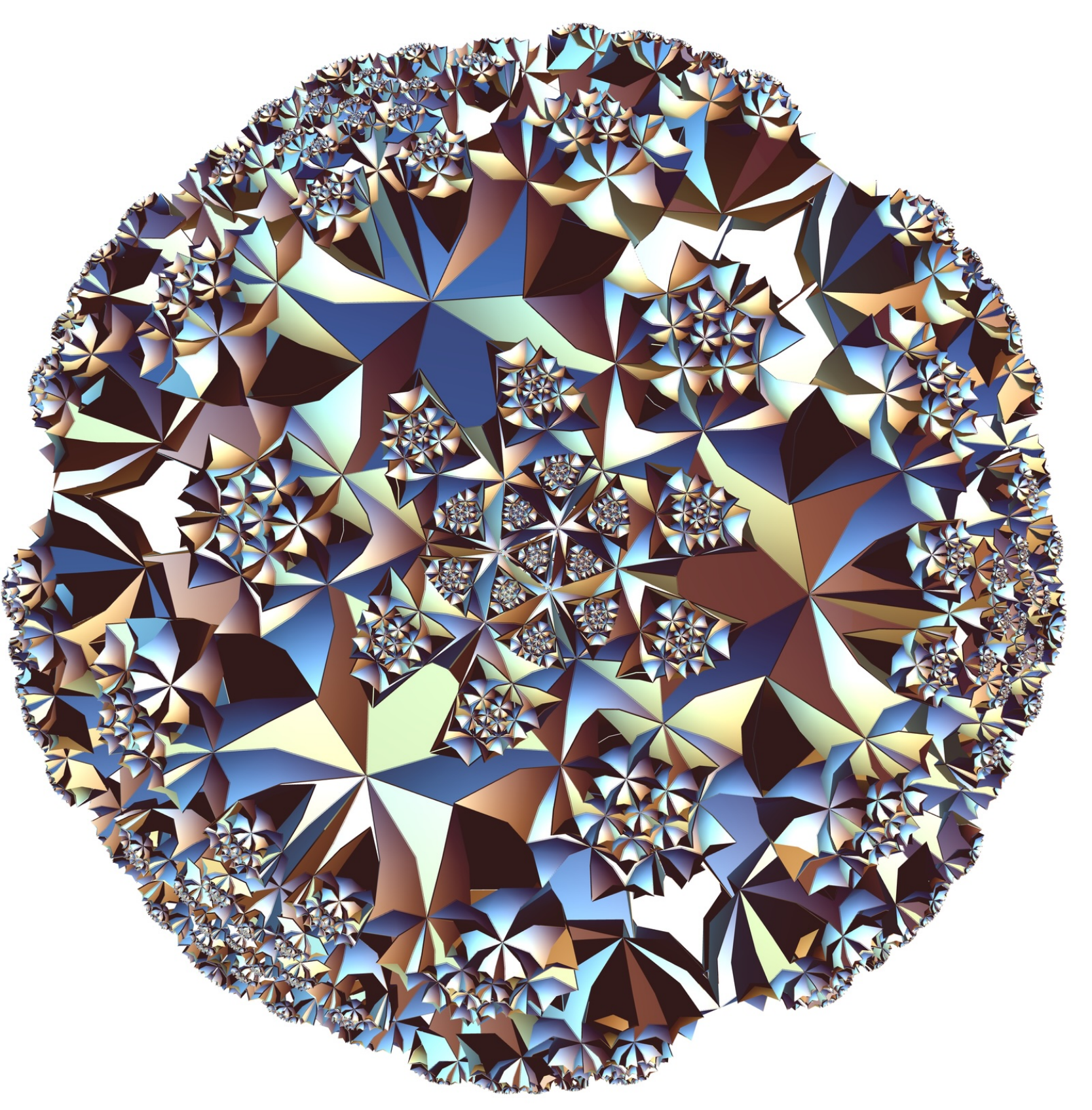}
      \caption{Tiling with $KP_{3,5}^3$ and an antiprismatic icosahedron $AP_{3,5}^3$}
   \label{fig:AntiprismaticIcosahedra}
\end{figure}

\bibliography{minlit}
\bibliographystyle{alpha}

\end{document}